\pdfoutput=1
\RequirePackage{ifpdf}
\ifpdf 
\documentclass[pdftex]{sigma}
\else
\documentclass{sigma}
\fi

\numberwithin{equation}{section}

\newtheorem{Theorem}{Theorem}[section]
\newtheorem{Corollary}[Theorem]{Corollary}
\newtheorem{Lemma}[Theorem]{Lemma}
\newtheorem{Proposition}[Theorem]{Proposition}
{ \theoremstyle{definition}
\newtheorem{Definition}[Theorem]{Definition}

\newtheorem{Remark}[Theorem]{Remark}

\newtheorem{fact}[Theorem]{Fact}
\newtheorem{notation}[Theorem]{Notation}}

\usepackage{mathrsfs}
\usepackage{braket}

\usepackage{mathtools}
\usepackage{extarrows}
\usepackage{wasysym}

\usepackage{epic, eepic}

\newcommand{\vl}{\boldsymbol{\lambda}}
\newcommand{\vm}{\boldsymbol{\mu}}

\newcommand{\vx}{\boldsymbol{x}}

\newcommand{\vs}{\boldsymbol{s}}
\newcommand{\lo}{\lambda^{(1)}}
\newcommand{\lt}{\lambda^{(2)}}
\newcommand{\lN}{\lambda^{(N)}}
\newcommand{\mo}{\mu^{(1)}}
\newcommand{\mt}{\mu^{(2)}}
\newcommand{\mN}{\mu^{(N)}}
\newcommand{\vu}{\boldsymbol{u}}
\newcommand{\vv}{\boldsymbol{v}}

\newcommand{\Xo}{X^{(1)}}
\newcommand{\Xt}{X^{(2)}}

\newcommand{\ketzero}{\ket{\boldsymbol{0}}}
\newcommand{\brazero}{\bra{\boldsymbol{0}}}

\newcommand{\cF}{\mathcal{F}}
\newcommand{\cU}{\mathcal{U}}

\newcommand{\cV}{\mathcal{V}}

\newcommand{\cA}{\mathcal{A}}
\newcommand{\cAs}{\mathcal{A}^*}

\newcommand{\wPhi}{\widehat{\Phi}}
\newcommand{\wS}{\widetilde{S}}
\newcommand{\wcT}{\widetilde{\mathcal{T}}}

\newcommand{\sfN}{\mathsf{N}}

\newcommand{\hg}{c}

\newcommand{\cTH}{\mathcal{T}^H}
\newcommand{\cTV}{\mathcal{T}^V}

\newcommand{\tr}{\mathrm{tr}}

\newcommand{\elG}[1]{\Gamma(#1 ; q,p)}
\newcommand{\Phicross}{\Phi^{\mathrm{cr}}}
\newcommand{\tilTV}{\widetilde{\mathcal{T}}^{V}}
\newcommand{\tilTH}{\widetilde{\mathcal{T}}^{H}}
\newcommand{\mkdloop}{\widetilde{\mathcal{T}}^{\mathrm{loop}}}
\newcommand{\wg}[2]{\Theta(#1;q,p)_{#2}}
\newcommand{\phicont}{\phi^{\mathrm{cont}}}
\newcommand{\Iellip}{I^{\mathrm{ellip}}}
\newcommand{\parset}{\mathsf{P}}
\newcommand{\parap}{\mathtt{P}}
\newcommand{\cellip}{c^{\mathrm{ellip}}}

\newcommand{\ordmac}{f^{\mathfrak{gl}_N}}
\newcommand{\nonstrui}{f^{\widehat{\mathfrak{gl}}_N}}
\newcommand{\fellip}{f^{\mathrm{ellip}}}

\newcommand{\cVweight}[2]{
	\cV\Big(
	\arraycolsep=1pt
	\renewcommand{\arraystretch}{0.8}
	\begin{array}{c}#1\end{array} ; #2
	\renewcommand{\arraystretch}{1.0}
	\Big)
}

\begin{document}
\allowdisplaybreaks

\newcommand{\arXivNumber}{2002.00243}

\renewcommand{\thefootnote}{}

\renewcommand{\PaperNumber}{116}

\FirstPageHeading

\ShortArticleName{Non-Stationary Ruijsenaars Functions for $\kappa=t^{-1/N}$}

\ArticleName{Non-Stationary Ruijsenaars Functions for $\boldsymbol{\kappa=t^{-1/N}}$\\ and Intertwining Operators of Ding--Iohara--Miki\\ Algebra\footnote{This paper is a~contribution to the Special Issue on Elliptic Integrable Systems, Special Functions and Quantum Field Theory. The full collection is available at \href{https://www.emis.de/journals/SIGMA/elliptic-integrable-systems.html}{https://www.emis.de/journals/SIGMA/elliptic-integrable-systems.html}}}

\Author{Masayuki FUKUDA~$^\dag$, Yusuke OHKUBO~$^\ddag$ and Jun'ichi SHIRAISHI~$^\ddag$}

\AuthorNameForHeading{M.~Fukuda, Y.~Ohkubo and J.~Shiraishi}

\Address{$^\dag$~Department of Physics, Faculty of Science, The University of Tokyo,\\
\hphantom{$^\dag$}~Hongo 7-3-1, Bunkyo-ku, Tokyo 113-0033 Japan}
\EmailD{\href{phy.m.fukuda@gmail.com}{phy.m.fukuda@gmail.com}}

\Address{$^\ddag$~Graduate School of Mathematical Sciences, The University of Tokyo,\\
\hphantom{$^\ddag$}~Komaba 3-8-1, Meguro-ku, Tokyo 153-8914 Japan}
\EmailD{\href{yusuke.ohkubo.math@gmail.com}{yusuke.ohkubo.math@gmail.com}, \href{shiraish@ms.u-tokyo.ac.jp}{shiraish@ms.u-tokyo.ac.jp}}

\ArticleDates{Received April 23, 2020, in final form November 01, 2020; Published online November 18, 2020}

\Abstract{We construct the non-stationary Ruijsenaars functions (affine analogue of the Macdonald functions) in the special case $\kappa=t^{-1/N}$, using the intertwining operators of the Ding--Iohara--Miki algebra (DIM algebra) associated with $N$-fold Fock tensor spaces. By the $S$-duality of the intertwiners, another expression is obtained for the non-stationary Ruijsenaars functions with $\kappa=t^{-1/N}$, which can be regarded as a natural elliptic lift of the asymptotic Macdonald functions to the multivariate elliptic hypergeometric series. We~also investigate some properties of the vertex operator of the DIM algebra appearing in the present algebraic framework; an integral operator which commutes with the elliptic Ruijsenaars operator, and the degeneration of the vertex operators to the Virasoro primary fields in the conformal limit $q \rightarrow 1$.}

\Keywords{Macdonald function; Rujisenaars function; Ding--Iohara--Miki algebra}

\Classification{33D52; 81R10}

\renewcommand{\thefootnote}{\arabic{footnote}}
\setcounter{footnote}{0}

\section{Introduction}

The non-stationary Ruijsenaars function $\nonstrui(\vx,p|\vs,\kappa|q,t)$
introduced by one of the authors in~\cite{Shiraishi2019affine}
is by definition given in the form of the Nekrasov partition function,
\begin{gather}
f^{\widehat{\mathfrak{gl}}_N}(\vx,p|\vs,\kappa|q,t)\nonumber
\\ \qquad
{}=
\sum_{\lambda^{(1)},\ldots,\lambda^{(N)}\in {\mathsf P}}
\prod_{i,j=1}^N
{\sfN^{(j-i|N)}_{\lambda^{(i)},\lambda^{(j)}} (ts_j/s_i|q,\kappa) \over \sfN^{(j-i|N)}_{\lambda^{(i)},\lambda^{(j)}} (s_j/s_i|q,\kappa)}
 \prod_{\beta=1}^N\prod_{\alpha\geq 1} ( p x_{\alpha+\beta}/tx_{\alpha+\beta-1})^{\lambda^{(\beta)}_\alpha}.\label{eq: nonstrui intro}
\end{gather}
Here $\vx=(x_1,\ldots, x_N)$, $\vs=(s_1,\ldots, s_N)$, and $N\in \mathbb{Z}_{\geq 1}$.
As for the detail, see Definition~\ref{def: non-st. Ruij}.
Our goal in the present paper is to establish the transformation formula stated in
Theorem~\ref{thm: fgln=fEG} below
for the non-stationary Ruijsenaars function $\nonstrui$ in the special case $\kappa=t^{-1/N}$,
by using the $S$-duality of the intertwining operators of the Ding--Iohara--Miki (DIM) algebra.

Define the elliptic shifted product $\wg{a}{n}$ as the ratio of the Ruijsenaars elliptic
gamma function $\Gamma(a;q,p)$ by
\begin{gather}\label{eq: def wg}
\wg{a}{n} :=\frac{\Gamma(q^n a;q,p)}{\Gamma(a;q,p)}, \qquad
\Gamma(a; q, p) := \frac{(qp/a;q,p)_\infty}{(a;q,p)_\infty}.
\end{gather}

\begin{definition*}[Definition~\ref{def: ellip Mac}]
Define $\fellip_N(\vx;\vs|q,t,p) \in \mathbb{Q}(q,t,\vs)[[p,x_2/x_1,\ldots,x_N/x_{N-1}]]$ by
\begin{gather*}
\fellip_N(\vx;\vs|q,t,p)
=\sum_{\theta \in \mathsf{M}_N} \cellip_N(\theta;\boldsymbol{s}|q,q/t,p)
\prod_{1\leq i<j\leq N} (x_j/x_i)^{\theta_{ij}},
\end{gather*}
where $\mathsf{M}_N=\{ (\theta_{ij})_{1\leq i, j\leq N}\,|\,
\theta_{ij} \in \mathbb{Z}_{\geq 0}, \, \theta_{kl}=0 \mbox{ if } k \geq l\}$ is the set of
$N \times N$ strictly upper triangular matrices
with nonnegative integer entries,
and
\begin{gather*}
\cellip_N(\theta;\boldsymbol{s}|q,t,p)=
\prod_{k=2}^{N}\prod_{1\le i<j\le k}
\dfrac{\wg{q^{\sum_{a>k}(\theta_{ia}-\theta_{ja})}ts_j/s_i}{\theta_{ik}}}
{\wg{q^{\sum_{a>k}(\theta_{ia}-\theta_{ja})}qs_j/s_i}{\theta_{ik}}}
\\ \hphantom{\cellip_N(\theta;\boldsymbol{s}|q,t,p)=}
{}\times\prod_{k=2}^N \prod_{1\le i\le j<k}
\dfrac{\wg{q^{-\theta_{jk}+\sum_{a>k}(\theta_{ia}-\theta_{ja})}qs_j/ts_i}{\theta_{ik}}}
{\wg{q^{-\theta_{jk}+\sum_{a>k}(\theta_{ia}-\theta_{ja})}s_j/s_i}{\theta_{ik}}}.
\end{gather*}
\end{definition*}

\begin{theorem*}[Theorem~\ref{thm: fgln=fEG}]
As a formal series in $p$, $s_{i+1}/s_i$, $x_{i+1}/x_i$ $(i=1,\ldots, N-1)$
and $px_1/x_N$, we have
\begin{gather}\label{eq: fgln=fEG intro}
f^{\widehat{\mathfrak{gl}}_N}\big(\vx {}',p^{{1}/{N}}|\vs',t^{-{1}/{N}}|q,t\big)
=\mathfrak{C}\times \fellip_N(\boldsymbol{s};\boldsymbol{x}|q,t,p),
\end{gather}
where
\begin{gather*}
\mathfrak{C}:=
\left(\frac{(pq/t;q,p)_{\infty}}{(p;p)_{\infty}(pt;q,p)_{\infty}} \right)^{N}
\prod_{1\leq i<j\leq N} \frac{\Gamma(tx_j/x_i;q,p)}{\Gamma(qx_j/x_i;q,p)}
\prod_{1\leq i<j\leq N} \frac{(ts_j/s_i;q)_{\infty}}{(qs_j/s_i;q)_{\infty}},\\
\vs'=(s'_1,\ldots ,s'_N),\qquad s'_k =t^{{k}/{N}}s_k,\qquad
\vx'=(x'_1,\ldots ,x'_N),\qquad x'_k= p^{-{k}/{N}}x_k.
\end{gather*}
\end{theorem*}

The shifts of $\vs'$ and $\vx'$ correspond to the ones used in the limit $p\rightarrow 0$
 (Fact~\ref{fact: p->0 lim}).
These shifts appear naturally
in the construction by $p$-trace of vertex operators that we will explain later.\footnote{As is in this theorem,
we use $p$ instead of $p^N$
in the main text as in Definition~\ref{def: ellip Mac}
(since it simplifies our description).
Similarly,
though it is meant that we study the non-stationary Ruijsenaars functions for
$\kappa=t^{-1/N}$ (as in the title of this paper) of~(\ref{eq: nonstrui intro}),
the parameter $\kappa$ is occasionally used instead of~$\kappa^N$.}

Contrary to the case $N\geq 2$,
the case $N=1$ seems somewhat special from the point of view of our vertex operator approach,
and the parameter $\kappa$ can be treated as an arbitrary constant.
As~a~result, we have the following summation formula, which has been already proved by \cite{CNO2013Five, RW2018nekrasov}.
\begin{theorem*}[Theorem~\ref{thm: N=1}]
We have
\begin{gather}
 \exp\left(\sum \frac{1}{n}\frac{(1-q^n \kappa^n)(1-\kappa^n/t^n)\kappa^{-n}p^n}{(1-q^n)(1-t^{-n})(1-p^n)} \right) \nonumber\\
\qquad {}= \sum_{\lambda \in \parset} (p/\kappa)^{|\lambda|} \frac{\prod_{1 \leq i\leq j }
(\kappa q^{-\lambda_i+\lambda_{j+1}} t^{i-j};q)_{\lambda_j-\lambda_{j+1}}
(\kappa q^{\lambda_i-\lambda_j} t^{-i+j+1};q)_{\lambda_{j}-\lambda_{j+1}}}{\prod_{1 \leq i\leq j }
(q^{-\lambda_i+\lambda_{j+1}} t^{i-j};q)_{\lambda_j-\lambda_{j+1}}
(q^{\lambda_i-\lambda_j} t^{-i+j+1};q)_{\lambda_{j}-\lambda_{j+1}}} . \label{eq: N=1}
\end{gather}
\end{theorem*}

Remark that setting $\kappa = t^{-1}$ in (\ref{eq: N=1}),
we recover~(\ref{eq: fgln=fEG intro}) for $N=1$.
To prove Theorem~\ref{thm: fgln=fEG} and Theorem~\ref{thm: N=1},
we use the technique of the topological vertex operator.
This consists of the DIM algebra, the trivalent vertex operators,
and the web diagrams encoding the structure of the Fock tensor spaces
which the DIM algebra is acting on.

To fix a good starting point,
we need to recall some facts about the asymptotically free eigenfunctions $\ordmac(\vx,\vs|q,t)$
for the Macdonald $q$-difference operator. See \cite{BFS2014Macdonald, NS2012direct, Shiraishi2005conjecture} and Appendix~\ref{sec: asymp mac} as to the basic facts.

\begin{definition*}[Definition~\ref{def: ordinary Mac}]
Define the formal series $\ordmac(\vx;\vs|q,t) \in
\mathbb{Q}(q,t,\vs)[[x_2/x_1,\ldots, \allowbreak x_N/x_{N-1}]]$ by\footnote{$\ordmac(\vx;\vs|q,t)$
coincides with $p_N(\vx;\vs|q,t)$ in \cite{FOS2019Generalized}.}
\begin{gather*}
\ordmac(\boldsymbol{x};\boldsymbol{s}|q,t) =
\sum_{\theta \in \mathsf{M}_N} c_N(\theta;\boldsymbol{s}|q,t)
\prod_{1\leq i<j\leq N}(x_j/x_i)^{\theta_{ij}},
\end{gather*}
where the coefficient $c_N(\theta;\boldsymbol{s}|q,t)$ is defined by
\begin{gather*}
c_N(\theta;\boldsymbol{s}|q,t)=
\prod_{k=2}^{N}
\prod_{1\le i<j\le k}
\dfrac{(q^{\sum_{a>k}(\theta_{ia}-\theta_{ja})}ts_j/s_i;q)_{\theta_{ik}}}
{(q^{\sum_{a>k}(\theta_{ia}-\theta_{ja})}qs_j/s_i;q)_{\theta_{ik}}}
\\ \hphantom{c_N(\theta;\boldsymbol{s}|q,t)=}
{}\times\prod_{k=2}^N\prod_{1\le i\le j<k}
\dfrac{(q^{-\theta_{jk}+\sum_{a>k}(\theta_{ia}-\theta_{ja})}qs_j/ts_i;q)_{\theta_{ik}}}
{(q^{-\theta_{jk}+\sum_{a>k}(\theta_{ia}-\theta_{ja})}s_j/s_i;q)_{\theta_{ik}}}.
\end{gather*}
\end{definition*}

The Macdonald $q$-difference operator is
derived from the trigonometric case of the Ruijsenaars models~\cite{R1987complete}
which is given as a relativistic generalization of the
Calogero--Morser--Sutherland systems~\cite{OP1983quantum}.
Explicit eigenfunctions of the Macdonald $q$-difference operator
(trigonometric Ruijsenaars operator)
have been studied, and they are called Macdonald symmetric polyno\-mials~\cite{Macdonald2015Symmetric}.
Whereas the Macdonald polynomials are parametrized by partitions,
the function~$\ordmac$ in Defi\-ni\-tion~\ref{def: ordinary Mac} is series with parameters~$s_i$.
Specialization of~$s_i$'s gives us ordinary Macdonald polynomials associated with partitions.
Note that the~$\ordmac$ also enjoys the eigenvalue equation (Fact~\ref{fact: eigen fn of D}), the analytic property (Fact~\ref{fact: analyticity}),
the bispectral duality
(Fact~\ref{fact: bispec dual}), and the Poincar\'e duality
(Fact~\ref{fact: Poincare dual}).
Ruijsenaars introduced not only the trigonometric integral model but also the general elliptic version \cite{R1987complete}.
In the elliptic case,
however,
it seems we still lack a fundamental understanding of the properties
of eigenfunctions.
See \cite{FV2004hypergeometric,Ruijsenaars2009hilbert,Ruijsenaars2009hilbert2}, for example.
One of our motivations for studying the non-stationary Ruijsenaars functions
(\ref{eq: nonstrui intro})
comes from the strong hope that the non-stationary function might have much simpler combinatorial or analytic properties
than the original elliptic stationary Ruijsenaars functions.
Some of the observations given in \cite{Shiraishi2019affine} read:
(1)~it is conjectured that the non-stationary Ruijsenaars functions $\nonstrui$
with a suitable normalization procedure
give explicit solutions to the elliptic Ruijsenaars models
in the very (essentially) singular limit $\kappa \rightarrow 1$.
(2)~$\nonstrui$ reduce to the asymptotically free Macdonald eigenfunctions $\ordmac$
in the limit $p\rightarrow 0$.
(3)~We have the bispectral duality and the Poincar\'e duality for~$\nonstrui$.

We have a ``web diagrammatic description'' of the combinatorial structure
of $\ordmac$
(Definition~\ref{def: ordinary Mac}),
based on the DIM algebra.
In other words, the $\ordmac$ has an interpretation as a certain Nekrasov partition function associated with the web diagram.
See \cite{FOS2019Generalized,NPZ2017TUS,Zenkevich2018Higgs}.

The DIM algebra has two kinds of
intertwining operators, $\Phi\colon \cF^{(0,1)}\otimes \cF^{(1,0)} \rightarrow \cF^{(1,1)}$ and
$\Phi^*\colon \cF^{(1,1)} \rightarrow \cF^{(1,0)}\otimes \cF^{(0,1)}$
among the triple of Fock spaces, introduced in \cite{AFS2012quantum}.
As for the definition of the modules $\cF^{(1,M)}$, $\cF^{(0,1)}$ and the intertwiners,
see Facts~\ref{fact: lv 0,1 mod},~\ref{fact: lv 1,M mod}, and \ref{fact: intertwiner}.
By~using physics terminology,
we refer to the modules~$\cF^{(0,1)}$ as the preferred directions.
The matrix elements of these intertwiners are identical to the refined topological vertex.
We express the composition $\Phi^* \circ \Phi$
by the cross diagram in Fig.~\ref{fig: cross intro}, left.
\begin{figure}[t]\centering
 \includegraphics[width=7.5cm]{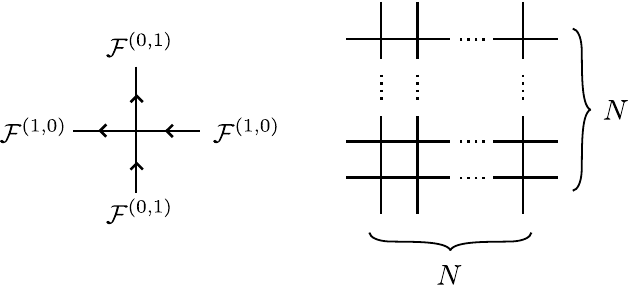}
	\caption{$\Phi^* \circ \Phi$ and $\ordmac$. }
	\label{fig: cross intro}
\end{figure}
Compose these operators reticulately
as in Fig.~\ref{fig: cross intro}, right.
Specialize spectral parameters in a certain manner.
Attach empty diagrams to all the external edges.
Then we have the Macdonald function~$\ordmac$
as thus constructed matrix element
(see Fact~\ref{fact: macdonald from mukade}).

\begin{figure}[t]\centering
 \includegraphics[width=3.2cm]{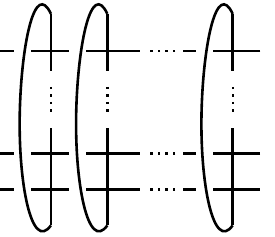}
	\caption{The cylindric web diagram for the non-stationary Ruijsenaars functions $\nonstrui$.}	\label{fig: loop}
\end{figure}

\begin{figure}[th!]
\centering
 \includegraphics[width=5.5cm]{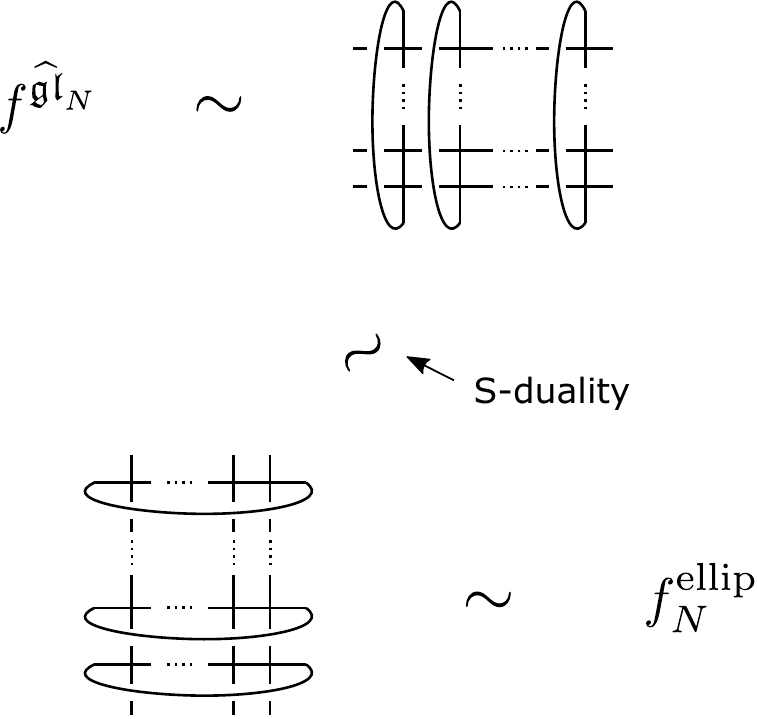}
	\caption{The sketch of the proof.}	\label{fig: proofsketch}
\end{figure}

Suppose we consider the trace in the preferred or vertical direction,
instead of the matrix element as above.
See Fig.~\ref{fig: loop} below.
We will prove that the $p$-trace associated with the web on a cylinder
gives us the non-stationary Ruijsenaars function $\nonstrui$
for the special parameter $\kappa=t^{-1/N}$ for $N\geq 2$,
and for generic $\kappa$ for $N=1$.

Thanks to the $S$-duality,
we can flip the diagram, obtaining the picture as in Fig.~\ref{fig: proofsketch}.
Finally, the trace in the horizontal direction can be calculated by the
standard technique,
thereby giving the elliptic lift $\fellip_{N}$ of the asymptotically
free eigenfunction $\ordmac$ for the Macdonald $q$-difference operator.
In Section~\ref{sec: nonst.R and Mukade} are given
our proofs of Theorems~\ref{thm: fgln=fEG} and~\ref{thm: N=1} which go along this idea.

Some explanations
are in order, concerning the physical background and recent related works
on the non-stationary systems,
including the non-stationary Heun, Lam\'{e}, and elliptic Calogero--Sutherland equations.
These non-stationary equations have been extensively studied
based on the perturbative approach
by Atai and Langmann in \cite{AL2018series}.
Recently in \cite{AL2019exact},
they obtained an integral formula for the eigenfunctions for the non-stationary elliptic
Calogero--Sutherland equation
for some special choices of the parameter in the ``time derivative'' term,
using the kernel function technique.

Note that
no explicit equations have been obtained for the non-stationary Ruijsenaars function $\nonstrui$
unfortunately.\footnote{In \cite{LNS2020Basic}
is given a conjecture about an eigenvalue equation for the non-stationary Ruijsenaars functions $\nonstrui$
using a certain operator $\mathcal{T}^{\hat{\mathfrak{gl}}_N}$ which contains
$q^{\frac{1}{2}\Delta}$ (where $\Delta$ is the ordinary Laplacian),
is derived another version of the coefficients of $\nonstrui$
without using the Nekrasov's factor $\sfN^{(i|N)}_{\lambda,\mu}(u|q,\kappa)$,
and is proved that $\nonstrui$ absolutely converges in a certain domain.}
So the authors have been lead to use the representation theories of the DIM algebra,
to bypass the troublesome situation without any eigenvalue equations or associated kernel functions.
One may find, nevertheless,
the resulting formula in Theorem~\ref{thm: fgln=fEG} can be regarded
as a $q$-difference analogue of Atai and Langmann's integral formula in~\cite{AL2019exact},
suggesting the existence of $q$-analogue of the kernel functions.

The function $\nonstrui$ is by definition a series
which can be regarded as a Jackson integral,
whereas the eigenfunction in \cite{AL2019exact} is given in terms of a contour integral.
The functions $\nonstrui$ is parametrized by the continuous parameters $s_i$,
while Atai and Langmann's integral formula contains
a partition as a set of discrete parameters.
The integral formula in \cite{AL2019exact} works not only for $\kappa'=g$ but also for
$\kappa'=kg$ ($k=2,3,\ldots$),
where $\kappa'$ and $g$ are parameters corresponding to~$\kappa$ and~$t$
as $\kappa={\rm e}^{\kappa' \hbar}$, $t={\rm e}^{g \hbar}$.
It is an interesting problem to find a similar construction of the non-stationary Ruijsenaars functions for $\kappa =t^{-k/N}$
($k=2,3,\ldots$).

In this occasion, we address some other related problems in the representation
theory of the DIM algebra.
First,
we derive an integral operator
$I(s_1/s_0,\ldots,s_N/s_0)$ introduced in
\cite{shiraishi2005commutative1,shiraishi2005commutative2,Shiraishi2006family}
from the intertwining operators of the DIM algebra.
As for the definition of $I(s_1/s_0,\ldots,s_N/s_0)$, see
Definition~\ref{def: operator I}.
In \cite{shiraishi2005commutative1,shiraishi2005commutative2,Shiraishi2006family},
it was conjectured that the integral operator
$I(s_1/s_0,\ldots,s_N/s_0)$ commutes with
Macdonald's difference operator.
A proof has been given in \cite{NS2012direct}.
We provide yet another perspective based on the vertex operator.
We also discuss the elliptic analogue of
the integral operator.
Our elliptic integral operator can be constructed by
taking the loop of intertwining operators.
It commutes with the Ruijsenaars operator.

Secondly, we study the conformal limit $q \rightarrow 1$ of the $2N$-valent intertwining operator $\cV(z)$,
which is a main object in the authors' previous paper \cite{FOS2019Generalized}, restricting ourselves to the simplest nontrivial case $N=2$.
The operator $\cV(z)$ is defined by the relations with the $q$-$W_N$ algebra, the $q$-Vir algebra for $N=2$.
We will derive the well-known relation of the Virasoro primary fields
from those defining relations in the limit $q\rightarrow 1$.
In \cite{FOS2019Generalized}, a formula for the matrix elements of $\cV(z)$ with respect to
generalized Macdonald functions (Fact~\ref{fact:existence thm of Gn Mac})
are proved.
The result Theorem~\ref{thm: the Vir rel} and the matrix elements formula for $\cV(z)$
prove that the generalized Macdonald functions are reduced to
Alba, Fateev Litvinov and Tarnopolskii's basis~\cite{AFLT2011combinatorial}.
Hence, the matrix element formula for~$\cV(z)$
provides us with another proof of the 4-dimensional AGT correspondence~\cite{AGT2010liouville}.

This paper is organized as follows.
Basic facts with respect to the intertwining operators of the DIM algebra
are summarized in Section~\ref{sec: DIM intertwiner Mukade}.
In Section~\ref{sec: nonst.R and Mukade},
we reproduce the non-stationary Ruijsenaars functions from the intertwiners on the cylindric web diagrams.
We also prove Theorem~\ref{thm: fgln=fEG} by using the $S$-duality.
In Section~\ref{sec: int op.},
we derive the integral operator $I(s_1/s_0,\ldots,s_N/s_0)$
and prove the commutativity with the Macdonald $q$-difference operator.
The elliptic extension of the integral operator is also discussed.
We take the conformal limit $q\rightarrow 1$
in Section~\ref{sec: q->1}.
In~Appen\-dix~\ref{sec: asymp mac},
some facts with respect to Macdonald functions $\ordmac$ are explained.
In~Appen\-dix~\ref{sec: nonst. Ruijsenaars and affine sc.},
we describe the construction of the non-stationary Ruijsenaars functions
from the affine screening operators in the case of general $\kappa$.
In~Appen\-dix~\ref{sec: proofs}
are given some straightforward but cumbersome details needed for our proofs.

\subsection*{Notation}
We use the following standard symbols for
the $q$-shifted factorials, the theta functions and the elliptic gamma functions \cite{GR2004basic}:
\begin{gather*}
(a;q)_{\infty}:=\prod_{n=1}^{\infty}\big(1-q^{n-1}a\big),\qquad
(a;q)_{m}:= \frac{(a;q)_\infty}{(a q^m ;q)_\infty} ,
\\
\theta_q(a):=(a;q)_{\infty}\big(qa^{-1};q\big)_{\infty} ,
\\
(a; q, p)_\infty := \prod_{n,m=1}^{\infty}\big(1-q^{n-1}p^{m-1}a\big) ,
\qquad
\Gamma(a; q, p) := \frac{(qp/a;q,p)_\infty}{(a;q,p)_\infty} .
\end{gather*}
For integers $n\leq m$, we use the following symbols for the tensor product and the ordered products
\begin{gather*}
\bigotimes_{n\leq i \leq m}^{\curvearrowright} A_i:= A_{n} \otimes A_{n+1} \otimes \cdots \otimes A_{m},
\\
\prod_{n\leq i \leq m}^{\curvearrowright} A_i:= A_{n} \times A_{n+1} \times \cdots \times A_{m},
\\
\prod_{n\leq i \leq m}^{\curvearrowleft} A_i:= A_{m} \times A_{m-1} \times \cdots \times A_{n}.
\end{gather*}

A partition $\lambda=(\lambda_1, \lambda_2,\ldots )$
is a sequence of nonnegative integers such that $\lambda_1 \geq \lambda_2 \geq \cdots$ with finitely many nonzero elements $\lambda_j$.
The empty diagram is denoted by $\emptyset = (0,0,\ldots)$.
$\parset$~denotes the set of all partitions.
For $\lambda \in \parset$,
we write $|\lambda|=\sum_{i\geq 1}\lambda_i$
and $\ell(\lambda)=\#\{i\,|\,\lambda_i \neq 0 \}$.
The conjugate partition of $\lambda$ is
denoted by $\lambda'$.
If $\lambda_i \leq \mu_i$ for all $i$,
we write $\lambda \subset \mu$.
For an $N$-tuple of partitions $\vl=\big(\lo, \lt, \ldots, \lN\big) \in \parset^N$,
write $|\vl|=\big|\lo\big|+\cdots +\big|\lN\big|$.
For a pair of positive integers $(i,j)\in \mathbb{Z}^2_{>0}$,
the arm length $a_{\lambda}(i,j)$ and the leg length $\ell_{\lambda}(i,j)$
are defined by
\begin{gather*}
a_{\lambda}(i,j)= \lambda_i-j , \qquad
\ell_{\lambda}(i,j)=\lambda'_j-i .
\end{gather*}

\section{DIM algebra, intertwiners, and Mukad\'e operators}
\label{sec: DIM intertwiner Mukade}

We briefly recall the definition of the Ding--Iohara--Miki (DIM) algebra
\cite{DI1997Generalization,Miki:2007}, its intertwining operators and the $S$-duality formula for the
intertwining operators.
Let $q$ and $t$ be generic complex parameters
with $|q|$, $|t^{-1}|<1$.

\begin{Definition}
The DIM algebra,
which we denote by $\cU=\cU_{q,t}$,
is a unital associative algebra generated
by the currents $x^{\pm}(z)=\sum_{n\in \mathbb{Z}}x^{\pm}_n z^{-n}$,
$\psi^{\pm}(z)=\sum_{\pm n \in \mathbb{Z}_{\geq 0}}\psi^{\pm}_{n}z^{-n}$
and the central elements $c^{\pm 1/2}$.
The defining relations are
\begin{gather*}
\psi^+(z)x^\pm(w)=g\big(c^{\mp 1/2}w/z\big)^{\mp1} x^\pm(w)\psi^+(z),
\\
\psi^-(z)x^\pm(w)=g\big(c^{\mp 1/2}z/w\big)^{\pm1} x^\pm(w)\psi^-(z),
\\
\psi^\pm(z) \psi^\pm(w)= \psi^\pm(w) \psi^\pm(z),\qquad
\psi^+(z)\psi^-(w)=\dfrac{g\big(c^{+1} w/z\big)}{g\big(c^{-1}w/z\big)}\psi^-(w)\psi^+(z),
\\
\big[x^+(z),x^-(w)\big]=\dfrac{(1-q)(1-1/t)}{1-q/t}\big( \delta\big(c^{-1}z/w\big)\psi^+\big(c^{1/2}w\big)-
\delta(c z/w)\psi^-\big(c^{-1/2}w\big) \big),
\\
G^{\mp}(z/w)x^\pm(z)x^\pm(w)=G^{\pm}(z/w)x^\pm(w)x^\pm(z) ,
\end{gather*}
where
\begin{gather*}
g(z)=\dfrac{G^+(z)}{G^-(z)} ,\qquad
G^\pm(z)=\big(1-q^{\pm1}z\big)\big(1-t^{\mp 1}z\big)\big(1-q^{\mp1}t^{\pm 1}z\big) ,\qquad
\delta(z) = \sum_{n\in\mathbb{Z}}z^n .
\end{gather*}
\end{Definition}

\begin{fact}[\cite{DI1997Generalization}]
The Drinfeld coproduct
\begin{gather*}
\Delta\big(\hg^{\pm 1/2}\big)=\hg^{\pm 1/2} \otimes \hg^{\pm 1/2} ,
\\
\Delta \big(x^+(z)\big)=x^+(z)\otimes 1+\psi^-\big(c_{(1)}^{1/2}z\big)\otimes x^+(c_{(1)}z) ,
\\
\Delta \big(x^-(z)\big)=x^-(c_{(2)}z)\otimes \psi^+\big(c_{(2)}^{1/2}z\big)+1 \otimes x^-(z) ,
\\
\Delta \big(\psi^\pm(z)\big)=\psi^\pm \big(c_{(2)}^{\pm 1/2}z\big)\otimes \psi^\pm \big(c_{(1)}^{\mp 1/2}z\big)
\end{gather*}
gives rise to a bialgebra structure.
Further, $\cU$ has a Hopf algebra structure.
We omit the counit and the antipode.
\end{fact}

A $\cU$-module is called of level-$(n,m)$ if
the central elements
act as $c=(t/q)^{n/2}$ and $(\psi^+_0/\psi^-_0)^{1/2}\allowbreak =(q/t)^{m/2}$.
In this paper, we use two kinds of $\cU$-modules.
The first one is a free field representation with the following boson.
Let $\mathcal{H}$
be the Heisenberg algebra
generated by $\{a_n \,|\, n\in \mathbb{Z}\}$ with the commutation relation
\begin{gather*}
[a_n,a_m]=n\frac{1-q^{|n|}}{1-t^{|n|}}\delta_{n+m,0}.
\end{gather*}
Let $\ket{0}$ and $\bra{0}$ be
the highest weight vectors defined
by $a_n \ket{0}=0$ ($n\geq 0$) and $\bra{0}a_n=0$ ($n\leq 0$), respectively.
Denote by $\cF$ (resp.~$\cF^*$) the Fock space
generated from the highest weight vector $\ket{0}$ (resp.~$\bra{0}$).
The bilinear form $\cF^* \otimes \cF \rightarrow \mathbb{C}$ is defined by setting $\braket{0|0}=1$.

\begin{Definition}
Define the vertex operators $\eta(z)$, $\xi(z)$ and
$ \varphi^{\pm}(z) \in \mathrm{End}(\cF) \big[\big[z^{\pm 1}\big]\big]$ by
\begin{gather*}
\eta(z)=\exp\bigg( \sum_{n=1}^{\infty} \dfrac{1-t^{-n}}{n}a_{-n} z^{n} \bigg)
\exp\bigg({-}\sum_{n=1}^{\infty} \dfrac{1-t^{n} }{n}a_n z^{-n}\bigg), 
\\
\xi(z)=\exp\bigg({-}\sum_{n=1}^{\infty} \dfrac{1-t^{-n}}{n}q^{-n/2}t^{n/2}a_{-n} z^{n}\bigg)
\exp\bigg( \sum_{n=1}^{\infty} \dfrac{1-t^{n}}{n} q^{-n/2}t^{n/2} a_n z^{-n}\bigg),
\\
\varphi^{+}(z)=\exp\bigg({-}\sum_{n=1}^{\infty} \dfrac{1-t^{n}}{n} \big(1-t^n q^{-n}\big)
 q^{n/4}t^{-n/4} a_n z^{-n} \bigg),
 \\
\varphi^{-}(z)=\exp\bigg(\sum_{n=1}^{\infty}
\dfrac{1-t^{-n}}{n} \big(1-t^n q^{-n}\big) q^{n/4}t^{-n/4} a_{-n}z^{n}\bigg).
\end{gather*}
\end{Definition}

\begin{fact}[\cite{FHHSY2009commutative}]
\label{fact: lv 1,M mod}
Let $u$ be an indeterminate and $M$ be an integer.
The algebra homomorphism $\rho_u\colon \cU \rightarrow \mathrm{End}(\cF)$
defined by
\begin{gather*}
\hg^{1/2}\mapsto (t/q)^{1/4},\qquad
x^+(z)\mapsto u z^{-M} q^{-M/2}t^{M/2}\eta(z),\qquad
x^-(z)\mapsto u^{-1} z^{M}q^{M/2}t^{-M/2} \xi(z) ,
\\
\psi^+(z)\mapsto q^{M/2}t^{-M/2}\varphi^+(z),\qquad
\psi^-(z)\mapsto q^{-M/2}t^{M/2}\varphi^-(z)
\end{gather*}
endows $\cF$ with the level $(1,M)$-module structure.
\end{fact}

We denote by $\cF^{(1,M)}_u$ the Fock space endowed with the level $(1,M)$-module structure.
The dual space $\cF^*$ can also be endowed with the right $\cU$-module structure
through $\rho_u$.
Then it is denoted by $\cF^{(1,M)*}_u$.
The $\rho_u$ is called the horizontal representation.

Next, we consider the level $(0,1)$-module.
Let $\cF^{(0,1)}$ be the vector space
spanned by the vectors $\ket{\lambda}$ with $\lambda \in \parset $.
Define $( \bra{\lambda} \,|\, \lambda \in \parset )$ to be the dual basis
such that $\braket{\lambda|\mu}=\delta_{\lambda, \mu}$.

\begin{fact}[\cite{FFJMM2011semiinfinite,FT2011Equivariant}]
\label{fact: lv 0,1 mod}
Let $u$ be an indeterminate.
The following action gives the level $(0,1)$-module structure to $\cF^{(0,1)}$:
\begin{gather*}
c^{1/2} \ket{\lambda}=\ket{\lambda},
\\
x^+(z) \ket{\lambda}=\sum_{i=1}^{\ell(\lambda)+1}
A^+_{\lambda,i}\delta\big(q^{\lambda_i}t^{-i+1}u/z\big)\ket{\lambda+{\bf 1}_i},
\\
x^-(z) \ket{\lambda}=q^{1/2}t^{-1/2}\sum_{i=1}^{\ell(\lambda)}
A^-_{\lambda,i}\delta\big(q^{\lambda_i-1}t^{-i+1}u/z\big)\ket{\lambda-{\bf 1}_i},
\\
\psi^+(z)\ket{\lambda}=q^{1/2}t^{-1/2}B^+_\lambda(u/z)\ket{\lambda},
\\
\psi^-(z)\ket{\lambda}=q^{-1/2}t^{1/2} B^-_\lambda(z/u)\ket{\lambda}.
\end{gather*}
Here,
$A^{\pm}_{\lambda,i}\in \mathbb{Q}(q,t)$
and $B^+_\lambda(z) \in \mathbb{Q}(q,t)[[z]]$
are defined by
\begin{gather*}
A^+_{\lambda,i}=(1-t)\prod_{j=1}^{i-1}{\big(1-q^{\lambda_i-\lambda_j}t^{-i+j+1}\big)
\big(1-q^{\lambda_i-\lambda_j+1}t^{-i+j-1}\big)\over
\big(1-q^{\lambda_i-\lambda_j}t^{-i+j}\big)\big(1-q^{\lambda_i-\lambda_j+1}t^{-i+j}\big)},
\\
A^-_{\lambda,i}=(1-t^{-1})
{1-q^{\lambda_{i+1}-\lambda_i} \over 1-q^{\lambda_{i+1}-\lambda_i+1} t^{-1}}
\prod_{j=i+1}^{\infty}{\big(1-q^{\lambda_j-\lambda_i+1}t^{-j+i-1}\big)
\big(1-q^{\lambda_{j+1}-\lambda_i}t^{-j+i}\big)\over
\big(1-q^{\lambda_{j+1}-\lambda_i+1}t^{-j+i-1}\big)\big(1-q^{\lambda_j-\lambda_i}t^{-j+i}\big)} ,
\\
B^+_\lambda(z)={1-q^{\lambda_{1}-1} t z \over 1-q^{\lambda_{1}} z}
\prod_{i=1}^{\infty}{\big(1-q^{\lambda_i}t^{-i}z\big)
\big(1-q^{\lambda_{i+1}-1}t^{-i+1}z\big)\over
\big(1-q^{\lambda_{i+1}}t^{-i}z\big)\big(1-q^{\lambda_i-1}t^{-i+1}z\big)},
\\
B^-_\lambda(z)={1-q^{-\lambda_{1}+1} t^{-1} z \over 1-q^{-\lambda_{1}} z}
\prod_{i=1}^{\infty}{\big(1-q^{-\lambda_i}t^{i}z\big)
\big(1-q^{-\lambda_{i+1}+1}t^{i-1}z\big)\over
\big(1-q^{-\lambda_{i+1}}t^{i}z\big)\big(1-q^{-\lambda_i+1}t^{i-1}z\big)}.
\end{gather*}
\end{fact}

We denote this module by $\cF^{(0,1)}_u$.
This is called the vertical representation or the preferred direction.
By using the two representations,
the trivalent intertwiners $\Phi$, $\Phi^\ast$ of the DIM algebra were introduced in \cite{AFS2012quantum}.

\begin{fact}[\cite{AFS2012quantum}]
\label{fact: intertwiner}
Let $M$ be an integer.
If $w=-vu$,
there exists a unique linear operator
\begin{gather*}
\Phi\left[(1,M+1),w \atop (0,1),v; (1,M) ,u\right]\colon\ \cF^{(0,1)}_v\otimes \cF^{(1,M)}_u\longrightarrow
\cF^{(1,M+1)}_{w}
\end{gather*}
such that $\bra{0}\Phi(\ket{\emptyset}\otimes \ket{0})=1$ and
\begin{gather*}
a \Phi= \Phi \Delta(a) \qquad (\forall\, a\in \cU).
\end{gather*}
Similarly,
there exists a unique linear operator
\begin{gather*}
\Phi^*\left[ (1,M) ,u;(0,1),v \atop (1,M+1),-vu\right]\colon\ \cF^{(1,M+1)}_{-uv}\longrightarrow
\cF^{(1,M)}_u\otimes \cF^{(0,1)}_v
\end{gather*}
such that $(\bra{0}\otimes \bra{\emptyset}) \Phi^* \ket{0}=1$ and
\begin{gather*}
\Delta(a) \Phi^*= \Phi^* a \qquad (\forall\, a\in \cU).
\end{gather*}
\end{fact}

It is known that these intertwining operators can be realized as follows.
\begin{Definition}
For a partition $\lambda$,
define the $\lambda$-component $\Phi_\lambda$ of $\Phi$
\begin{gather*}
\Phi_\lambda\left[(1,M+1),-uv \atop (0,1),v; (1,M) ,u\right]
\colon\ \cF^{(1,M)}_{u} \rightarrow \cF^{(1,M+1)}_{-uv}
\end{gather*}
by
\begin{gather*}
\Phi_\lambda (\alpha)=\Phi(\ket{\lambda} \otimes \alpha )
\qquad \big(\forall\, \alpha \in \cF^{(1,M)}_u\big).
\end{gather*}
Similarly,
define the $\lambda$-component $\Phi^\ast_\lambda$ of $\Phi^\ast$
\begin{gather*}
\Phi^*_{\lambda}\left[ (1,M) ,u;(0,1),v \atop (1,M+1),-vu\right]
\colon\ \cF^{(1,M+1)}_{-uv} \rightarrow \cF^{(1,M)}_u
\end{gather*}
by
\begin{gather*}
\Phi^*( \alpha )
=\sum_{\lambda\in \parset} \Phi^*_\lambda (\alpha)\otimes \ket{\lambda}
\qquad \big(\forall\, \alpha \in \cF^{(1,M+1)}_{-uv}\big).
\end{gather*}
\end{Definition}

\begin{notation}
For $\lambda \in \parset$,
set
\begin{gather*}
n(\lambda)=\sum_{i\geq 1}(i-1)\lambda_i, \qquad
f_{\lambda}=(-1)^{|\lambda|}q^{n(\lambda')+|\lambda|/2}t^{-n(\lambda)-|\lambda|/2},
\\
c_{\lambda}:= \prod_{(i,j)\in \lambda}\big(1-q^{a_{\lambda}(i,j)}t^{\ell_{\lambda}(i,j)+1}\big), \qquad
c'_{\lambda}:= \prod_{(i,j)\in \lambda}\big(1-q^{a_{\lambda}(i,j)+1}t^{\ell_{\lambda}(i,j)}\big).
\end{gather*}
\end{notation}

\begin{fact}
[\cite{AFS2012quantum}]
$\Phi_{\lambda}$ is of the form
\begin{gather*}
\Phi_\lambda\left[ (1,M+1),-vu \atop (0,1),v; (1,M) ,u\right] = \hat{t}(\lambda,u,v,M) \widehat{\Phi}_\lambda(v) ,
\end{gather*}
where
\begin{gather*}
\hat t(\lambda,u,v,M)=(-vu)^{|\lambda|} (-v)^{-(M+1)|\lambda|}
f_\lambda^{-M-1}q^{n(\lambda')}/ c_\lambda ,
\\
\widehat{\Phi}_\lambda(v) = {:}\Phi_{\emptyset}(v) \eta_\lambda(v){:},
\\
\Phi_{\emptyset}(v) = \exp \bigg(
 {-}\sum_{n=1}^{\infty} \dfrac{1}{n}\dfrac{1}{1-q^n} a_{-n}v^n \bigg)
 \exp\bigg({-} \sum_{n=1}^{\infty} \dfrac{1}{n}\dfrac{q^n}{1-q^{n}} a_{n}v^{-n} \bigg),
 \\
\eta_\lambda(v)={:}\prod_{i=1}^{\ell(\lambda)}\prod_{j=1}^{\lambda_i}\eta(q^{j-1}t^{-i+1} v){:}.
\end{gather*}
The symbol ${:}\cdots{:}$ means the usual normal ordering product.
Similarly, $\Phi^*_{\lambda}$ is of the form
\begin{gather*}
\Phi^*_\lambda\left[ (1,M) ,v;(0,1),u \atop (1,M+1),-vu\right] = \hat{t}^*(\lambda,u,v,M) \wPhi^*_\lambda(u) ,
\end{gather*}
where\footnote{
Note that we modify the normalization of $\Phi^*_\lambda$
from the previous paper \cite{FOS2019Generalized}
by $c_{\lambda}$ and $c'_{\lambda}$. }
\begin{gather*}
\hat{t}^*(\lambda,u,v,M)=(q^{-1} v)^{-|\lambda|} (-u)^{M|\lambda|}
f_\lambda^{M}q^{n(\lambda')}/ c'_\lambda ,
\\
\wPhi^*_\lambda (u) ={:}\Phi^*_{\emptyset}(u) \xi_\lambda(u){:},
\\
\Phi^*_{\emptyset}(u) =\exp \bigg(\sum_{n=1}^{\infty} \dfrac{1}{n}\dfrac{1}{1-q^n} q^{-n/2}t^{n/2}a_{-n}u^n\bigg)
\exp\bigg(\sum_{n=1}^{\infty} \dfrac{1}{n}\dfrac{q^n}{1-q^{n}} q^{-n/2}t^{n/2}a_{n}u^{-n}\bigg),
\\
\xi_\lambda(u)={:}\prod_{i=1}^{\ell(\lambda)}\prod_{j=1}^{\lambda_i}\xi(q^{j-1}t^{-i+1} u){:}.
\end{gather*}
\end{fact}

\begin{figure}[h]\centering
\includegraphics{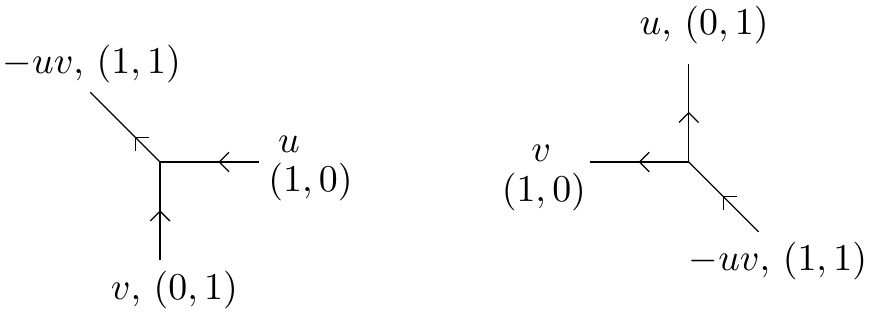}
\caption{The trivalent intertwining operators. Left: $\Phi$, Right: $\Phi^*$.
Note that the diagrams in this papar are flipped horizontally for the ones in \cite{FOS2019Generalized}.}
\label{fig: trivalent}
\end{figure}

Throughout the paper,
the case $M=0$ is concerned.
The intertwining operators $\Phi$ and $\Phi^*$
are expressed as the trivalent diagrams in Fig.~\ref{fig: trivalent}.
It is known that
their matrix elements coincide with the Iqbal, Kozcaz and Vafa's or Awata and Kanno's
refined topological vertecies \cite{AFS2012quantum,AK2008Refined,IKV2007Refined},
and the vertical representation corresponds to the preferred direction.
We prepare the formula for the normal ordering of the intertwiners,
in which the Nekrasov factor appears.

\begin{Definition}
Define the Nekrasov factor to be
\begin{gather*}
N_{\lambda \mu}(u):=
\prod_{(i,j)\in \lambda} \big( 1- u q^{-a_{\mu}(i,j)-1}t^{-\ell_{\lambda}(i,j)} \big) \prod_{(i,j)\in \mu} \big( 1- u q^{a_{\lambda}(i,j)} t^{\ell_{\mu}(i,j)+1} \big)
\\ \hphantom{N_{\lambda \mu}(u):\!}
{}= \prod_{(i,j)\in \lambda} \big( 1- u q^{a_{\lambda}(i,j)}t^{\ell_{\mu}(i,j)+1} \big) \prod_{(i,j)\in \mu} \big( 1- u q^{-a_{\mu}(i,j)-1} t^{-\ell_{\lambda}(i,j)} \big).
\end{gather*}
\end{Definition}

\begin{fact}[\cite{AFS2012quantum}]
\label{fact: intertwiner OPE}
Put $\gamma = (t/q)^{1/2}$.
We have
\begin{gather*}
\Phi_{\lambda^{(i)}}(v_i) \Phi^*_{\mu^{(j)}}(u_j) = \mathcal{G}(u_j/\gamma v_i)^{-1} N_{\mu^{(j)} \lambda^{(i)}}(u_j/\gamma v_i){:}\Phi_{\lambda^{(i)}}(v_i) \Phi^*_{\mu^{(j)}}(u_j){:},
\\
\Phi^*_{\mu^{(j)}}(u_j)\Phi_{\lambda^{(i)}}(v_i)
= \mathcal{G}(v_i/\gamma u_j)^{-1} N_{ \lambda^{(i)} \mu^{(j)}}(v_i/\gamma u_j){:}\Phi^*_{\mu^{(j)}}(u_j)\Phi_{\lambda^{(i)}}(v_i){:},
\\
\Phi_{\lambda^{(i)}}(v_i) \Phi_{\lambda^{(j)}}(v_j)
= \frac{\mathcal{G}(v_j/\gamma^2 v_i)}{N_{\lambda^{(j)}\lambda^{(i)}}(v_j/\gamma^2 v_i)}{:}\Phi_{\lambda^{(i)}}(v_i) \Phi_{\lambda^{(j)}}(v_j){:},
\\
\Phi^*_{\mu^{(i)}}(u_i) \Phi^*_{\mu^{(j)}}(u_j)
= \frac{\mathcal{G}(u_j/u_i)}{N_{\mu^{(j)}\mu^{(i)}}(u_j/u_i)}{:}\Phi^*_{\mu^{(i)}}(u_i) \Phi^*_{\mu^{(j)}}(u_j){:},
\end{gather*}
where $\mathcal{G}(z) = \prod_{i,j=0}^\infty \big(1 - z q^i t^{-j}\big) = \big(z;q,t^{-1}\big)_\infty$.
\end{fact}

Furthermore, we introduce the following operators for convenience,
which is expressed by the cross diagram in Fig.~\ref{fig: cross op.}.

\begin{figure}[h]\centering
\includegraphics{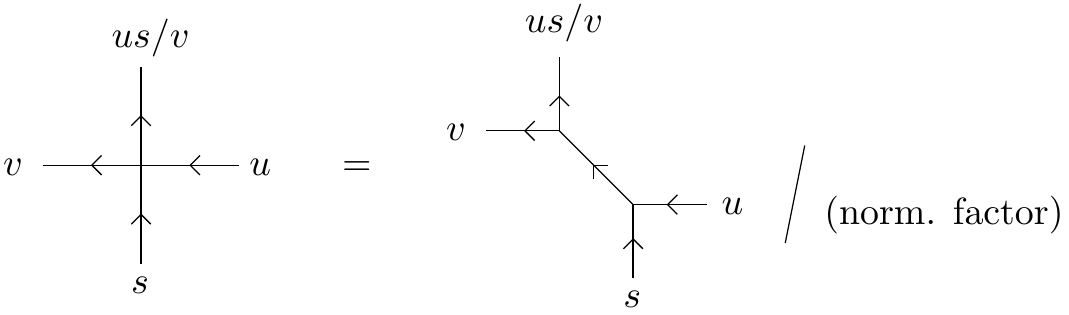}
\caption{The operator $\Phicross$.}\label{fig: cross op.}
\end{figure}

\begin{Definition}
Define
\begin{gather*}
\Phicross \left[ v ;{ us/v\atop s} ;u \right]
=
\dfrac{
\Phi^* \left[ \displaystyle (1,0) ,v ;(0,1),us/v \atop \displaystyle (1,1), -us\right]
\circ
\Phi\left[ \displaystyle (1,1),-us \atop \displaystyle (0,1),s ; (1,0) , u\right]}
{\bra{0}\Phi^*_{\emptyset} \left[ \displaystyle (1,0) ,v ;(0,1),us/v \atop\displaystyle (1,1), -us\right]
\Phi_{\emptyset}\left[ \displaystyle (1,1),-us \atop\displaystyle (0,1),s ; (1,0) , u\right]\ket{0}},
\\
\Phicross \left[ v ;{ us/v, \mu \atop s, \lambda} ;u \right]
=
\dfrac{
\Phi^*_{\mu} \left[ \displaystyle (1,0) ,v ;(0,1),us/v \atop \displaystyle (1,1), -us\right]
\Phi_{\lambda}\left[ \displaystyle (1,1),-us \atop \displaystyle (0,1),s ; (1,0) , u\right]}
{\bra{0}\Phi^*_{\emptyset} \left[ \displaystyle (1,0) ,v ;(0,1),us/v \atop\displaystyle (1,1), -us\right]
\Phi_{\emptyset}\left[ \displaystyle (1,1),-us \atop\displaystyle (0,1),s ; (1,0) , u\right]\ket{0}
}.
\end{gather*}
\end{Definition}

Let $N \in \mathbb{Z}_{\geq 1}$.
The main objects in the previous paper \cite{FOS2019Generalized}
are the following $2N$-valent intertwiners,
which we called the Mukad\'{e} operators after the shape of diagrams
obtained by connecting the trivalent diagrams.
(``Mukade" means centipedes in Japanese. See Fig.~\ref{fig: duality}.)

\begin{figure}[h]
\centering
 \includegraphics[width=14cm]{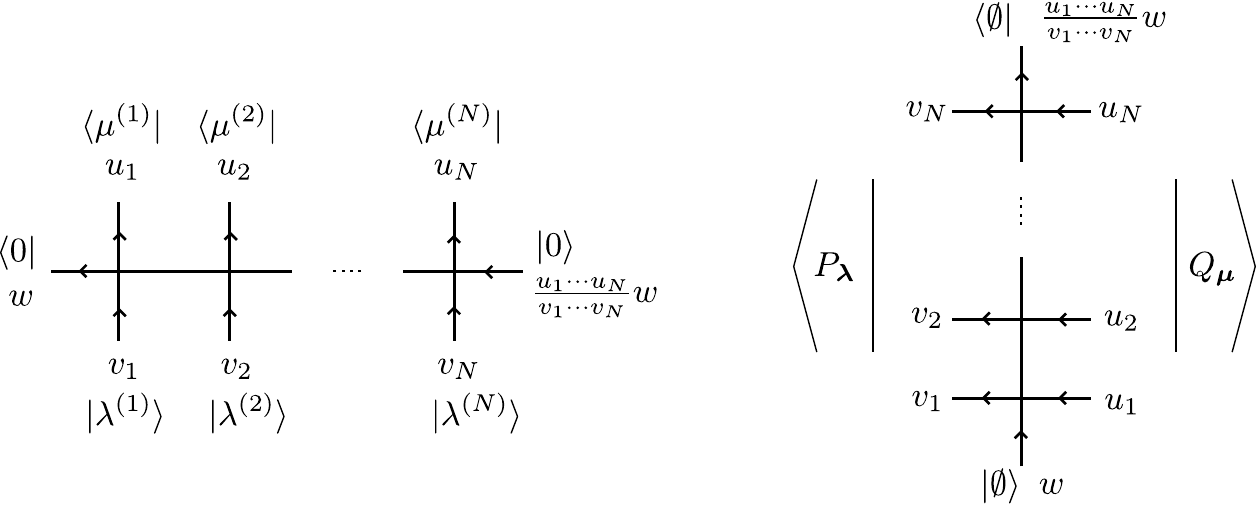}
	\caption{The $S$-duality formula.}	\label{fig: duality}
\end{figure}

\begin{Definition}
Take the composition
\begin{gather}\label{eq: comp. for PhiH}
\bigg( \bigotimes_{1\leq i \leq N}^{\curvearrowright} \cF^{(0,1)}_{v_i} \bigg)\otimes
\cF^{(1,0)}_{w'}\xrightarrow{{\rm id}\otimes \cdots \otimes {\rm id} \otimes \Phicross}
\bigg(\bigotimes_{1\leq i \leq N-1}^{\curvearrowright} \cF^{(0,1)}_{v_i}\bigg)\otimes
\cF^{(1,0)}_{\frac{v_N}{u_N}w'} \otimes \cF^{(0,1)}_{u_N} \nonumber
\\ \hphantom{\bigg( \bigotimes_{1\leq i \leq N}^{\curvearrowright} \cF^{(0,1)}_{v_i} \bigg)\otimes }
\xrightarrow{{\rm id}\otimes \cdots \otimes {\rm id} \otimes \Phicross \otimes {\rm id}}
\cdots \xrightarrow{\Phicross {\rm id}\otimes \cdots \otimes {\rm id}}
\cF^{(1,0)}_{w}\otimes \bigotimes_{1\leq i\leq N}^{\curvearrowright} \cF^{(0,1)}_{u_i}.
\end{gather}
Here $w'=\frac{u_1\cdots u_N}{v_1 \cdots v_N}w$.
Define the operator
\begin{gather*}
T^{H}(\vu, \vv,w)\colon\
\bigotimes_{1\leq i \leq N}^{\curvearrowright} \cF^{(0,1)}_{v_i}\rightarrow
\bigotimes_{1\leq i \leq N}^{\curvearrowright} \cF^{(0,1)}_{u_i}
\end{gather*}
as the vacuum expectation value $\bra{0} \cdots \ket{0}$
of the operator $(\ref{eq: comp. for PhiH})$
with respect to the level $(1,0)$ representation.
Furthermore, we define the normalized operator
\begin{gather*}
\mathcal{T}^{H}(\vu, \vv,w)=\frac{T^{H}(\vu, \vv,w)}
{\bra{\boldsymbol{\emptyset}}T^{H}(\vu, \vv,w) \ket{\boldsymbol{\emptyset}}}.
\end{gather*}
Here, we have set $\ket{\boldsymbol{\emptyset}}=\ket{\emptyset} \otimes \cdots \otimes \ket{\emptyset}$.
\end{Definition}

\begin{Definition}
Define the operators
\begin{gather*}
T^{V}(\vu,\vv,w), \mathcal{T}^{V}(\vu,\vv,w)\colon \
\bigotimes_{1\leq i\leq N}^{\curvearrowright} \cF^{(1,0)}_{u_i}\rightarrow
\bigotimes_{1\leq i\leq N}^{\curvearrowright} \cF^{(1,0)}_{v_i}
\end{gather*}
by
\begin{gather*}
T^{V}(\vu,\vv,w)=
\sum_{\mu^{(1)}, \ldots, \mu^{(N-1)} \in \parset}
\bigotimes^{\curvearrowright}_{1\leq k \leq N}
\Phicross \left[ v_k ;{ \frac{u_k}{v_k}w_k, \mu^{(k)} \atop w_k, \mu^{(k-1)}} ;u_k \right] \qquad \big(\mu^{(0)}=\mu^{(N)}=\emptyset\big),
\\
\mathcal{T}^{V}(\vu,\vv,w)=\frac{T^{V}(\vu,\vv,w)}{\brazero T^{V}(\vu,\vv,w) \ketzero}.
\end{gather*}
Here, $w_k=\frac{u_1\cdots u_{k-1}}{v_1\cdots v_{k-1}}w$.
\end{Definition}

We prepare some notations to treat the $N$-fold Fock tensor spaces.

\begin{notation}
We write
\begin{gather*}
\cF_{\vu}^{(N,0)}=\bigotimes_{1\leq i\leq N}^{\curvearrowright} \cF^{(1,0)}_{u_i}, \qquad
\cF_{\vu}^{(N,0)*}=\bigotimes_{1\leq i\leq N}^{\curvearrowright} \cF^{(1,0)*}_{u_i},
\\
\cF_{\vu}^{(0,N)}=\bigotimes_{1\leq i\leq N}^{\curvearrowright} \cF^{(0,1)}_{u_i} \qquad
(\vu=(u_1,\ldots, u_N)),
\\
a^{(i)}_n=\overbrace{1\otimes \cdots \otimes 1}^{i-1}\otimes
a_n\otimes\overbrace{ 1 \otimes \dots \otimes 1}^{N-i}.
\end{gather*}
For $\vl=\big(\lo,\ldots,\lN\big)\in \parset^N$, set
\begin{gather*}
\ket{\vl}=\bigotimes_{1\leq i \leq N}^{\curvearrowright} \ket{\lambda^{(i)}}, \qquad
\bra{\vl}=\bigotimes_{1\leq i \leq N}^{\curvearrowright} \bra{\lambda^{(i)}},
\\
\ket{a_{\vl}}=a^{(1)}_{-\lo_1}a^{(1)}_{-\lo_2}\cdots a^{(2)}_{-\lt_1}a^{(2)}_{-\lt_2}
\cdots a^{(N)}_{-\lN_1}a^{(N)}_{-\lN_2}\ketzero,
\\
\bra{a_{\vl}}=\brazero\cdots a^{(N)}_{\lN_2}a^{(N)}_{\lN_1}
\cdots a^{(2)}_{\lt_2}a^{(2)}_{\lt_1}\cdots a^{(1)}_{\lo_2} a^{(1)}_{\lo_1}.
\end{gather*}

\end{notation}

In \cite{FOS2019Generalized},
the $S$-duality formula for the matrix elements
of $\mathcal{T}^{V}$ and $\mathcal{T}^{H}$
is proved.
Recall that the matrix elements of $\mathcal{T}^{H}$
with respect to
the basis $( \ket{\lo}\otimes \cdots \otimes \ket{\lN} )$
can be easily calculated by operator products.
On the other hand, the basis on $\cF^{(N,0)}_{\vu}$
which corresponds to $\ket{\lo}\otimes \cdots \otimes \ket{\lN}$
is defied as the eigenfunctions of the operator $X^{(1)}_0$ given as follows.
\begin{Definition}
Define the operator $\Xo(z)=\sum_{n \in \mathbb{Z}} \Xo_n z^{-n}
\in \mathrm{End}\big(\cF^{(N,0)}_{\vu}\big)$ by
\begin{gather*}
\Xo(z) =(\rho_{u_1}\otimes\rho_{u_2}\otimes\cdots \otimes\rho_{u_N})
\circ\Delta^{(N)}(x^+(z)).
\end{gather*}
Here,
\begin{gather*}
\Delta^{(1)}:={\rm id}, \qquad
\Delta^{(N)}:=(\underbrace{{\rm id} \otimes \cdots \otimes{\rm id}}_{N-2} \otimes \Delta)
\circ \cdots \circ({\rm id} \otimes \Delta) \circ \Delta \qquad
(N\geq 2).
\end{gather*}
\end{Definition}

\begin{Definition}\label{Def_X^k}
For $k = 1,2,\dots,N$, set
\begin{gather*}
\Lambda^{(i)}(z) := \varphi^-(\gamma^{1/2}z)\otimes \cdots \otimes\varphi^-(\gamma^{i-3/2}z)\otimes\overbrace{\eta(\gamma^{i-1} z)}^{i\text{-th Fock space}}\otimes 1\otimes\cdots\otimes1 .
\end{gather*}
\end{Definition}

\begin{fact}
On $\cF^{(N,0)}_{\vu}$, we get
\begin{gather*}
\Xo(z)=\sum_{i=1}^N u_i \Lambda^{(i)}(z).
\end{gather*}
\end{fact}

Let $\Lambda$ be the ring of symmetric functions,
and $p_n$ be the power sum symmetric function of degree $n$.
Then the map
\begin{gather}\label{eq: F and Lam}
\cF \ni \ket{a_{\lambda}} \mapsto \prod_{i\geq 1} p_{\lambda_i} \in \Lambda
\end{gather}
gives the isomorphism as graded vector spaces
between $\cF$ and $\Lambda$.
If $N=1$,
the opera\-tor~$X^{(1)}_0$ is essentially the same
as Macdonald's difference operator under this isomorphism \cite{AMOS1995collective}.
Therefore,
its eigenfunctions can be identified with the ordinary Macdonald functions.
In the case of general~$N$,
the eigenfunctions of $\Xo_0$ can be viewed
as a generalization of Macdonald functions.
Their existence theorem
is given in terms of the following generalized dominance partial ordering.

\begin{Definition}
We write $\vl \geq^{\mathrm{L}} \vm$
(resp.~$\vl \geq^{\mathrm{R}} \vm$)
if and only if $|\vl|=|\vm|$ and
\begin{gather*}
|\lambda^{(N)}|+ \cdots + |\lambda^{(j+1)}|+\sum_{k=1}^i \lambda^{(j)}_k
\geq |\mu^{(N)}|+ \cdots + |\mu^{(j+1)}|+\sum_{k=1}^i \mu^{(j)}_k
\\
\bigg(\mathrm{resp.~}|\lambda^{(1)}|+ \cdots + |\lambda^{(j-1)}|+\sum_{k=1}^i \lambda^{(j)}_k
\geq |\mu^{(1)}|+ \cdots + |\mu^{(j-1)}|+\sum_{k=1}^i \mu^{(j)}_k \bigg)
\end{gather*}
for all $i\geq 1$ and $1 \leq j \leq N$.
\end{Definition}

Let us prepare the notation for the vectors
corresponding to the monomial symmetric functions.

\begin{notation}
Let $m_{\lambda}(a_{-n}) \in \mathbb{C}[a_{-1},a_{-2},\ldots]$
be the element in the Heisenberg algebra $\mathcal{H}$
such that $m_{\lambda}(a_{-n})\ket{0}$ coincides with
the monomial symmetric function under the identification~(\ref{eq: F and Lam}).
$m_{\lambda}(a_{-n})$ is the abbreviation for
$m_{\lambda}(a_{-1},a_{-2},\ldots)$.
Note that we often substitute $a_{n}$ or another boson for $a_{-n}$.
\end{notation}

We state
the existence theorem of the generalized Macdonald functions.

\begin{fact}[existence and uniqueness \cite{AFHKSY2011notes,AFOCrystallization2015}]
\label{fact:existence thm of Gn Mac}
For an $N$-tuple of partitions $\vl$,
there exists a~unique vector $\Ket{P_{\vl}} =\Ket{P_{\vl}(\vu)} \in \mathcal{F}^{(N,0)}_{\vu}$
such that
\begin{gather*}
\bullet \quad \Ket{P_{\vl}(\vu)}
 = \prod_{i=1}^N m_{\lambda^{(i)}}(a^{(i)}_{-n}) \ketzero
 + \sum_{\vm <^{\mathrm{L}} \vl} v_{\vl, \vm} \prod_{i=1}^N m_{\mu^{(i)}}(a^{(i)}_{-n}) \ketzero,
\qquad v_{\vl, \vm}\in \mathbb{C}(\vu);
\\
\bullet \quad X^{(1)}_0 \Ket{P_{\vl}(\vu)} = \epsilon_{\vl}(\vu) \Ket{P_{\vl}(\vu)}, \qquad
\epsilon_{\vl}(\vu) \in \mathbb{C}(\vu).
\end{gather*}
Similarly, there exists a unique vector
$\Bra{P_{\vl}}=\Bra{P_{\vl}(\vu) }\in \mathcal{F}_{\vu}^{(N,0)*}$ such that
\begin{gather*}
 \bullet \quad\Bra{P_{\vl}(\vu)}
 = \brazero \prod_{i=1}^N m_{\lambda^{(i)}}(a^{(i)}_{n})
 + \sum_{\vm <^{\mathrm{R}} \vl} v_{\vl, \vm}^* \brazero \prod_{i=1}^N m_{\mu^{(i)}}(a^{(i)}_{n}),
\qquad v^*_{\vl, \vm}\in \mathbb{C}(\vu);
\\
 \bullet \quad \Bra{P_{\vl}(\vu)} \Xo_0 = \epsilon_{\vl}^*(\vu) \Bra{P_{\vl}(\vu)},\qquad
\epsilon_{\vl}^*(\vu) \in \mathbb{C}(\vu).
\end{gather*}
The eigenvalues $\epsilon_{\vl}$ and $\epsilon^*_{\vl}$ are of the forms
\begin{gather*}
\epsilon_{\vl}(\vu)=
\epsilon_{\vl}^*(\vu)=
\sum_{k=1}^N u_k e_{\lambda^{(k)}}, \qquad
e_{\lambda}:= 1+(t-1) \sum_{i \geq 1} (q^{\lambda_i}-1)t^{-i}.
\end{gather*}
\end{fact}
\begin{Definition}
Set
\begin{gather*}
\ket{Q_{\vl}}:=\prod_{i=1}^N\frac{c_{\lambda^{(i)}}}{c'_{\lambda^{(i)}}}
\, \ket{P_{\vl}}.
\end{gather*}
\end{Definition}

\begin{fact}[\cite{AFHKSY2011notes}]
It follows that
\begin{gather*}
\braket{P_{\vl}|Q_{\vm}}=\delta_{\vl,\vm}.
\end{gather*}
\end{fact}

The following is the $S$-duality formula
for changing the preferred directions.
See also Fig.~\ref{fig: duality}.

\begin{Theorem}[\cite{FOS2019Generalized}]\label{fact: chang pref. direc.}
We have
\begin{gather*}
\bra{\vm} \cTH(\vu,\vv;w) \ket{\vl}=
\bra{P_{\vl}} \cTV(\vu,\vv;w) \ket{Q_{\vm}} \times (-1)^{|\vl|+|\vm|}.
\end{gather*}	
\end{Theorem}
Theorem~\ref{fact: chang pref. direc.}
is essentially proved in \cite{FOS2019Generalized}.
See Appendix~\ref{sec: pf of chang pref. direc.}
as to the appearance of the factor $(-1)^{|\vl|+|\vm|}$.
For the explicit form of the matrix elements,
see Fact~\ref{fact: mat el. of TV}.

\section{Proofs of main theorems}\label{sec: nonst.R and Mukade}

\subsection{Non-stationary Ruijsenaars functions and intertwining operators}\label{sec: nonst.R and intertwiner}

In \cite{Shiraishi2019affine},
an operator formula is given for
the non-stationary Ruijsenaars functions
by using the affine screening currents \cite{FKSW2007integrals1,FKSW2007integrals2,KS2008integrals}.
In this subsection,
we show that
the affine screening currents can be reproduced from the intertwiners
of the DIM algebra in the special case of~$\kappa=t^{-1}$, giving an expression of the non-stationary Ruijsenaars functions
in terms of the Mukad\'e operators.
To~help the interested readers, the operator product formulas for the affine screenings given in~\cite{Shiraishi2019affine} are reproduced in Appendix~\ref{sec: nonst. Ruijsenaars and affine sc.}.


\begin{Definition}
Define
\begin{gather*}
\mathcal{A}(z)=
\exp\bigg( {-}\sum_{n>0}\frac{1-t^{-n}}{n(1-q^n)} a_{-n}z^n\bigg)
\exp\bigg( \sum_{n>0}\frac{1-t^n}{n(1-q^{-n})} a_{n} z^{-n}\bigg),
\\
\mathcal{A}^*(z)=
\exp\bigg( \sum_{n>0}\frac{1-t^{-n}}{n(1-q^n)}\gamma^n a_{-n}z^n\bigg)
\exp\bigg( {-}\sum_{n>0}\frac{1-t^n}{n(1-q^{-n})} \gamma^{n}a_{n} z^{-n}\bigg).
\end{gather*}
\end{Definition}

These operators $\mathcal{A}(z)$ and $\mathcal{A}^*(z)$
appear in the following decomposition of
the intertwi\-ners~$\wPhi_{\lambda}(z)$ and $\wPhi^*_{\lambda}(z)$.

\begin{Proposition}\label{prop: decomp. of Phi and A}
For a partition $\lambda = (m_1, \dots, m_{\ell})$,
\begin{gather*}
\wPhi_{(m_1,\ldots, m_{\ell})}(z)=
{:}\wPhi_{\emptyset}\big(t^{-\ell}z\big) \mathcal{A}(q^{m_1}z)
\mathcal{A}\big(q^{m_2}t^{-1}z\big)\cdots \mathcal{A}\big(q^{m_{\ell}}t^{-\ell+1}z\big){:},
\\
\wPhi^*_{(m_1,\ldots, m_{\ell})}(z)=
{:}\wPhi_{\emptyset}^*\big(t^{-\ell}z\big) \mathcal{A}^*\big(q^{m_1}z\big)\mathcal{A}^*\big(q^{m_2}t^{-1}z\big)\cdots \mathcal{A}^*\big(q^{m_{\ell}}t^{-\ell+1}z\big){:}.
\end{gather*}
\end{Proposition}

Let $N\geq 2$ in this subsection. The case $N=1$ will be considered in Section~\ref{sec: N=1}.
Define the screening currents as follows.

\begin{notation}
For an $N$-tuple of the parameters $\vu=(u_1,\ldots, u_N)$,
we write
\begin{gather*}
t^{\alpha_i}\cdot \vu:=\big(u_1,\ldots, u_{i-1}, t u_i, t^{-1} u_{i+1}, u_{i+2},\ldots, u_N\big),
\qquad 1\leq i\leq N-1,
\\
t^{\alpha_0}\cdot \vu:=\big(t^{-1}u_1, u_2,\ldots,u_{N-1}, t u_N\big).
\end{gather*}
Here, $\alpha_0$, $\alpha_1$, \dots, $\alpha_{N-1}$ are regarded as the classical part of the real simple roots of the affine Lie algebra $\widehat{\mathfrak{gl}}_N$.
\end{notation}

\begin{Definition}\label{def: scr. current}
Define the screening currents
$S^{(i)}(y)\colon \cF_{t^{\alpha_i} \cdot \vu} \to \cF_{\vu}$ by
\begin{gather*}
\widetilde{S}^{(i)}(z) := \overbrace{1\otimes\cdots\otimes 1}^{i-1}\otimes
\mathcal{A}^*\big(\gamma^{-i}z\big)\otimes \mathcal{A}\big(\gamma^{-i}z\big)
\otimes \overbrace{1\otimes \cdots \otimes 1}^{N-i-1} ,\qquad i=1,\dots,N-1,
\\
\widetilde{S}^{(0)}(z):=\mathcal{A}(z)\otimes\overbrace{1\otimes\cdots\otimes 1}^{N-2}\otimes\mathcal{A}^*\big(\gamma^{-N}t^{-1}z\big).
\end{gather*}
We cyclically identify $\widetilde{S}^{(i+N)}(z) = \widetilde{S}^{(i)}(z) $.
\end{Definition}

\begin{Remark}
Note that these screening currents essentially coincide with
those in \cite{FKSW2007integrals1,FKSW2007integrals2,KS2008integrals}
when $\kappa = t^{-1}$.
\end{Remark}

\begin{Proposition}
We have
\begin{gather*}
\cA(z) \cA(w) = \frac{(qw/tz;q)_\infty}{(qw/z;q)_\infty}{:}\cA(z) \cA(w){:} ,\qquad
\cAs(z) \cAs(w) = \frac{(w/z;q)_\infty}{(tw/z;q)_\infty}{:}\cAs(z) \cAs(w){:} ,
\\
\cA(z) \cAs(w) = \frac{(q\gamma w/z;q)_\infty}{(q\gamma w/t z;q)_\infty}{:}\cA(z) \cAs(w){:},\quad
\cAs(z) \cA(w) = \frac{(q\gamma w/z;q)_\infty}{(q\gamma w/t z;q)_\infty}{:}\cAs(z) \cA(w){:},
\end{gather*}
and for $N\geq 3$,\footnote{For $N=2$, we have a different form of the normal ordering between $S^{(0)}$ and $S^{(1)}$. However, our results in what follows hold for general $N\geq 2$. } we obtain
\begin{gather*}
 \wS^{(i)}(z)\wS^{(i)}(w) = (1- w/z) \frac{(q w/ t z ; q)_\infty}{(t w/z ; q)_\infty}{:}\wS^{(i)}(z)\wS^{(i)}(w){:} \qquad
 (i = 0,\dots, N-1 ),
 \\
 \wS^{(i)}(z)\wS^{(i+1)}(w) = \frac{(q w/ z ; q)_\infty}{( qw/tz ; q)_\infty}{:}\wS^{(i)}(z)\wS^{(i+1)}(w){:} \qquad
 (i = 0,\dots, N-2 ) ,
 \\
 \wS^{(i+1)}(z)\wS^{(i)}(w)= \frac{(tw/ z ; q)_\infty}{( w/ z ; q)_\infty}{:}\wS^{(i+1)}(z)\wS^{(i)}(w){:} \qquad
 (i = 0,\dots, N-2 ),
 \\
\wS^{(0)}(z)\wS^{(N-1)}(w) = \frac{( t^2 w/ z ; q)_\infty}{( t w/ z ; q)_\infty}{:}\wS^{(0)}(z)\wS^{(N-1)}(w){:},
\\
\wS^{(N-1)}(z)\wS^{(0)}(w) = \frac{( qw/t z ; q)_\infty}{(q w/t^2 z ; q)_\infty}{:}\wS^{(N-1)}(z)\wS^{(0)}(w){:},
\\
\wS^{(i)}(z)\wS^{(j)}(w) = {:}\wS^{(i)}(z)\wS^{(j)}(w): \qquad
\text{for } |i-j|>2.
\end{gather*}
\end{Proposition}

Let us introduce the following vertex operator.%
\footnote{Comparing the notation in \cite{FOS2019Generalized},
we have $\phi_0(z)=\Phi^{(0)}(t^{-1}z)$ with exception for the
spectral parameters of the Fock space.}

\begin{notation}
Write
\begin{gather*}
\gamma^{-1}t^{\pm \delta_i}\cdot \vu:=
\big(\gamma^{-1}u_1,\ldots,\gamma^{-1}u_{i-1}, \gamma^{-1}t^{\pm 1} u_i, \gamma^{-1}u_{i+1}, \ldots, \gamma^{-1}u_N\big).
\end{gather*}
\end{notation}

\begin{Definition}
Define $\phi_0(z)\colon \cF^{(N,0)}_{\gamma^{-1}t^{-\delta_i}\cdot \vu}
\rightarrow \cF^{(N,0)}_{\vu}$ by
\begin{gather*}
\phi_0(z) =
\bigotimes_{1\leq k\leq N}^{\curvearrowright}
{:}\Phi^*_\emptyset\big(t^{-1}\gamma^{-k}z\big)
\Phi_\emptyset\big(t^{-1+\delta_{k,1}}\gamma^{-k+1}z \big){:}.
\end{gather*}
\end{Definition}

\begin{Proposition}\label{prop: phi0 S OPE}
We have
\begin{gather*}
\phi_0(z) \wS^{(1)}(w)=\frac{(q w/ z ; q)_\infty}{(qw/tz ; q)_\infty}{:}\phi_0(z)\wS^{(1)}(w){:}
\\
\phi_0(z) \wS^{(i)}(w)={:}\phi_0(z) \wS^{(i)}(w){:} \qquad (2 \leq i \leq N-1),
\\
\wS^{(1)}(z)\phi_0(w)=\frac{(tw/ z ; q)_\infty}{( w/ z ; q)_\infty}{:}\wS^{(1)}(z)\phi_0(w){:},
\\
\wS^{(i)}(z)\phi_0(w)={:}\wS^{(i)}(z)\phi_0(w){:} \qquad (2 \leq i \leq N-1 ),
\\
\phi_0(z) \wS^{(0)}(w)=\frac{( w/ z ; q)_\infty}{(tw/z ; q)_\infty}{:}\phi_0(z)\wS^{(0)}(w){:},
\\
\wS^{(0)}(z)\phi_0(w)=\frac{(qw/ t z ; q)_\infty}{( qw/ z ; q)_\infty}{:}\wS^{(0)}(z)\phi_0(w){:},
\\
\phi_0(z)\phi_0(w)=\frac{(qw/ t z ; q)_\infty}{( tw/ z ; q)_\infty}{:}\phi_0(z)\phi_0(w){:}.
\end{gather*}
\end{Proposition}

These screening currents and $\phi_0(z)$ can be obtained by a specialization of
the Mukad\'e operators.
Firstly,
we consider the non-affine case and
derive the Macdonald functions from specialized Mukad\'e operators to fix our starting point for making the $p$-traces (Fig.~\ref{fig: cross intro}).

\begin{Definition}
For $1\leq i \leq N$, define
\begin{gather*}
\tilTV_i(z) =\tilTV_i(\vu; z):=
T^V(\vv, \vu ;z)
\Big|_{\substack{v_k \to \gamma^{-1} t^{-\delta_{k,i}}u_k \\ (1\leq k\leq N)}},
 \\
\tilTH_i(z) =\tilTH_i(\vu; z)
:=
T^H(\vv, \vu ;z)
\Big|_{\substack{v_k \to \gamma^{-1} t^{-\delta_{k,i}}u_k \\ (1\leq k\leq N)}}.
\end{gather*}
\end{Definition}

When we construct $T^{V}$,
we need to compose many $\Phicross$'s producing a big summation running over the set of the partitions in $\mathsf{P}^{N-1}$.
By giving a certain condition to the spectral parameters attached to the internal edges, we have the ``restricted operator" $\tilTV_i(z)$.
Then, one finds that all the internal partitions are allowed to run over the one row diagrams satisfying certain interlacing conditions among them.

\begin{fact}[Appendix A in \cite{FOS2019Generalized}]
\label{fact: til Ti = PhiSS..}
We have
\begin{gather*}
\tilTV_i(\vu; z) = \bigg( \frac{(q/t;q)_{\infty}}{(q;q)_{\infty}}\bigg)^{i-1}
 \!\!\!\sum_{0\leq m_{1}\leq m_{2}\leq \leq m_{i-1} <\infty}\!\!\!
\phi_0(z)\prod_{1\leq j \leq i-1}^{\curvearrowright} \widetilde{S}^{(j)}(q^{m_{j}}z)
\prod_{k=1}^{i-1}(u_{k+1}/u_{k})^{m_{k}}.
\end{gather*}
\end{fact}

We call $\tilTV_i$ the ``screened vertex operator".
From these screened vertex operators,
we can construct the Macdonald functions.

\begin{Definition}\label{def: ordinary Mac}
Let $\vs=(s_1,\ldots, s_N)$,
$\vx=(x_1,\ldots, x_N)$ be $N$-tuples of indeterminates.
Define the formal series $\ordmac(\vx;\vs|q,t) \in
\mathbb{Q}(q,t,\vs)[[x_2/x_1,\ldots, x_N/x_{N-1}]]$ by\footnote{$\ordmac(\vx;\vs|q,t)$
coincides with $p_N(\vx;\vs|q,t)$ in \cite{FOS2019Generalized}.}
\begin{gather*}
\ordmac(\boldsymbol{x};\boldsymbol{s}|q,t) =
\sum_{\theta \in \mathsf{M}_N} c_N(\theta;\boldsymbol{s}|q,t)
\prod_{1\leq i<j\leq N}(x_j/x_i)^{\theta_{ij}},
\end{gather*}
where $\mathsf{M}_N=\{ (\theta_{ij})_{1\leq i, j\leq N}\,|\,
\theta_{ij} \in \mathbb{Z}_{\geq 0},\ \theta_{kl}=0 \mbox{ if } k \geq l\}$
is the set of
$N \times N$ strictly upper triangular matrices
with nonnegative integer entries,
and
the coefficient $c_N(\theta;\boldsymbol{s}|q,t)$ is defined~by
\begin{gather*}
c_N(\theta;\boldsymbol{s}|q,t)
=\prod_{k=2}^{N}\prod_{1\le i<j\le k}
\dfrac{\big(q^{\sum_{a>k}(\theta_{ia}-\theta_{ja})}ts_j/s_i;q\big)_{\theta_{ik}}}
{\big(q^{\sum_{a>k}(\theta_{ia}-\theta_{ja})}qs_j/s_i;q\big)_{\theta_{ik}}}
\\ \phantom{c_N(\theta;\boldsymbol{s}|q,t)=}
{}\times\prod_{k=2}^N\prod_{1\le i\le j<k}
\dfrac{\big(q^{-\theta_{jk}+\sum_{a>k}(\theta_{ia}-\theta_{ja})}qs_j/ts_i;q\big)_{\theta_{ik}}}
{\big(q^{-\theta_{jk}+\sum_{a>k}(\theta_{ia}-\theta_{ja})}s_j/s_i;q\big)_{\theta_{ik}}}.
\end{gather*}
\end{Definition}

It is known that $\ordmac(\vx;\vs|q,t)$
is an eigenfunction of Macdonald's difference operator \cite{BFS2014Macdonald, NS2012direct,Shiraishi2005conjecture}.
For some basic facts about $\ordmac(\vx;\vs|q,t)$,
see Appendix~\ref{sec: asymp mac}.
This function can be reproduced as follows.

\begin{fact}[Appendix A of \cite{FOS2019Generalized}]
\label{fact: macdonald from mukade}
It follows that
\begin{gather}
\brazero \tilTV_1(\vx; s_1)\tilTV_2(s_2) \cdots \tilTV_N(s_N) \ketzero
=\prod_{1\leq i<j \leq N }
\frac{(qx_j/x_i;q)_{\infty}}{(tx_j/x_i;q)_{\infty}} \,
\ordmac(\boldsymbol{s};\boldsymbol{x}|q,q/t). \label{eq: Mac from TV}
\end{gather}
\end{fact}

In Appendix A in \cite{FOS2019Generalized},
(\ref{eq: Mac from TV}) was proved up to proportionality.
We can easily calculate the proportional constant
by taking the constant term of $s_i$'s
and using $q$-binomial theorem.

\begin{Remark}
By using the bispectral duality proved in \cite{NS2012direct},
the right hand side in (\ref{eq: Mac from TV}) can be rewritten as
\begin{gather*}
\prod_{1\leq i<j \leq N }
\frac{(qx_j/x_i;q)_{\infty}}{(tx_j/x_i;q)_{\infty}} \,
\ordmac(\boldsymbol{s};\boldsymbol{x}|q,q/t)=
\prod_{1\leq i<j \leq N }
\frac{(qs_j/s_i;q)_{\infty}}{(ts_j/s_i;q)_{\infty}} \,
\ordmac(\boldsymbol{x};\boldsymbol{s}|q,q/t).
\end{gather*}
We can also obtain this equation
by applying the $S$-duality formula for the intertwiners
(Theorem~\ref{fact: chang pref. direc.})
to the left hand side in (\ref{eq: Mac from TV})
and using Fact~\ref{fact: macdonald from mukade} again.
\end{Remark}

The formula (\ref{eq: Mac from TV}) should be understood as the equation
as formal power series in $s_{i+1}/s_i$ and $x_{i+1}/x_i$
($i=1,\ldots, N-1$).
By Fact~\ref{fact: analyticity},
we can also treat the variables $x_i$ and $s_i$ as complex numbers.
We will give an affine analogue of the above facts.
Since analyticity of the non-stationary Ruijsenaars functions
has not been clarified,
we treat $x_i$'s and $s_i$'s as indeterminates in the affine case.

Let $p$ be an indeterminate,
and
consider the following ``loop operator" obtained by the loop of the Mukad\'e operator.
(See Fig.~\ref{fig: mkdloop}.)

\begin{figure}[h]
\centering
\includegraphics[width=8cm]{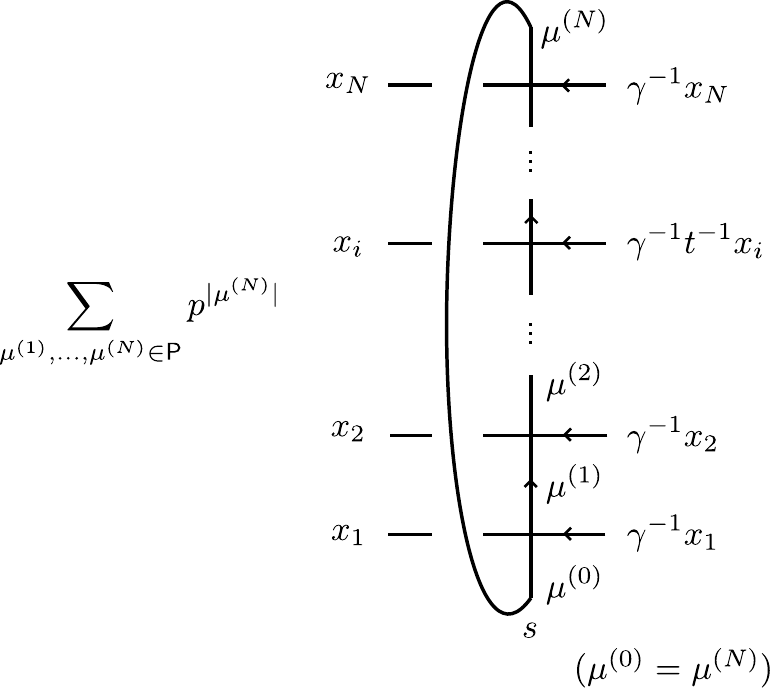}	\caption{$\mkdloop_i(\vx,p;s)$. }	\label{fig: mkdloop}
\end{figure}

\begin{Definition}
Define $\mkdloop_i(\vx,p; s)\colon \cF^{(N,0)}_{\gamma^{-1}t^{-\delta_i}\cdot \vx}\rightarrow \cF^{(N,0)}_{\vx}$ by
\begin{gather*}\label{eq: mkd loop def}
\mkdloop_i(\vx,p; s)=
\sum_{\substack{\mu^{(1)}, \ldots, \mu^{(N)} \in \parset\\ (\mu^{(0)}=\mu^{(N)})}}
p^{|\mN|}
\bigotimes^{\curvearrowright}_{1\leq k \leq N}
\Phicross
\left[ x_k ;
{ \frac{y_1\cdots y_{k}}{x_1\cdots x_{k}}s, \mu^{(k)} \atop \frac{y_1\cdots y_{k-1}}{x_1\cdots x_{k-1}} s, \mu^{(k-1)}} ;
y_k \right]
\Bigg|_{y_k \rightarrow \gamma^{-1} t^{-\delta_{i,k}}x_k}.
\end{gather*}
\end{Definition}

\begin{Definition}
Set the shifted screening currents
\begin{gather*}
S_i(z)=\widetilde{S}^{(i)}\big(t^{-i/N}z\big), \qquad 0 \leq i \leq N-1.
\end{gather*}
\end{Definition}

The screening currents $S_i(z)$ are
the realization of the operator in Appendix~\ref{sec: nonst. Ruijsenaars and affine sc.} in the case of~$\kappa=t^{-1}$.
In Fact~\ref{fact: til Ti = PhiSS..},
we expressed $\tilTV_i$ by composition of screening currents $\widetilde{S}^{(i)}(z)$ and~$\phi_0(z)$.
In the affine case,
we compose the screening currents as follow.

\begin{Definition}
Define the affine screened vertex operators
\begin{gather*}
\phi_i(z)={:}\phi_{i-1}\big(t^{-{1}/{N}}z\big) S_i(z){:} \qquad (i=1,\ldots, N-1),
\\
\phi_{i+N}(z)=\phi_{i}(z) ,
\\
\phi^i_{\lambda} (z) =\phi_{i-\ell(\lambda)} \big(t^{-(\ell(\lambda) + 1)/N} z\big)
\prod^{\curvearrowleft}_{1\leq j\leq \ell(\lambda)} S_{i-j+1}\big(t^{-j/N} q^{\lambda_j} z\big),
\\
\Phi^i_{\lambda}(z)=\left( \frac{(q/t;q)_{\infty}}{(q;q)_{\infty}} \right)^{\ell(\lambda)} \phi^i_{\lambda}(z).
\end{gather*}
\end{Definition}

The operator $\mkdloop_i$ can be expressed as follows.
This is an affine analogue of Fact~\ref{fact: til Ti = PhiSS..}.

\begin{Proposition}\label{prop: aff. scr. vertex}
Let $i=1,\ldots, N$. Then we have
\begin{gather}
\mkdloop_i(\vx,p; s) =\sum_{\lambda \in \parset}
p^{|\lambda|^{(i-1)}}\Phi^{i-1}_{\lambda} \big(t^{i/N}s\big)
\prod_{1\leq k \leq N} x_k^{|\lambda|^{(i-k)}-|\lambda|^{(i-k-1)}} ,
\label{eq: aff. scr. vertex}
\end{gather}
where $|\lambda|^{(i)} = \sum_{j \equiv i+1 (\mathrm{mod} N)}\lambda_j$.
\end{Proposition}

For the proof, we prepare two lemmas.

\begin{Lemma}\label{lem: rational num eq}
Let $i=1,\ldots, N$.
For a partition $\lambda \in \parset$,
set $\mu^{(k)}=
(\lambda_j ;\ 1\leq j \leq \ell(\lambda),\ i-j \equiv k \mod N)$.
Then, we have
\begin{gather*}
\prod_{1\leq k \leq N}
\frac{ N_{\mu^{(k-1)}, \mu^{(k)}}(t^{\delta_{k,i}}) }
{N_{\mu^{(k)}, \mu^{(k)}}(1)}
=\frac{\sfN^{(0|N)}_{\lambda \lambda}\big(t|q,t^{-1/N}\big)}
{\sfN^{(0|N)}_{\lambda \lambda}\big(1|q,t^{-1/N}\big) } ,
\end{gather*}
where $\sfN^{(k|N)}_{\lambda \mu}(z|q,\kappa)$ is defined in Definition~{\rm \ref{def: non-st. Ruij}}.
\end{Lemma}

The proof is given in Section~\ref{sec: pf of Nek fac lemm}.

\begin{Lemma}\label{lem: OPE of Phi_lambda(z)}
Let $i=1,\ldots, N$.
Then we have
\begin{gather*}
\Phi^{i-1}_{\lambda}(z)=t^{-|\lambda|^{(0)}}
\frac{\sfN^{(0|N)}_{\lambda \lambda}\big(t|q,t^{-1/N}\big)}
{\sfN^{(0|N)}_{\lambda \lambda}(1|q,t^{-1/N}\big) }{:}\phi^{i-1}_{\lambda}(z){:} .
\end{gather*}
\end{Lemma}

The proof is given in \cite[Section~2.5]{Shiraishi2019affine}.

\begin{proof}[Proof of Proposition~\ref{prop: aff. scr. vertex}]

First,
the Nekrasov factors
$\prod_{k=1}^N N_{\mu^{(k-1)}, \mu^{(k)}}\big(t^{\delta_{k,i}}\big)$
arise
from the normal ordering product of the operator
\begin{gather*}
\bigotimes^{\curvearrowright}_{1\leq k \leq N}
\Phicross
\left[ x_k ;
{ \frac{y_1\cdots y_{k}}{x_1\cdots x_{k}}s, \mu^{(k)} \atop \frac{y_1\cdots y_{k-1}}{x_1\cdots x_{k-1}} s, \mu^{(k-1)}} ;
y_k \right]
\Bigg|_{y_k \rightarrow \gamma^{-1} t^{-\delta_{i,k}}x_k}.
\end{gather*}
(See Fact~\ref{fact: intertwiner OPE}.)
In general,
for partitions $\nu$ and $\rho$,
we have $N_{\nu, \rho}(1)\neq 0$
(resp.~$N_{\nu, \rho}(t)\neq 0$)
if and only if $\nu\subset \rho $
(resp.~$\bar{\nu}\subset\rho $).
Here,
we put $\bar{\nu}=(\nu_2, \nu_3, \ldots)$
for $\nu=(\nu_1,\nu_2,\nu_3, \ldots)$.
Thus the partitions $\mu^{(k)}$ in (\ref{eq: mkd loop def}) are
restricted by the cyclic interlacing conditions
\begin{gather*}
\mu^{(k-1)} \subset \mu^{(k)}, \qquad k=1,\ldots, N\qquad (k\neq i);
\\
\bar{\mu}^{(i-1)} \subset \mu^{(i)}.
\end{gather*}
Therefore, the $\mu^{(k)}$'s can be expressed by
the single partition
\begin{gather*}
\lambda=\big(\mu^{(i-1)}_1,\mu^{(i-2)}_1,\ldots, \mu^{(1)}_1, \mu^{(N)}_1,\mu^{(N-1)}_1,\ldots, \mu^{(i)}_1,
\\ \hphantom{\lambda=\big(}
\mu^{(i-1)}_2,\mu^{(i-2)}_2,\ldots, \mu^{(1)}_2,\mu^{(N)}_2,\ldots,\mu^{(i)}_2,\mu^{(i-1)}_3,\ldots\big).
\end{gather*}
By using this $\lambda$,
the partitions $\mu^{(k)}$'s can be written as
\begin{gather*}
\mu^{(k)}=(\lambda_j ;\ 1\leq j \leq \ell(\lambda),\ i-j \equiv k \mod N).
\end{gather*}
Recalling Proposition~\ref{prop: decomp. of Phi and A}
and Definition~\ref{def: scr. current},
we have
\begin{gather}
\bigotimes^{\curvearrowright}_{1\leq k \leq N}
\Phicross\left[x_k;{\frac{y_1\cdots y_{k}}{x_1\cdots x_{k}}s, \mu^{(k)}
\atop \frac{y_1\cdots y_{k-1}}{x_1\cdots x_{k-1}} s, \mu^{(k-1)}};y_k \right]
\Bigg|_{y_k \rightarrow \gamma^{-1} t^{-\delta_{i,k}}x_k}\nonumber
\\ \hphantom{\bigotimes^{\curvearrowright}_{1\leq k \leq N}\Phicross}
{}=\prod_{1\leq k \leq N} N_{\mu^{(k-1)}, \mu^{(k)}}\big(t^{\delta_{k,i}}\big)
\frac{q^{2 n(\mu^{(k)'})}}{c_{\mu^{(k)}}c'_{\mu^{(k)}}}
\big(\gamma^{-1}t^{-\delta_{k,i}} x_k\big)^{|\mu^{(k-1)}|}f^{-1}_{\mu^{(k-1)}}
q^{|\mu^{(k)}|}x_k^{-|\mu^{(k)}|}\nonumber
\\ \hphantom{\bigotimes^{\curvearrowright}_{1\leq k \leq N}\Phicross\quad}
{} \times {:} \phi_0(t^{-l'}s)\prod_{k=1}^{i-1} \prod_{\alpha=0}^{l'}
\widetilde{S}^{(k)}\big(q^{\mu^{(k)}_{\alpha+1}} t^{-\alpha}s\big)
\prod_{\alpha=0}^{l'-1}
\widetilde{S}^{(0)}\big(q^{\mu^{(0)}_{\alpha+1}} t^{-\alpha}s\big)\nonumber
\\ \hphantom{\bigotimes^{\curvearrowright}_{1\leq k \leq N}\Phicross\quad}
{}\times \prod_{k=i}^{N-1} \prod_{\alpha=1}^{l'}
\widetilde{S}^{(k)}\big(q^{\mu^{(k)}_{\alpha}} t^{-\alpha}s\big){:},
\label{eq: decomp. by scr}
\end{gather}
where we put $l'=\big\lfloor \frac{\ell(\lambda)-i}{N} \big\rfloor+1$.
For $a\in \mathbb{Q}$,
$\lfloor a \rfloor$ is defined to be the integer $n$ satisfying
$n \leq a <n+1$.
Lemma~\ref{lem: rational num eq} and the equation
\begin{gather*}
c_{\lambda}c'_{\lambda}=(-1)^{(|\lambda|)} q^{n(\lambda') +|\lambda|}t^{n(\lambda)}N_{\lambda,\lambda}(1)
\end{gather*}
show that
\begin{gather*}
\prod_{1\leq k \leq N} N_{\mu^{(k-1)}, \mu^{(k)}}(t^{\delta_{k,i}})
\frac{q^{2 n(\mu^{(k)'})}}{c_{\mu^{(k)}}c'_{\mu^{(k)}}}
\big(\gamma^{-1}t^{-\delta_{k,i}} x_k\big)^{|\mu^{(k-1)}|}f^{-1}_{\mu^{(k-1)}}
q^{|\mu^{(k)}|}x_k^{-|\mu^{(k)}|}
\\ \hphantom{\prod_{1\leq k \leq N}}
{}=t^{-|\lambda|^{(0)}}
\frac{\sfN^{(0)}_{\lambda \lambda}\big(t|q,t^{-1/N}\big)}
{\sfN^{(0)}_{\lambda \lambda}\big(1|q,t^{-1/N}\big) }
\prod_{1\leq k \leq N} x_k^{|\lambda|^{(i-k)}-|\lambda|^{(i-k-1)}}.
\end{gather*}
Furthermore, by using the shifted screening current $S_k(z)$'s,
the operator part in (\ref{eq: decomp. by scr})
can be rewritten as
\begin{gather*}
{:} \phi_0(t^{-l'}s)\prod_{k=1}^{i-1} \prod_{\alpha=0}^{l'}
\widetilde{S}^{(k)}\big(q^{\mu^{(k)}_{\alpha+1}} t^{-\alpha}s\big)
\prod_{\alpha=0}^{l'-1}
\widetilde{S}^{(0)}\big(q^{\mu^{(0)}_{\alpha+1}} t^{-\alpha}s\big)
 \prod_{k=i}^{N-1} \prod_{\alpha=1}^{l'}
\widetilde{S}^{(k)}\big(q^{\mu^{(k)}_{\alpha}} t^{-\alpha}s\big){:}
\\ \hphantom{{:} \phi_0(t^{-l'}s)}
{}={:}\phi_0\big(t^{-l'}s\big)S_{i-\ell(\lambda)-1}\big(t^{-l'}t^{[i-\ell(\lambda)-1]/N}s\big)
\cdots S_2\big(t^{-l'}t^{2/N}\big)S_1\big(t^{-l'}t^{1/N}s\big)
\\ \hphantom{{:} \phi_0(t^{-l'}s)={:}}
{} \times \prod_{1\leq j \leq \ell(\lambda)}S_{i-j}\big(t^{-j/N}q^{\lambda_j} t^{i/N} s\big){:}
={:}\phi^{i-1}_{\lambda}\big(t^{i/N}s\big){:}.
\end{gather*}
Here, [a] is the integer satisfying that $0 \leq [a] <N$
and $[a] \equiv a \mod N$.
Therefore, by Lemma~\ref{lem: OPE of Phi_lambda(z)},
we can show that (\ref{eq: decomp. by scr}) coincides with
the RHS of (\ref{eq: aff. scr. vertex}).
\end{proof}

This proposition says that the vertex operators $\mkdloop_i(\vx,p; s)$ can be identified with the screened vertex operators which are used to construct the non-stationary Ruijsenaars function in \cite{Shiraishi2019affine}, though in our case, $\kappa$ should be specialized to $t^{-1/N}$.
(See Appendix~\ref{sec: nonst. Ruijsenaars and affine sc.}.)
This motivates us to state the affine analogue of Fact~\ref{fact: macdonald from mukade}, that is, to construct the non-stationary Ruijsenaars function as the matrix element of the composition of $\mkdloop_i(\vx,p; s)$'s.

In order to state the claim, we introduce the non-stationary Ruijsenaars function.
\begin{Definition}[\cite{Shiraishi2019affine}]\label{def: non-st. Ruij}
Let $\vs=(s_1,\ldots, s_N)$,
$\vx=(x_1,\ldots, x_N)$ be $N$-tuples of indeterminates.
Define $\nonstrui(\vx,p|\vs,\kappa|q,t)$
in $\mathbb{Q}(q,t,\vs)[[px_2/x_1,\ldots, px_N/x_{N-1}, px_1/x_N]]$ by
\begin{gather*}
f^{\widehat{\mathfrak{gl}}_N}(\vx,p|\vs,\kappa|q,t)
=
\sum_{\lambda^{(1)},\ldots,\lambda^{(N)}\in {\mathsf P}}
\prod_{i,j=1}^N
{\sfN^{(j-i|N)}_{\lambda^{(i)},\lambda^{(j)}} (ts_j/s_i|q,\kappa) \over \sfN^{(j-i|N)}_{\lambda^{(i)},\lambda^{(j)}} (s_j/s_i|q,\kappa)}
 \prod_{\beta=1}^N\prod_{\alpha\geq 1} ( p x_{\alpha+\beta}/tx_{\alpha+\beta-1})^{\lambda^{(\beta)}_\alpha}.
\end{gather*}
We cyclically identify $x_{i+N}=x_i$
and put
\begin{gather*}
\sfN^{(k|N)}_{\lambda,\mu}(u|q,\kappa)=\sfN^{(k)}_{\lambda,\mu}(u|q,\kappa)
\\ \hphantom{\sfN^{(k|N)}_{\lambda,\mu}(u|q,\kappa)}
{}= \!\!\!\prod_{j\geq i\geq 1 \atop j-i \equiv k ({\rm mod} N)}\!\!\!\!\!
\big(u q^{-\mu_i+\lambda_{j+1}} \kappa^{-i+j};q\big)_{\lambda_j-\lambda_{j+1}}
\!\!\!\!\!\!\!\prod_{\beta\geq \alpha \geq 1 \atop \beta-\alpha \equiv -k-1 ({\rm mod} N)}\!\!\!\!\!\!\!\!\!\!\!
\big(u q^{\lambda_{\alpha}-\mu_\beta} \kappa^{\alpha-\beta-1};q\big)_{\mu_{\beta}-\mu_{\beta+1}}.
\end{gather*}
\end{Definition}

Then, we obtain the following theorem.
(See also Fig.~\ref{fig: loop_script}.)

\begin{figure}[h]
\centering
 \includegraphics[width=10cm]{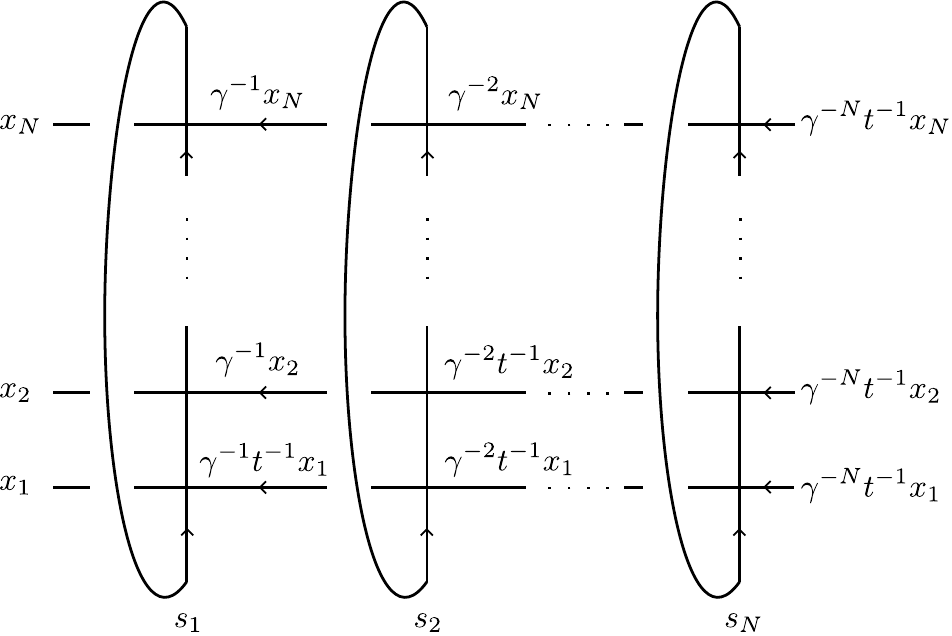}
	\caption{Non-stationary Ruijsenaars function $\nonstrui\big(\vx',p^{{1}/{N}}|\vs',t^{-{1}/{N}}|q,t\big)$. }
	\label{fig: loop_script}
\end{figure}

\begin{Theorem}\label{thm: Tr TH TH...}
Let
\begin{gather*}
\vx^{(i)}=\big(\gamma^{-i}t^{-1} x_1, \ldots, \gamma^{-i}t^{-1} x_{i},
\gamma^{-i} x_{i+1}, \ldots,\gamma^{-i} x_N\big),\qquad
0\leq i \leq N.
\end{gather*}
Then we obtain
\begin{gather*}
\brazero
\mkdloop_1\big(\vx^{(0)},p; s_1\big) \mkdloop_2\big(\vx^{(1)},p; s_2\big) \cdots \mkdloop_N\big(\vx^{(N-1)},p; s_N\big)\ketzero
\\ \hphantom{\brazero\mkdloop_1\big(}
{} =\prod_{1\leq i<j\leq N}
\frac{(qs_j/s_i;q)_{\infty}}{(ts_j/s_i;q)_{\infty}}\,
\nonstrui\big(\vx',p^{{1}/{N}}|\vs',t^{-{1}/{N}}|q,t\big).
\end{gather*}
Here, we set
\begin{gather*}
\vs'=(s'_1,\ldots ,s'_N),\qquad s'_k =t^{{k}/{N}}s_k,
\\
\vx'=(x'_1,\ldots ,x'_N),\qquad x'_k= p^{-{k}/{N}}x_k.
\end{gather*}
\end{Theorem}

\begin{proof}
Proposition~\ref{prop: aff. scr. vertex} gives
\begin{gather}
\brazero\mkdloop_1\big(\vx^{(0)},p; s_1\big) \mkdloop_2\big(\vx^{(1)},p; s_2\big) \cdots \mkdloop_N\big(\vx^{(N-1)},p; s_N\big)\ketzero \nonumber
\\ \hphantom{\brazero\mkdloop_1\big(}
{}=\sum_{\lambda^{(1)}, \ldots, \lambda^{(N)}}
\brazero \Phi^0_{\lo}\big(t^{1/N}s_1\big) \Phi^1_{\lt}\big(t^{2/N}s_2\big)\cdots
\Phi^{N-1}_{\lN}\big(t^{N/N}s_N\big) \ketzero \nonumber
\\\hphantom{\brazero\mkdloop_1\big(=}
{} \times \prod_{i=1}^N \prod_{k=1}^N (\gamma^{-i+1} t^{-\delta_{k\leq i-1}} x_k)^{|\lambda^{(i)}|^{(i-k)}-|\lambda^{(i)}|^{(i-k-1)}}
 \prod_{i=1}^N p^{|\lambda^{(i)}|^{(i-1)}}.
\label{eq: cal Tloop}
\end{gather}
Here, $\delta_{a\leq b}$ is $1$ if $a \leq b$ or $0$ if $a >b$.
Then, we have
\begin{gather*}
\prod_{i=1}^N \prod_{k=1}^N (\gamma^{-i+1})^{|\lambda^{(i)}|^{(i-k)}-|\lambda^{(i)}|^{(i-k-1)}}
=1,
\\
\prod_{i=1}^N \prod_{k=1}^N
(t^{-\delta_{k\leq i-1}})^{|\lambda^{(i)}|^{(i-k)}-|\lambda^{(i)}|^{(i-k-1)}}
=\prod_{i=1}^N t^{-|\lambda^{(i)}|^{(i-1)}+|\lambda^{(i)}|^{(0)}}
\end{gather*}
and
\begin{gather*}
\prod_{i=1}^N \prod_{k=1}^N
x_k^{|\lambda^{(i)}|^{(i-k)}-|\lambda^{(i)}|^{(i-k-1)}}
 \prod_{i=1}^N p^{|\lambda^{(i)}|^{(i-1)}}=
\prod_{i=1}^N \prod_{j=1}^{\ell(\lambda^{(i)})}
(x_{i-j+1}/x_{i-j})^{\lambda^{(i)}_j} \prod_{i=1}^N p^{|\lambda^{(i)}|^{(i-1)}}
\\ \hphantom{\prod_{i=1}^N \prod_{k=1}^N
x_k^{|\lambda^{(i)}|^{(i-k)}-|\lambda^{(i)}|^{(i-k-1)}}
 \prod_{i=1}^N p^{|\lambda^{(i)}|^{(i-1)}}}
{}=\prod_{i=1}^N \prod_{j=1}^{\ell(\lambda^{(i)})}
(p^{1/N}x'_{i-j+1}/x'_{i-j})^{\lambda^{(i)}_j}.
\end{gather*}
Therefore, Fact~\ref{fact: <Phi...>=f} shows that (\ref{eq: cal Tloop})
is equal to
\begin{gather}
\prod_{1\leq i<j\leq N}
\frac{(qs_j/s_i;q)_{\infty}}{(ts_j/s_i;q)_{\infty}}\,
\nonstrui\big((1/x'_N,\ldots, 1/x'_1),p^{{1}/{N}}\big|(1/s'_N,\ldots, 1/s'_N),t^{-{1}/{N}}\big|q,t\big).
\label{eq: transformed nonstrui}
\end{gather}
Finally,
it can be easily shown that
the non-stationary Ruijsenaars function in (\ref{eq: transformed nonstrui})
coincides with
$\nonstrui\big(\vx',p^{{1}/{N}}|\vs',t^{-{1}/{N}}|q,t\big)$.
\end{proof}

\begin{Remark}\label{rem: T^H trace}
The LHS in this theorem can be rewritten by the trace of the operators
$\tilTH$.
For an operator $A \in \mathrm{End}\big((\cF^{(0,1)})^{\otimes N} \big)$, set
the formal power series
\begin{gather*}
\mathrm{Tr}_p ( A )= \sum_{\vl \in \parset^N} p^{|\vl|} \bra{\vl} A \ket{\vl}.
\end{gather*}
Then it is clear that
\begin{gather*}
\brazero
\mkdloop_1\big(\vx^{(0)},p; s_1\big) \cdots \mkdloop_N\big(\vx^{(N-1)},p ; s_N\big)
\ketzero
=\mathrm{Tr}_p \big(\tilTH_N \big(\vs^{(N-1)};x_N\big) \cdots
\tilTH_1 \big( \vs^{(0)};x_1\big)\big).
\end{gather*}
Here,
\begin{gather*}
\vs^{(i)}=\big(\gamma^{-i}t^{-1} s_1, \ldots, \gamma^{-i}t^{-1} s_{i},
\gamma^{-i} s_{i+1}, \ldots,\gamma^{-i} s_N\big),\qquad
0\leq i \leq N.
\end{gather*}
\end{Remark}

\subsection[Lift to elliptic hypergeometric series and non-stationary Ruijsenaars function]
{Lift to elliptic hypergeometric series\\ and non-stationary Ruijsenaars function}
In the previous subsection,
we took the loop at vertical direction of
the reticulate diagram.
Next,
we calculate a loop at horizontal direction.
This loop can reproduce the lift
$\fellip_N(\vx,\vs|q,t,p)$
of the Macdonald function $\ordmac$ by the elliptic gamma functions.
Let us recall the definition of $\fellip_N(\vx,\vs|q,t,p)$.

\begin{Definition}\label{def: ellip Mac}
Define $\fellip_N(\vx;\vs|q,t,p) \in \mathbb{Q}(q,t,\vs)[[p,x_2/x_1,\ldots,x_N/x_{N-1}]]$ by
\begin{gather*}
\fellip_N(\vx;\vs|q,t,p)
=\sum_{\theta \in \mathsf{M}_N} \cellip_N(\theta;\boldsymbol{s}|q,q/t,p)
\prod_{1\leq i<j\leq N} (x_j/x_i)^{\theta_{ij}},
\end{gather*}
where
\begin{gather*}
\cellip_N(\theta;\boldsymbol{s}|q,t,p)
=\prod_{k=2}^{N}\prod_{1\le i<j\le k}
\dfrac{\wg{q^{\sum_{a>k}(\theta_{ia}-\theta_{ja})}ts_j/s_i}{\theta_{ik}}}
{\wg{q^{\sum_{a>k}(\theta_{ia}-\theta_{ja})}qs_j/s_i}{\theta_{ik}}}
\\ \hphantom{\cellip_N(\theta;\boldsymbol{s}|q,t,p)=}
\times
\prod_{k=2}^N
\prod_{1\le i\le j<k}
\dfrac{\wg{q^{-\theta_{jk}+\sum_{a>k}(\theta_{ia}-\theta_{ja})}qs_j/ts_i}{\theta_{ik}}}
{\wg{q^{-\theta_{jk}+\sum_{a>k}(\theta_{ia}-\theta_{ja})}s_j/s_i}{\theta_{ik}}}.
\end{gather*}
$\wg{-}{n}$ is defined in~(\ref{eq: def wg}).
\end{Definition}

It is clear that
the function $\fellip_N(\vx,\vs|q,t,p)$ is reduced to the ordinary Macdonald function,
i.e.,
\begin{gather*}
\fellip_N(\vx,\vs|q,t,p)
\underset{p \rightarrow 0}{\longrightarrow}
f^{\mathfrak{gl}_N}(\vx,\vs|q,q/t).
\end{gather*}

We give a realization of this elliptic lift by taking the trace of the Mukad\'e operators.
\begin{Definition}
For $A \in \mathrm{End}(\cF^{\otimes N})$,
define the $p$-trace $\tr (p^d A)$ to be
\begin{gather*}
\tr (p^d A) =
\sum_{\vl \in \parset^N} p^{|\vl|}
\frac{\bra{a_{\vl}} A \ket{a_{\vl}}}{\braket{a_{\vl}|a_{\vl}}}.
\end{gather*}
Note that the trace $\tr (p^d {-})$ certainly does not depend
on bases.
\end{Definition}

Our main purpose in this subsection is to compute the trace of the following operator.
See also Fig.~\ref{fig: loop_script_hor}.

\begin{figure}[h]
\centering
 \includegraphics[width=9cm]{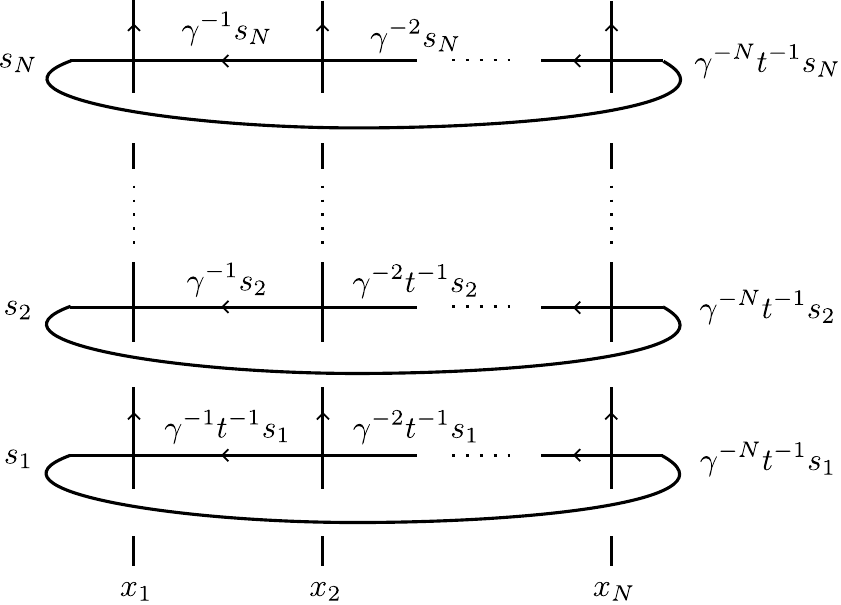}
	\caption{$\tr \big( p^d \wcT^N(\vs;\vx) \big)$. }
	\label{fig: loop_script_hor}
\end{figure}

\begin{Definition}
Define the $N$-compositions of operators $\tilTV_i$ as
\begin{gather*}
\wcT^N(\vs;\vx) := \tilTV_1\big(\vs^{(0)}; x_1\big) \tilTV_2\big(\vs^{(1)}; x_2\big)\cdots \tilTV_N\big(\vs^{(N-1)}; x_N\big),
\end{gather*}
where
\begin{gather}\label{eq: notation of s^i}
\vs^{(i)} = \big(\gamma^{-i} t^{-1} s_1, \dots, \gamma^{-i} t^{-1} s_{i},\gamma^{-i}s_{i+1},\dots,\gamma^{-i}u_N \big) .
\end{gather}
\end{Definition}

We state the key property of the $p$-trace.

\begin{notation}
Write
\begin{gather*}
\widehat{\mathcal{H}}^{N}[z]
:=
\left\{{\rm e}^{\sum_{r>0,1\leq i \leq N} R^{(i)}_{-r} a^{(i)}_{-r} z^r}
{\rm e}^{\sum_{r>0,1\leq i \leq N} R^{(i)}_{r} a^{(i)}_{r} z^{-r}}
\in \mathrm{End}\big(\cF^{\otimes N}\big)\big[\big[z^{\pm 1}\big]\big] \,|\,
R^{(i)}_r \in \mathbb{C} \right\}.
\end{gather*}
\end{notation}

\begin{Lemma}\label{lem: tr}
Let $A_i(z) \in \widehat{\mathcal{H}}^{N}[z]$ $(i=1,\ldots, n)$
be operators which satisfy
\begin{gather*}
A_i(z)A_j(w)=\prod_{l} \frac{1-a^{i,j}_lw/z}{1-b^{i,j}_lw/z}{:}A_i(z)A_j(w){:}
\qquad (a^{i,j}_l,b^{i,j}_l \in \mathbb{C}, \; i,j=1,\ldots, n),
\end{gather*}
where the product with respect to $l$
can be either finite or infinite if it converges.
Then, it follows that
\begin{gather*}
\tr \big( p^d A_1(z_1) \cdots A_n(z_n)\big)=
\frac{1}{(p;p)_{\infty}^N}\prod_{1\leq i<j\leq n}
\prod_l \frac{(a^{i,j}_lz_j/z_i;p)_{\infty}}{(b^{i,j}_lz_j/z_i;p)_{\infty}}
 \prod_{1\leq j\leq i\leq n}
\prod_l \frac{(pa^{i,j}_lz_j/z_i;p)_{\infty}}{(pb^{i,j}_lz_j/z_i;p)_{\infty}}.
\end{gather*}
In particular,
if the operators satisfy
\begin{gather*}
A_1(z)A_2(w) = \frac{(a w/z;q)_\infty}{(b w/z;q)_\infty}{:}A_1(z)A_2(w){:} ,\quad\
A_2(w)A_1(z) = \frac{(q z/ bw ;q)_\infty}{(q z/a w;q)_\infty}{:}A_2(w)A_1(z){:},
\end{gather*}
then we can rewrite the result by the elliptic gamma functions:
\begin{gather*}
\tr \big(p^d A_1(z)A_2(w)\big)=\frac{1}{(p;p)_\infty^N}
\frac{\Gamma(b w/z;q,p)}{\Gamma(a w/z;q,p)}\prod_{i=1}^2
\prod_l \frac{(pa^{i,i}_l;p)_{\infty}}{(pb^{i,i}_l;p)_{\infty}}.
\end{gather*}

\begin{proof}
Let $A^{\pm}_i(z)$ be the operators of the forms
\begin{gather*}
A^{-}_i(z)={\rm e}^{\sum_{r>0,1\leq i \leq N} R^{(i)}_{-r} a^{(i)}_{-r} z^r},\qquad
A^{+}_i(z)={\rm e}^{\sum_{r>0,1\leq i \leq N} R^{(i)}_{r} a^{(i)}_{r} z^{-r}} \qquad
\big(R^{(i)}_{\pm r} \in \mathbb{C}\big)
\end{gather*}
such that $A_i(z)=A^{-}_i(z)A^{+}_i(z)$.
Then we have
\begin{gather*}
\tr \big( p^d A_1(z_1) \cdots A_n(z_n)\big)
\\ \qquad
{}=\prod_{1\leq i<j\leq n}
\prod_l \frac{1-a^{i,j}_lz_j/z_i}{1-b^{i,j}_lz_j/z_i}\,
\tr \big(p^dA^-_1(z_1) \cdots A^-_n(z_n) A^+_1(z_1) \cdots A^+_n(z_n)\big)
\\ \qquad
{}=\prod_{1\leq i<j\leq n}
\prod_l \frac{1-a^{i,j}_lz_j/z_i}{1-b^{i,j}_lz_j/z_i}\,
\tr \big( p^d A^+_1\big(p^{-1}z_1\big) \cdots A^+_n(p^{-1}z_n)
A^-_1(z_1) \cdots A^-_n(z_n)\big)
\\ \qquad
{}=\prod_{1\leq i<j\leq n}
\prod_l \frac{1-a^{i,j}_lz_j/z_i}{1-b^{i,j}_lz_j/z_i}
\prod_{i,j=1}^n\prod_l \frac{1-p a^{i,j}_lz_j/z_i}{1-p b^{i,j}_lz_j/z_i}
\\ \phantom{\qquad{}=\prod_{1\leq i<j\leq n}}
\times\tr \big(p^d A^-_1(z_1)\cdots A^-_n(z_n)A^+_1\big(p^{-1}z_1\big) \cdots A^+_n\big(p^{-1}z_n\big)\big).
\end{gather*}
Since $\tr(p^d\cdot 1)=\frac{1}{(p;p)^N_{\infty}}$,
by repeating the calculation above,
it can shown that for any $m \in \mathbb{Z}_{>0}$,
\begin{gather*}
\tr \big( p^d A_1(z_1) \cdots A_n(z_n)\big)
\\ \qquad
{}=\frac{1}{(p;p)_{\infty}^N} \prod_{1\leq i<j\leq n}
\prod_l \frac{\big(a^{i,j}_lz_j/z_i;p\big)_{m}}{\big(b^{i,j}_lz_j/z_i;p\big)_{m}}
 \prod_{1\leq j\leq i\leq n}
\prod_l \frac{\big(pa^{i,j}_lz_j/z_i;p\big)_{m}}{\big(pb^{i,j}_lz_j/z_i;p\big)_{m}}
+\mathcal{O}(p^m).
\end{gather*}
This completes the proof.
\end{proof}
\end{Lemma}

\begin{Proposition}\label{prop: trace wcT = fellip}
We obtain
\begin{gather*}
 \tr \big( p^d \wcT^N(\vs;\vx) \big)
=
\left(\frac{(pq/t;q,p)_{\infty}}{(p;p)_{\infty}(pt;q,p)_{\infty}} \right)^{N}
\prod_{1\leq i<j\leq N} \frac{\Gamma(tx_j/x_i;q,p)}{\Gamma(qx_j/x_i;q,p)}
\, \fellip_N(\vs;\vx|q,t,p).
\end{gather*}
\end{Proposition}

\begin{proof}
Using Fact~\ref{fact: til Ti = PhiSS..} and Lemma~\ref{lem: tr},
we can compute the trace as
\begin{gather*}
\tr \big( p^d \wcT^N(\vs;\vx) \big)
\\ \hphantom{\tr \big(}
{}=\left( \frac{(q/t;q)_{\infty}}{(q;q)_{\infty}}\right)^{N(N-1)/2}(p;p)^{N(N-3)/2}
\left(\frac{(pq/t;q,p)_\infty}{(pt;q,p)_\infty}\right)^{N(N+1)/2}
\\ \hphantom{\tr \big(=}
{}\times\prod_{1\leq i< j \leq N}\frac{\elG{tx_j/x_i}}{\elG{qx_j/tx_i}}
\sum_{(m_{i,j})_{1 \leq i<j\leq N}}
\prod_{1\leq i<j\leq N} \frac{\elG{ q q^{m_{i,j} - m_{i-1,j}}/t}}{\elG{q q^{m_{i,j} - m_{i-1,j}}}}
\\ \hphantom{\tr \big(=}
\times \prod_{1\leq i< j\leq N }
\frac{\elG{qq^{m_{1,j}} x_j/tx_i}}{\elG{q q^{m_{1,j}}x_j/x_i}}
\\ \hphantom{\tr \big(=}
 \times \prod_{k=2}^{N-1}\prod_{k\leq i < j \leq N }
\frac{\elG{q q^{m_{k-1,j} - m_{k-1,i}}x_j/x_i}}{\elG{q^{m_{k-1,j} - m_{k-1,i}}x_j/x_i}}
\frac{\elG{tq^{m_{k-1,j} - m_{k-1,i}}x_j/x_i}}{\elG{q q^{m_{k-1,j} - m_{k-1,i}}x_j/tx_i}}
\\ \hphantom{\tr \big(=}
 \times \prod_{k=2}^{N-1}\prod_{k\leq i < j \leq N }
\frac{\elG{q^{m_{k-2,j} - m_{k-1,i}}x_j/x_i}}{\elG{t q^{m_{i-2,j} - m_{k-1,i}}x_j/x_i}}
\frac{\elG{qq^{m_{k,j} - m_{k-1,i}}x_j/tx_i}}{\elG{q q^{m_{k,j} - m_{k-1,i}}x_j/x_i}}
\\ \hphantom{\tr \big(=}
\times\prod_{k=2}^N \prod_{i=1}^{k-1} \big(t^{\delta_{i,k-1}}s_{i+1}/s_{i}\big)^{m_{i,k}}.
\end{gather*}
Here, the summation runs over all integers such that
$m_{j-1,j}\geq m_{j-2,j} \geq \cdots \geq m_{1,j} \geq 0$
($j=2,\ldots , N$).
Then, it can be shown that
\begin{gather*}
\left( \frac{(q/t;q)_{\infty}}{(q;q)_{\infty}}\right)^{N(N-1)/2} (p;p)^{N(N-3)/2} \left(\frac{(pq/t;q,p)_\infty}{(pt;q,p)_\infty}\right)^{N(N+1)/2}
\\ \qquad
=\left(\frac{\Gamma(q;q,p)}{\Gamma(q/t;q,p)}\right)^{N(N-1)/2}
\left(\frac{(pq/t;q,p)_{\infty}}{(p;p)_{\infty}(pt;q,p)_{\infty}} \right)^{N}.
\end{gather*}
Put $\theta_{1,j}:=m_{1,j}$ ($j=2,\ldots, N$)
and $\theta_{i,j}:=m_{i,j}-m_{i-1,j}$ ($j \geq 3$, $i=2,\ldots, j-1$).
We have
\begin{gather*}
\left(\frac{\Gamma(q;q,p)}{\Gamma(q/t;q,p)}\right)^{N(N-1)/2}
\prod_{1\leq i<j\leq N} \frac{\elG{ q q^{m_{i,j} - m_{i-1,j}}/t}}{\elG{q q^{m_{i,j} - m_{i-1,j}}}} =
\prod_{1\leq i<j\leq N}\frac{\wg{q/t}{\theta_{i,j}}}{\wg{q}{\theta_{i,j}}},
\\
\prod_{1\leq i< j\leq N }
\frac{\elG{qq^{m_{1,j}} x_j/tx_i}}{\elG{q q^{m_{1,j}}x_j/x_i}}
=
\prod_{1\leq i< j\leq N }
\frac{\wg{qx_j/tx_i}{\theta_{1,j}}}{\wg{qx_j/x_i}{\theta_{1,j}}}
\frac{\elG{qx_j/tx_i}}{\elG{qx_j/x_i}}
\end{gather*}
and
\begin{gather}
\prod_{k=2}^{N-1}\prod_{k\leq i < j \leq N }
\frac{\elG{q q^{m_{k-1,j} - m_{k-1,i}}x_j/x_i}}{\elG{q^{m_{k-1,j} - m_{k-1,i}}x_j/x_i}}
\frac{\elG{tq^{m_{k-1,j} - m_{k-1,i}}x_j/x_i}}{\elG{q q^{m_{k-1,j} - m_{k-1,i}}x_j/tx_i}}\nonumber
\\ \quad
\times \prod_{k=2}^{N-1}\prod_{k\leq i < j \leq N }
\frac{\elG{q^{m_{k-2,j} - m_{k-1,i}}x_j/x_i}}{\elG{t q^{m_{i-2,j} - m_{k-1,i}}x_j/x_i}}
\frac{\elG{qq^{m_{k,j} - m_{k-1,i}}x_j/tx_i}}{\elG{q q^{m_{k,j} - m_{k-1,i}}x_j/x_i}}\nonumber
\\ \quad
=\prod_{k=2}^{N-1}\prod_{k\leq i < j\leq N}\!\!
\frac{\wg{tq^{m_{k-2,j}-m_{k-1,i}} x_j/x_i }{\theta_{k-1,j}}}
{\wg{q^{m_{k-2,j}-m_{k-1,i}} x_j/x_i }{\theta_{k-1,j}}}
\frac{\wg{qq^{m_{k-1,j}-m_{k-1,i}} x_j/tx_i }{\theta_{k,j}}}
{\wg{q q^{m_{k-1,j}-m_{k-1,i}} x_j/x_i }{\theta_{k,j}}}.\!\! \label{eq: elG trf}
\end{gather}
The exponent of $q$ in (\ref{eq: elG trf}) can be rewritten as
\begin{gather*}
m_{k-2,j}-m_{k-1,i}= -\theta_{k-1,i}+\sum_{a=1}^{k-2}(\theta_{a,j}-\theta_{a,i}),
\qquad
m_{k-1,j}-m_{k-1,i}= \sum_{a=1}^{k-1}(\theta_{a,j}-\theta_{a,i}).
\end{gather*}
Moreover, we have
\begin{gather*}
\prod_{k=2}^N \prod_{i=1}^{k-1} \big(t^{\delta_{i,k-1}}s_{i+1}/s_{i} \big)^{m_{i,k}}=
\prod_{1\leq i<j\leq N}t^{\theta_{i,j}} (s_j/s_i)^{\theta_{i,j}}.
\end{gather*}
By the calculation above,
we obtain
\begin{gather*}
 \tr \big( p^d \wcT^N(\vs;\vx) \big)=
\left(\frac{(pq/t;q,p)_{\infty}}{(p;p)_{\infty}(pt;q,p)_{\infty}} \right)^{N}
\prod_{1\leq i< j\leq N }
\frac{\elG{tx_j/x_i}}{\elG{qx_j/x_i}}
\\ \hphantom{\tr \big( p^d \wcT^N(\vs;\vx) \big)=}
{}\times\sum_{\theta \in \mathsf{M}_N}
\prod_{1\leq i< j\leq N }\frac{\wg{q/t}{\theta_{i,j}}}{\wg{q}{\theta_{i,j}}}
\prod_{1\leq i< j\leq N }
\frac{\wg{qx_j/tx_i}{\theta_{1,j}}}{\wg{qx_j/x_i}{\theta_{1,j}}}
\\ \hphantom{\tr \big( p^d \wcT^N(\vs;\vx) \big)=}
{}\times\prod_{k=2}^{N-1}\prod_{k\leq i < j\leq N}
\frac{\wg{tq^{-\theta_{k-1,i}+\sum_{a=1}^{k-2}(\theta_{a,j}-\theta_{a,i})} x_j/x_i }{\theta_{k-1,j}}}
{\wg{q^{-\theta_{k-1,i}+\sum_{a=1}^{k-2}(\theta_{a,j}-\theta_{a,i})} x_j/x_i }{\theta_{k-1,j}}}
\\ \hphantom{\tr \big( p^d \wcT^N(\vs;\vx) \big)=}
{}\times
\frac{\wg{qq^{\sum_{a=1}^{k-1}(\theta_{a,j}-\theta_{a,i})} x_j/tx_i }{\theta_{k,j}}}
{\wg{q q^{\sum_{a=1}^{k-1}(\theta_{a,j}-\theta_{a,i})} x_j/x_i }{\theta_{k,j}}}
\prod_{1\leq i<j\leq N}t^{\theta_{i,j}} (s_j/s_i)^{\theta_{i,j}}
\\ \hphantom{\tr \big( p^d \wcT^N(\vs;\vx) \big)}
{}=\left(\frac{(pq/t;q,p)_{\infty}}{(p;p)_{\infty}(pt;q,p)_{\infty}} \right)^{N}
\prod_{1\leq i< j\leq N }
\frac{\elG{tx_j/x_i}}{\elG{qx_j/x_i}}
\\ \hphantom{\tr \big( p^d \wcT^N(\vs;\vx) \big)=}
{}\times\sum_{\theta \in \mathsf{M}_N}
\prod_{k=1}^{N-1}\prod_{k< i \leq j\leq N}
\frac{\wg{tq^{-\theta_{k,i}+\sum_{a=1}^{k-1}(\theta_{a,j}-\theta_{a,i})} x_j/x_i }{\theta_{k,j}}}
{\wg{q^{-\theta_{k,i}+\sum_{a=1}^{k-1}(\theta_{a,j}-\theta_{a,i})} x_j/x_i }{\theta_{k,j}}}
\\ \hphantom{\tr \big( p^d \wcT^N(\vs;\vx) \big)=}
{}\times\prod_{k=1}^{N-1}\prod_{k\leq i < j\leq N}
\frac{\wg{qq^{\sum_{a=1}^{k-1}(\theta_{a,j}-\theta_{a,i})} x_j/tx_i }{\theta_{k,j}}}
{\wg{q q^{\sum_{a=1}^{k-1}(\theta_{a,j}-\theta_{a,i})} x_j/x_i }{\theta_{k,j}}}
\prod_{1\leq i<j\leq N}(s_j/s_i)^{\theta_{i,j}}.
\end{gather*}
Since $\theta_{i,j}$, $x_{i}$ and $s_i$
correspond to $\theta_{N-j+1,N-i+1}$,
$1/s_{N-i+1}$ and $1/x_{N-i+1}$
in Definition~\ref{def: ellip Mac}, respectively,
we have
\begin{gather*}
\tr \big( p^d \wcT^N(\vs;\vx) \big)=
\left(\frac{(pq/t;q,p)_{\infty}}{(p;p)_{\infty}(pt;q,p)_{\infty}} \right)^{N}
\prod_{1\leq i< j\leq N }
\frac{\elG{tx_j/x_i}}{\elG{qx_j/x_i}}
\\ \hphantom{\tr \big( p^d \wcT^N(\vs;\vx) \big)=}
{} \times
\fellip_N(1/s_N,\ldots,1/s_1;1/x_N,\ldots,1/x_1|q,q/t,p).
\end{gather*}
Therefore, Proposition~\ref{prop: trace wcT = fellip} follows from the symmetry
\begin{gather*}
\fellip_N(1/x_N,\ldots,1/x_1;1/s_N,\ldots,1/s_1|q,t,p)=
\fellip_N(x_1,\ldots,x_N;s_1,\ldots,s_N|q,t,p). \tag*{\qed}
\end{gather*}
\renewcommand{\qed}{}
\end{proof}

Now, we obtain a relationship between
the non-stationary Ruijsenaars functions and
the functions $\fellip$.

\begin{Theorem}\label{thm: fgln=fEG}
As formal series in $p$, $s_{i+1}/s_i$, $x_{i+1}/x_i$ ($i=1,\ldots, N-1$)
and $px_1/x_N$,
we obtain
\begin{gather*}
f^{\widehat{\mathfrak{gl}}_N}\big(\vx {}',p^{{1}/{N}}|\vs',t^{-{1}/{N}}|q,t\big)
= \mathfrak{C} \times \fellip_N(\boldsymbol{s};\boldsymbol{x}|q,t,p),
\end{gather*}
where we put
\begin{gather*}
\mathfrak{C}: =\left(\frac{(pq/t;q,p)_{\infty}}{(p;p)_{\infty}(pt;q,p)_{\infty}} \right)^{N}
\prod_{1\leq i<j\leq N} \frac{\Gamma(tx_j/x_i;q,p)}{\Gamma(qx_j/x_i;q,p)}
\prod_{1\leq i<j\leq N} \frac{(ts_j/s_i;q)_{\infty}}{(qs_j/s_i;q)_{\infty}}.
\end{gather*}
$\vx'$ and $\vs'$ are the same ones in Theorem~{\rm \ref{thm: Tr TH TH...}}:
\begin{gather*}
\vs'=(s'_1,\ldots ,s'_N),\qquad s'_k =t^{{k}/{N}}s_k,
\\
\vx'=(x'_1,\ldots ,x'_N),\qquad x'_k= p^{-{k}/{N}}x_k.
\end{gather*}
\end{Theorem}

For the proof,
we prepare the following lemma.

\begin{Lemma}\label{lem: Mukade VEV}
Let $1\leq k\leq N$.
The vacuum expectation values of the Mukad\'e operators are
\begin{gather*}
\bra{\boldsymbol{\emptyset}}
\tilTH_k \left( \vu, ;x\right)
\ket{\boldsymbol{\emptyset}}
=\prod_{1\leq i < k }
\frac{(qu_k/tu_i;q)_{\infty}}{(u_k/u_i;q)_{\infty}},
\qquad
\brazero \tilTV_k(\vu; x)\ketzero
=
\prod_{i=1}^{k-1} \frac{(qu_k/tu_i;q)_{\infty}}{(u_k/u_i;q)_{\infty}}.
\end{gather*}
\end{Lemma}

\begin{proof}
The first vacuum expectation value can be directly calculated
as
\begin{gather*}
\bra{\boldsymbol{\emptyset}}\tilTH_k (\vu, ;x)
\ket{\boldsymbol{\emptyset}}=\prod_{1\leq i <j\leq N}
\frac{\mathcal{G}(t^{-\delta_{j,k}+\delta_{i,k}} u_j/u_i)}
{\mathcal{G}(t^{\delta_{i,k}} u_j/u_i)}
\frac{\mathcal{G}(q u_j/tu_i)}
{\mathcal{G}(q t^{-\delta_{j,k}-1} u_j/u_i)}
\\ \hphantom{\bra{\boldsymbol{\emptyset}}\tilTH_k (\vu,;x)\ket{\boldsymbol{\emptyset}}}
{}=\prod_{1\leq i < k }\frac{\mathcal{G}( u_j/tu_i)}{\mathcal{G}(u_j/u_i)}
\frac{\mathcal{G}(q u_j/tu_i)}{\mathcal{G}(q u_j/t^2u_i)}
=\prod_{1\leq i < k }
\frac{(qu_k/tu_i;q)_{\infty}}{(u_k/u_i;q)_{\infty}}.
\end{gather*}
The second one can be calculated
by using the $q$-binomial theorem:
\begin{gather*}
\brazero \tilTV_k(\vu; x)\ketzero =\left( \frac{(q/t;q)_{\infty}}{(q;q)_{\infty}}\right)^{k-1}
\sum_{0\leq m_1\leq \cdots \leq m_{k-1}}
\frac{(qq^{m_1};q)_{\infty}}{(qq^{m_1}/t;q)_{\infty}}\prod_{i=2}^{k-1}
\frac{(qq^{m_i-m_{i-1}};q)_{\infty}}{(qq^{m_i-m_{i-1}}/t;q)_{\infty}}
\\ \hphantom{\brazero \tilTV_k(\vu; x)\ketzero =}
{}\times\prod_{i=1}^{k-1}\left(\frac{u_{i+1}}{u_i} \right)^{m_i}
=\sum_{n_1, \ldots, n_{k-1} \in \mathbb{Z}_{\geq 0}}
\prod_{i=1}^{k-1} \frac{(q/t;q)_{n_i}}{(q;q)_{n_i}}
\prod_{i=1}^{k-1} \left(\frac{u_{k}}{u_i} \right)^{n_i}
\\ \hphantom{\brazero \tilTV_k(\vu; x)\ketzero}
{}=\prod_{i=1}^{k-1} \frac{(qu_k/tu_i;q)_{\infty}}{(u_k/u_i;q)_{\infty}}. \tag*{\qed}
\end{gather*}
\renewcommand{\qed}{}
\end{proof}

\begin{proof}[Proof of Theorem~\ref{thm: fgln=fEG}]
In this proof,
we use the same notation as in
(\ref{eq: notation of s^i}).
For the sketch of the proof,
see Fig.~\ref{fig: proofsketch} in Introduction.
By virtue of Theorem~\ref{fact: chang pref. direc.},
we have
\begin{gather*}
\tr \big( p^d \wcT^N(\vs;\vx) \big)
=\sum_{\vm_0, \ldots, \vm_{N-1} \in \parset^N}p^{|\vm_0|}\prod_{i=1}^{N}
\bra{P_{\vm_{i-1}}} \tilTV_i\big(\vs^{(i-1)}; x_i\big) \ket{Q_{\vm_i} }
\\ \hphantom{\tr \big( p^d \wcT^N(\vs;\vx) \big)}
{}=\sum_{\vm_0, \ldots, \vm_{N-1} \in \parset^N}p^{|\vm_0|}
\prod_{i=1}^{N}\bra{\vm_{i}}\tilTH_i \big(\vs^{(i-1)}, ;x_i\big)
\ket{\vm_{i-1} }\prod_{i=1}^N
(-1)^{|\vm_{i}|+|\vm_{i-1}|}
\\ \hphantom{\tr \big( p^d \wcT^N(\vs;\vx) \big)= }
{}\times\prod_{i=1}^N
\frac{ \brazero \tilTV_i\big(\vs^{(i-1)}; x_i\big)\ketzero }
{ \bra{\boldsymbol{\emptyset}}\tilTH_i \big( \vs^{(i-1)}, ;x_i\big)
 \ket{\boldsymbol{\emptyset}} },
\end{gather*}
where $\vm_i=\big(\mu^{(i,1)}, \ldots, \mu^{(i,N)}\big) \in \parset^N$ and
$\vm_0=\vm_N$.
It is clear that
\begin{gather*}
\prod_{i=1}^N (-1)^{|\vm_{i}|+|\vm_{i-1}|}=1.
\end{gather*}
By Lemma~\ref{lem: Mukade VEV}, we have
\begin{gather*}
\prod_{i=1}^N
\frac{ \brazero \tilTV_i\big(\vs^{(i-1)}; x_i\big)\ketzero }
{ \bra{\boldsymbol{\emptyset}}
\tilTH_i \big(\vs^{(i-1)}, ;x_i\big)
 \ket{\boldsymbol{\emptyset}} } =1.
\end{gather*}
Therefore, it follows that
\begin{gather*}
\tr \big( p^d \wcT^N(\vs;\vx) \big)=
\sum_{\vm_0, \ldots, \vm_{N-1} \in \parset^N}
p^{|\vm_0|}\prod_{i=1}^{N}\bra{\vm_{i}}
\tilTH_i \big( \vs^{(i-1)} ;x_i\big)
\ket{\vm_{i-1} }
\\ \hphantom{\tr \big( p^d \wcT^N(\vs;\vx) \big)}
=\mathrm{Tr}_p
\big(\tilTH_N \big(\vs^{(N-1)} ;x_N\big)\cdots\tilTH_1 \big(\vs^{(0)} ;x_1\big)\big).
\end{gather*}
As a result, by Theorem~\ref{thm: Tr TH TH...}
(see also Remark~\ref{rem: T^H trace}), we obtain
\begin{gather}\label{eq: Tr TN in terms of fglN}
\tr \big( p^d \wcT^N(\vs;\vx) \big)=
\prod_{1\leq i<j\leq N} \frac{(qs_j/s_i;q)_{\infty}}{(ts_j/s_i;q)_{\infty}}\,
\nonstrui\big(\vx',p^{{1}/{N}}|\vs',t^{-{1}/{N}}|q,t\big).
\end{gather}
Combining Proposition~\ref{prop: trace wcT = fellip} and (\ref{eq: Tr TN in terms of fglN}) yields Theorem~\ref{thm: fgln=fEG}.
\end{proof}

\subsection{Another expression}

In the previous subsection,
we have established the relationship between
$\nonstrui$ and $\fellip$ by taking traces of intertwiners.
By changing the computation method to take the trace,
another expression can be obtained.
That is, we use the generalized Macdonald functions as a basis.
We first fix the normalization of
the generalized Macdonald functions $\ket{P_{\vl}}$,
which simplifies the matrix elements of the Mukad\'e operators.

\begin{Definition}
Define
\begin{gather*}
\mathcal{C}_{\vl}^{(+)}(\vu) := \xi_{\vl}^{(+)}(\vu) \times
\prod_{1\leq i<j \leq N} N_{\lambda^{(i)}, \lambda^{(j)}}(qu_i/tu_j)
\prod_{k=1}^N c_{\lambda^{(k)}}, 
\\
\mathcal{C}_{\vl} ^{(-)}(\vu) := \xi_{\vl}^{(-)}(\vu) \times
\prod_{1\leq i<j \leq N} N_{\lambda^{(j)}, \lambda^{(i)}}(qu_j/tu_i)
\prod_{k=1}^N c_{\lambda^{(k)}},
\\
\xi_{\vl}^{(+)}(\vu):=\prod_{i=1}^N(-1)^{(N-i+1)|\lambda^{(i)}|}
u_{i}^{(-N+i)|\lambda^{(i)}|+\sum_{k=1}^{i}|\lambda^{(k)}|}
\\ \hphantom{\xi_{\vl}^{(+)}(\vu):=}
{}\times \prod_{i=1}^N
(q/t)^{\left(\frac{1-i}{2}\right)|\lambda^{(i)}|}
q^{(i-N)(n(\lambda^{(i)'})+|\lambda^{(i)}|)}
t^{(N-i-1) (n(\lambda^{(i)})+|\lambda^{(i)}|)},
\\
\xi_{\vl}^{(-)}(\vu):=\prod_{i=1}^N(-1)^{i|\lambda^{(i)}|}
u_{i}^{(-i+1)|\lambda^{(i)}|+\sum_{k=i}^{N}|\lambda^{(k)}|}
\\ \hphantom{\xi_{\vl}^{(+)}(\vu) :=}
{}\times \prod_{i=1}^N
(q/t)^{\left(\frac{i-1}{2}\right)|\lambda^{(i)}|} t^{|\lambda^{(i)}|}
q^{(1-i)(n(\lambda^{(i)'})+|\lambda^{(i)}|)}
t^{(i-2) ( n(\lambda^{(i)})+|\lambda^{(i)}| )}.
\end{gather*}
\end{Definition}

\begin{Definition}
Define $\ket{K_{\vl}}=\ket{K_{\vl}(\vu)} \in \mathcal{F}_{\vu}$ and
$\bra{K_{\vl}}=\bra{K_{\vl}(\vu)} \in \mathcal{F}^*_{\vu}$ by
\begin{gather*}
\ket{K_{\vl}(\vu)}:=\mathcal{C}_{\vl}^{(+)}(\vu) \ket{P_{\vl}(\vu)}, \qquad
\bra{K_{\vl}(\vu)}:=\mathcal{C}_{\vl} ^{(-)}(\vu)\bra{P_{\vl}(\vu)}.
\end{gather*}
\end{Definition}

This normalization is based on our yet unfinished study of Conjecture 3.38 in \cite{FOS2019Generalized}.
Note, however, that we do not need the conjecture itself here.
\begin{fact}[\cite{FOS2019Generalized}]
We have
\begin{gather*}
 \braket{K_{\vl}(\vu)|K_{\vl}(\vu)}= \overline{\mathcal{M}}(\vu;\vl)
\prod_{i,j=1}^N N_{\lambda^{(i)}, \lambda^{(j)}}(qu_i/tu_j) ,
\end{gather*}
where
\begin{gather*}
\overline{\mathcal{M}}(\vu;\vl) = \big((-1)^{N}\gamma^2 e_N(\vu)\big)^{|\vl|}
\prod_{i=1}^N
\big(u_i^{|\lambda^{(i)}|}\gamma^{-2|\lambda^{(i)}|}g_{\lambda^{(i)}}\big)^{(2-N)}, \qquad
g_{\lambda}=q^{n(\lambda')}t^{-n(\lambda)},
\end{gather*}
with $e_{N}(\vu) = u_1\cdots u_N$.
\end{fact}

\begin{fact}[\cite{FOS2019Generalized}]
\label{fact: mat el. of TV}
We have
\begin{gather*}
\bra{K_{\vl}(\vv)} \mathcal{T}^V(\vu, \vv, x) \ket{K_{\vm}(\vu)}
=\mathcal{M}(\vu,\vv; \vl, \vm;x)
\prod_{i,j=1}^N N_{\lambda^{(i)},\mu^{(j)}}(v_i/\gamma u_j).
\end{gather*}
Here
\begin{gather*}
\mathcal{M}(\vu,\vv; \vl, \vm;x)=
\frac{\big((-\gamma)^N e_{N}(\vu)x \big)^{|\vl|}}
{(\gamma^2 x)^{|\vm|}}
\prod_{i=1}^N
\frac{u_i^{|\mu^{(i)}|}g_{\mu^{(i)}}}
{\big(v_i^{|\lambda^{(i)}|}g_{\lambda^{(i)}}\big)^{N-1}} .
\end{gather*}
\end{fact}

By this matrix element formula,
the trace of $\wcT^N(\vu;\vx)$ can be calculated as follows.

\begin{Proposition}\label{prop: tr TN by Nek}
It follows that
\begin{gather}
\tr \big( p^d \wcT^N(\vu;\vx) \big)=
\gamma^{-N^2}t^{-N}
\prod_{1\leq i<j\leq N} \frac{(qs_j/s_i;q)_{\infty}}{(ts_j/s_i;q)_{\infty}}\nonumber
\\ \hphantom{\tr \big( p^d \wcT^N(\vu;\vx) \big)=}
{}\times \sum_{\vl_0, \ldots, \vl_{N-1} \in \mathsf{P}^{N-1}}
\prod_{k=0}^{N-1} \prod_{i,j=1}^N
\frac{ N_{\lambda^{(k,i)},\lambda^{(k+1,j)}}(u_{k,i}/\gamma u_{k+1,j})}
{ N_{\lambda^{(k,i)},\lambda^{(k,j)}}(q u_{k,i}/ tu_{k,j})}
\, (p\gamma^{N} t)^{|\vl_0|} \nonumber
\\ \hphantom{\tr \big( p^d \wcT^N(\vu;\vx) \big)=}
{}\times\prod_{k=0}^{N-1}\big( x_{k+1}/\gamma^{2N} t x_k\big)^{|\vl_k|}.
\label{eq: tr TN by Nek}
\end{gather}
\end{Proposition}

\begin{proof}
By Fact~\ref{fact: mat el. of TV} and Lemma~\ref{lem: Mukade VEV},
we have
\begin{gather*}
\tr \big( p^d \wcT^N(\vu;\vx) \big)=
\sum_{\vl_0 } p^{|\vl_0|}
\frac{\bra{K_{\vl_0}\big(\vu^{(0)}\big)}\wcT^N(\vu;\vx)\ket{K_{\vl_0}\big(\vu^{(0)}\big)}}
{\langle K_{\vl_0}\big(\vu^{(0)}\big) \ket{K_{\vl_0} \big(\vu^{(0)}\big)}}
 \\ \hphantom{\tr \big( p^d \wcT^N(\vu;\vx) \big)}
{}=\sum_{\vl_0} p ^{|\vl_0|}
\frac{\xi^{(+)}_{\vl_0} \big(\vu^{(0)}\big)}{\xi^{(+)}_{\vl_0} \big(\vu^{(N)}\big)}
\frac{\bra{K_{\vl_0}\big(\vu^{(0)}\big)}\wcT^N(\vu;\vx)\ket{K_{\vl_0}\big(\vu^{(N)}\big)}}
{\langle K_{\vl_0}\big(\vu^{(0)}\big) \ket{K_{\vl_0} \big(\vu^{(0)}\big)}}
\\ \hphantom{\tr \big( p^d \wcT^N(\vu;\vx) \big)}
{}=\sum_{\vl_0, \ldots, \vl_{N-1}} \big(p \gamma^{N} t\big)^{|\vl_0|}
\prod_{k=0}^{N-1}
\frac{\bra{K_{\vl_k}\big(\vu^{(k)}\big)}\tilTV_{k+1}\big(\vu^{{(k)}};x_{k+1}\big)\ket{K_{\vl_{k+1}}\big(\vu^{(k+1)}\big)}}
{\langle K_{\vl_k}\big(\vu^{(k)}\big) \ket{K_{\vl_k} \big(\vu^{(k)}\big)}}
 \\ \hphantom{\tr \big( p^d \wcT^N(\vu;\vx) \big)}
{}=\sum_{\vl_0, \ldots, \vl_{N-1}} \big(p \gamma^{N} t\big)^{|\vl_0|}
\prod_{k=0}^{N-1}
\frac{\mathcal{M}\big(\vu^{(k+1)},\vu^{(k)}; \vl_k, \vl_{k+1};x_{k+1}\big)}{ \overline{\mathcal{M}}\big(\vu^{(k)};\vl_k\big)}
 \\ \hphantom{\tr \big( p^d \wcT^N(\vu;\vx) \big)=}
{}\times\prod_{1\leq i<j\leq N} \frac{(qs_j/s_i;q)_{\infty}}{(ts_j/s_i;q)_{\infty}}
\prod_{k=0}^{N-1} \prod_{i,j=1}^N
\frac{ N_{\lambda^{(k,i)},\lambda^{(k+1,j)}}(u_{k,i}/\gamma u_{k+1,j})}
{ N_{\lambda^{(k,i)},\lambda^{(k,j)}}(q u_{k,i}/ tu_{k,j})}.
\end{gather*}
Furthermore,
it can be shown that
\begin{gather*}
\prod_{k=0}^{N-1}
\frac{\mathcal{M}\big(\vu^{(k+1)},\vu^{(k)}; \vl_k, \vl_{k+1};x_{k+1}\big)}{ \overline{\mathcal{M}}\big(\vu^{(k)};\vl_k\big)}
=
\gamma^{-N^2}t^{-N}
\prod_{k=1}^{N}\big( x_{k+1}/\gamma^{2N} t x_k\big)^{|\vl_k|}.
\end{gather*}
Thus, 
we have (\ref{eq: tr TN by Nek}).
\end{proof}

\begin{Corollary}
We obtain
\begin{gather*}
f^{\widehat{\mathfrak{gl}}_N}\big(\vx,p^{{1}/{N}}|\vs',t^{-{1}/{N}}|q,t\big)
\\ \qquad
{}=\prod_{1\leq i<j\leq N}\frac{(t s_j/s_i;q)_{\infty}}{(qs_j/s_i;q)_{\infty}}\, \gamma^{-N^2}t^{-N}
\sum_{\vl_0, \ldots, \vl_{N-1} \in \parset^N}
\prod_{k=0}^{N-1} \prod_{i,j=1}^N
\frac{ N_{\lambda^{(k,i)},\lambda^{(k+1,j)}}(u_{k,i}/\gamma u_{k+1,j})}
{ N_{\lambda^{(k,i)},\lambda^{(k,j)}}(q u_{k,i}/ tu_{k,j})}
\\ \qquad \qquad
{}\times(t^3 \gamma^N)^{|\vl_0|}
\prod_{k=1}^{N}\big( p^{1/N} x_{N-k+1}/\gamma^{2N} q x_{N-k}\big)^{|\vl_k|}.
\end{gather*}
Here,
\begin{gather*}
\vs'=(s'_1,\ldots ,s'_N),\qquad
s'_k =t^{{k}/{N}}s_k.
\end{gather*}

\begin{proof}
Combine Proposition~\ref{prop: trace wcT = fellip},
Theorem~\ref{thm: fgln=fEG}
and Propositon \ref{prop: tr TN by Nek}.
\end{proof}
\end{Corollary}

\subsection[Case N=1]{Case $\boldsymbol{N=1}$}
\label{sec: N=1}

In this subsection,
we treat the case $N=1$.
This case is special in the sense that the $\kappa$ parameter is not specialized.
This is because in this case, the ratio of spectral parameters $v/u$ is the free parameter, and it becomes the $\kappa$ parameter.

\begin{Definition}
We put
\begin{gather*}
 \check{\Phi}^-(s)=\exp\bigg({-}\sum_n\frac{1-(1/\gamma \kappa)^n}{1-q^n}a_{-n} s^n \bigg) ,\qquad \check{\Phi}^+(s) = \exp\bigg(\sum_n\frac{1- (t\gamma \kappa/q)^n}{1-q^{-n}}a_n s^{-n} \bigg),
\end{gather*}
with $\kappa = qv/tu$
so that
\begin{gather*}
\Phicross
\bigg[ tu\kappa/q ;{ q s/t \kappa, \emptyset \atop s, \emptyset} ;u \bigg]
= \check{\Phi}^-(s)\check{\Phi}^+(s).
\end{gather*}
\end{Definition}

First, we take the trace in the horizontal direction. We obtain the next lemma.
\begin{Lemma}
We have
\begin{gather*}
\mathrm{tr}
\bigg(p^d
\Phicross
\bigg[ tu\kappa/q ;{ q s/t \kappa, \emptyset \atop s, \emptyset} ;u \bigg] \bigg)
= \exp\bigg(\sum \frac{1}{n}\frac{(1-q^n (\gamma \kappa)^n)(1-(\gamma \kappa)^n/t^n)(\gamma \kappa)^{-n}p^n}{(1-q^n)(1-t^{-n})(1-p^n)} \bigg) .
\end{gather*}
\begin{proof}
We note the formula for the normal ordering:
\begin{gather*}
\check{\Phi}^+\big(p^{-1} s\big)\check{\Phi}^-(s) =g(p){:}\check{\Phi}^+\big(p^{-1} s\big)\check{\Phi}^-(s){:} ,
\end{gather*}
where
\begin{gather*}
g(p) = \exp\bigg(\sum_n \frac{\big(1-(\gamma \kappa)^n\big)\big(1-q^n t^{-n}(\gamma \kappa)^{-n}\big)}{\big(1-t^n\big)\big(1-q^n\big)}p^n\bigg) .
\end{gather*}
By Lemma~\ref{lem: tr},
we can show that the given trace is
\begin{gather*}
\frac{1}{(p;p)_\infty}\prod_{k=1}^\infty g(p^{k})=
\exp\bigg(\sum \frac{1}{n}\frac{\big(1-q^n (\gamma \kappa)^n\big)\big(1-(\gamma \kappa)^n/t^n\big)(\gamma \kappa)^{-n}p^n}
{\big(1-q^n\big)\big(1-t^{-n}\big)\big(1-p^n\big)} \bigg) .\tag*{\qed}
\end{gather*}\renewcommand{\qed}{}
\end{proof}
\end{Lemma}

Next, we make the loop in the vertical direction. We obtain the following lemma.
\begin{Lemma}
\begin{gather*}
\mathrm{tr}\bigg(p^d
\Phicross
\bigg[ tu\kappa/q ;{ q s/t \kappa, \emptyset \atop s, \emptyset} ;u \bigg] \bigg)
=
\sum_{\lambda} p^{|\lambda|} (\gamma\kappa)^{-|\lambda|} \frac{N_{\lambda\lambda}(\gamma \kappa)}{ N_{\lambda\lambda}(1)} .
\end{gather*}
\begin{proof}
Use the $S$-duality formula
\begin{gather*}
\mathrm{tr}\bigg(p^d \Phicross
\bigg[ tu\kappa/q ;{ q s/t \kappa, \emptyset \atop s, \emptyset} ;u \bigg] \bigg)
=\sum_{\lambda}p^{|\lambda|}\bra{P_{\lambda}}\Phicross
\bigg[ tu\kappa/q ;{ q s/t \kappa, \emptyset \atop s, \emptyset} ;u \bigg] \ket{Q_{\lambda}}
\\ \phantom{\mathrm{tr}\bigg(p^d \Phicross
\bigg[ tu\kappa/q ;{ q s/t \kappa, \emptyset \atop s, \emptyset} ;u \bigg] \bigg)}
= \sum_{\lambda}p^{|\lambda|}\bra{0}\Phicross
\bigg[s ;{ u ,\; \lambda \atop tu\kappa /q, \lambda} ;qs/t\kappa \bigg] \ket{0} ,
\end{gather*}
and take the normal ordering (Fact~\ref{fact: intertwiner OPE}).
\end{proof}
\end{Lemma}

Combining these two lemmas results in the following summation formula.
\begin{Theorem}\label{thm: N=1}
We have
\begin{gather*}
\exp\bigg(\sum \frac{1}{n}\frac{(1-q^n \kappa^n)(1-\kappa^n/t^n)\kappa^{-n}p^n}{(1-q^n)(1-t^{-n})(1-p^n)} \bigg)
 \\
 \qquad = \sum_{\lambda \in \mathsf{P}} (p/\kappa)^{|\lambda|} \frac{\prod_{1 \leq i\leq j }
\big(\kappa q^{-\lambda_i+\lambda_{j+1}} t^{i-j};q\big)_{\lambda_j-\lambda_{j+1}}
\big(\kappa q^{\lambda_i-\lambda_j} t^{-i+j+1};q\big)_{\lambda_{j}-\lambda_{j+1}}}{\prod_{1 \leq i\leq j }
\big(q^{-\lambda_i+\lambda_{j+1}} t^{i-j};q\big)_{\lambda_j-\lambda_{j+1}}
\big(q^{\lambda_i-\lambda_j} t^{-i+j+1};q\big)_{\lambda_{j}-\lambda_{j+1}}} .
\end{gather*}
\end{Theorem}

This gives the proof of the conjecture in \cite{HR2008mixed}, which claims the two different forms of the mixed Hodge polynomials of certain twisted $\mathrm{GL}(n,\mathbb{C})$-character varieties of Riemann surfaces with $g=1$.
The identity was proposed also in \cite{AK2009Changing} motivated by the S-duality conjecture in the string theory.
The similar proof is given in \cite{CNO2013Five, RW2018nekrasov}.
Physically, this relates the partition function of the 5d $\mathcal{N} = 1^\ast$ $U(1)$ gauge theory to that of the 6d theory with one tensor multiplet.

\section{Integral operators}\label{sec: int op.}

\subsection{Integral operator of Macdonald functions}
We return to the non-affine case with $N \geq 2$.
In Fact~\ref{fact: macdonald from mukade},
the ordinary Macdonald functions were constructed
from the screened vertex operators.
In this section,
an integral operator introduced in
\cite{shiraishi2005commutative1,shiraishi2005commutative2,Shiraishi2006family}
will be constructed from them.
We treat the spectral parameters $\vs=(s_1,\ldots,s_N)$ as generic complex variables in this section.
First, we rewrite the screened vertex operators (non-affine case)
by the contour integrals.

\begin{Definition}
For $k=0,\ldots, N-1$, define
$\phicont_{k,N}(x)=\phicont_{k,N}(\vs;x)\colon
\cF^{(N,0)}_{\gamma^{-1} t^{-\delta_{k+1}}\cdot \vs}\rightarrow \cF^{(N,0)}_{\vs}$
by
\begin{gather*}
\phicont_{k,N}(x):=K(\vs)
\oint_C \prod_{i=1}^k \frac{{\rm d}y_i}{2 \pi \sqrt{-1}y_i}
\phi_0(x) \widetilde{S}^{(1)}(y_1)\cdots \widetilde{S}^{(k)}(y_k) g(\vs;x,y_1,\ldots,y_k) ,
\\
K(\vs):=\prod_{i=1}^k \frac{(q;q)_{\infty} (q/t;q)_{\infty}}
{(\frac{q s_i}{s_{k+1}};q)_{\infty}(\frac{q s_{k+1}}{ts_{i}};q)_{\infty}} ,
\\
g(\vs;x,y_1,\ldots,y_k):=\frac{\theta_q(qs_1 y_1/s_{k+1}x)}{\theta_q(q y_1/x)}
\prod_{i=1}^{k-1}\frac{\theta_q(qs_{i+1} y_{i+1}/s_{k+1}y_i)}{\theta_q(q y_{i+1}/y_i)} .
\end{gather*}
Here, the contour of the integration $C$ is taken such that
$|q/t|<|x/y_1|<|q|$ and $|t^{-1}|<y_i/y_{i+1}|<1$ for $i=1,\ldots,k$.%
\footnote{
This screened vertex operator corresponds to $\Phi^{(k)}(x)$
in \cite{FOS2019Generalized} after
transformation $x \rightarrow t^{-1}x$ and $y_i \rightarrow (q/t)^i t^{-1} y_i$
and modification of the integration contour.
Actually, a more strict condition is imposed for inte\-gration contour in \cite{FOS2019Generalized}
in order to show that the screening currents commute with $X^{(r)}(z)$
($r=1,\ldots, N$).
However, only commutativity with $\Xo(z)$ is required
to show Fact~\ref{fact: mac from phicont}.
Hence we adopt this integration contour in this paper.
}
\end{Definition}
\begin{Remark}\label{rem. phicont argument}
The spectral parameter $\gamma^{-1} t^{-\delta_{k+1}}\cdot \vs$ in the domain of $\phicont$ is determined by the spectral parameter $\vs$ in the codomain.
Though all $\phicont_{k,N}(x)$ depend on the parameter $\vs$,
we omit it in the argument
if spectral parameters of the domain and the codomain are
automatically determined such as the composition of the operators:
\begin{gather*}
\phicont_{1,N}(\vs, x)\,\phicont_{2,N}(x)\,\phicont_{3,N}(x)\cdots .
\end{gather*}
\end{Remark}

This operator can be expanded as follows.

\begin{Proposition}\label{prop: Raman expansion}
We have
\begin{gather*}
\phicont_{k,N}(x)\!=\!\oint_C \!\prod_{i=1}^k \frac{{\rm d}y_i}{2 \pi \sqrt{-1}y_i}
\!\!\sum_{r_1,\ldots, r_k \in \mathbb{Z}} \prod_{i=1}^k
\frac{(ts_i/s_{k+1};q)_{r_i} }{(qs_i/s_{k+1};q)_{r_i}}
\left(\frac{qy_i}{ty_{i-1}}\right)^{r_i}\!\!
{:}\phi_0(x) \widetilde{S}^{(1)}(y_1)\cdots \widetilde{S}^{(k)}(y_k){:}.
\end{gather*}

\begin{proof}
It follows from the operator product formulas (Proposition~\ref{prop: phi0 S OPE})
and Ramanujan's ${}_1\psi_1$ summation formula
((5.2.1) in \cite{GR2004basic}):
\begin{gather*}
{}_1\psi_1(a;b;q;z):=
\sum_{n=-\infty}^{\infty} \frac{(a;q)_{n}}{(b;q)_{n}}z^n=
\frac{(q;q)_{\infty}(b/a;q)_{\infty}(az;q)_{\infty}(q/az;q)_{\infty}}
{(b;q)_{\infty}(q/a;q)_{\infty}(z;q)_{\infty}(b/az;q)_{\infty}}
\\ \hphantom{{}_1\psi_1(a;b;q;z):=}
(|b/a|<z<1).\tag*{\qed}
\end{gather*}
\renewcommand{\qed}{}
\end{proof}
\end{Proposition}
%
\begin{Proposition}
$\phicont_k(x)$ is given by $\tilTV$ as
\begin{gather*}
\phicont_{k,N}(\vs, x)=
\prod_{i=1}^k \frac{(s_{k+1}/s_i;q)_{\infty}}{( qs_{k+1}/ts_i;q)_{\infty}} \,
\tilTV_{k+1}(\vs; x).
\end{gather*}
\begin{proof}
First,
we adjust the contour of the integration
to the condition $|q|<|y_{i-1}/y_{i}|<1$
($i=1,\ldots, k$). Here, we put $y_0=x$.
Note that no pole affects this change.
Then we have
\begin{gather*}
g(\vs;x,y_1,\ldots,y_k)
=\prod_{i=0}^{k-1}
\Bigg(
\frac{(\frac{s_{k+1}}{s_{i+1}};q)_{\infty} (\frac{qs_{i+1}}{s_{k+1}};q)_{\infty}}
{(q;q)_{\infty}(q;q)_{\infty}}
\sum_{n \in \mathbb{Z}}\frac{1}{1-q^n \frac{s_{k+1}}{s_{i+1}}}(y_i/y_{i+1})^n\Bigg).
\end{gather*}
By the deformation of the formal series
\begin{gather}
\sum_{n \in \mathbb{Z}}\frac{1}{1-q^n \frac{s_{k+1}}{s_{i+1}}}(y_i/y_{i+1})^n
=
\sum_{n \in \mathbb{Z}}\sum_{m\geq 0}\left(q^n \frac{s_{k+1}}{s_{i+1}}\right)^m
(y_i/y_{i+1})^n \nonumber
\\ \hphantom{\sum_{n \in \mathbb{Z}}\frac{1}{1-q^n \frac{s_{k+1}}{s_{i+1}}}(y_i/y_{i+1})^n}
{}=\sum_{m\geq 0}\delta(q^m y_i/y_{i+1}) \left(\frac{s_{k+1}}{s_{i+1}}\right)^m,
\label{eq: formal eq}
\end{gather}
we have
\begin{gather}
\phicont_{k,N}(x)=K(\vs)\prod_{i=0}^{k-1}
\frac{\big(\frac{s_{k+1}}{s_{i+1}};q\big)_{\infty} \big(\frac{qs_{i+1}}{s_{k+1}};q\big)_{\infty}}
{(q;q)_{\infty}(q;q)_{\infty}} \sum_{m_0, \ldots, m_{k-1} \geq 0}
\oint_C \prod_{i=1}^k \frac{{\rm d}y_i}{2 \pi \sqrt{-1}y_i}\phi_0(x) \nonumber
\\ \hphantom{\phicont_{k,N}(x)=}
{}\times\prod_{1\leq i \leq k}^{\curvearrowright}
\widetilde{S}^{(i)}(q^{m_0+\cdots +m_{i-1}}x)
 \prod_{i=0}^{k-1} \delta\left(\frac{q^{m_i} y_i}{y_{i+1}}\right)
\left(\frac{s_{k+1}}{s_{i+1}}\right)^{m_i}\nonumber
\\ \hphantom{\phicont_{k,N}(x)}
{}=K(\vs)\prod_{i=0}^{k-1}
\frac{\big(\frac{s_{k+1}}{s_{i+1}};q\big)_{\infty} \big(\frac{qs_{i+1}}{s_{k+1}};q\big)_{\infty}}
{(q;q)_{\infty}(q;q)_{\infty}} \!\!\!
\sum_{0\leq \ell_1 \leq \ell _2 \leq \cdots \leq \ell_{k}}\!\!\!
\phi_0(x)\prod_{1\leq i \leq k}^{\curvearrowright}
\widetilde{S}^{(i)}(q^{\ell_{i}}x)
 \prod_{i=1}^{k-1} \left(\frac{s_{i+1}}{s_{i}}\right)^{\ell_i} \nonumber
\\ \hphantom{\phicont_{k,N}(x)}
{}=\prod_{i=1}^k \frac{(s_{k+1}/s_i;q)_{\infty}}{( qs_{k+1}/ts_i;q)_{\infty}} \,
\tilTV_{k+1}(\vs; x).
\label{eq: phicont and TV}
\end{gather}
The deformation (\ref{eq: formal eq}) itself is not well-defined
because $\sum_{m \geq 0} (q^n \frac{s_{k+1}}{s_{i+1}})^m$
does not converge for arbitrary $n \in \mathbb{Z}$.
However,
considering the matrix elements
of the operators,
we can justify the calculation (\ref{eq: phicont and TV}).
For more detail, see Remark A.2 in \cite{FOS2019Generalized}.
\end{proof}
\end{Proposition}

Fact~\ref{fact: macdonald from mukade} can be rewritten as follows.

\begin{fact}[\cite{FOS2019Generalized}, Theorem 3.26]
\label{fact: mac from phicont}
It follows that
\begin{gather*}
\brazero \phicont_{0,N}(\vs;x_1)\,\phicont_{1,N}(x_2)\cdots \phicont_{N-1,N}(x_{N}) \ketzero = f^{\mathfrak{gl}_N}(\vx, \vs|q, q/t).
\end{gather*}
\end{fact}

We introduce the following integral operator,
which is essentially the same as the one in \cite{Shiraishi2006family}.

\begin{Definition}\label{def: operator I}
Define the integral operator $I(s_1/s_0,\ldots,s_N/s_0)$ on
$\mathbb{C}[[x_2/x_1,\ldots, x_N/x_{N-1}]]$ by
\begin{gather*}
I(s_1/s_0,\ldots,s_N/s_0)(f(x_1,\ldots , x_N))
\\ \qquad
{}=\mathcal{K}(s)
\prod_{1\leq i < j \leq N} \frac{(qx_j/t x_i ;q)_{\infty} }{(tx_j/x_i ;q)_{\infty}}
\oint_{C'} \prod_{i=1}^N \frac{{\rm d}y_{i}}{2 \pi \sqrt{-1}y_{i}}\,
\Pi^{(q,t)}(x|y)
\prod_{i=1}^{N} \frac{\theta_q(qs_0 y_i/s_{i}x_i)}{\theta_q(q y_i/x_i)}
\\ \phantom{=} \qquad
\times\prod_{1\leq i < j \leq N} \frac{\theta_q(tx_j/y_i)}{\theta_q(x_j/y_i)}
\prod_{1\leq i < j \leq N} (1-y_j/y_i) \, f(y_1,\ldots , y_N).
\end{gather*}
Here, we put
\begin{gather*}
\mathcal{K}(s)=\mathcal{K}(s_1/s_0,\ldots,s_N/s_0)=\prod_{i=1}^k \frac{(q;q)_{\infty} (q/t;q)_{\infty}}
{(q s_0/s_{i};q)_{\infty}(q s_{i}/ts_{0};q)_{\infty}}.
\end{gather*}
$\Pi^{(q,t)}(x|y)$ is the kernel function:
\begin{gather*}
\Pi^{(q,t)}(x|y)=\prod_{i,j=1}^{N}\frac{(qy_j/x_i ;q)_{\infty} }{(qy_j/tx_i ;q)_{\infty}}.
\end{gather*}
{\sloppy
We chose the integration contour $C'$
so that $|y_{k+1}/qy_k|<|t^{-1}|$ ($k=1,\ldots, N-1$) and
$|t^{-1}|<|x_i/y_i|<|q|$ ($i=1,\ldots, N$),
regarding the variables $x_i$'s
as complex variables satisfying $|x_{k+1}/x_k|<|t^{-1}|$
($k=1,\ldots,N-1$).

}
\end{Definition}

\begin{Remark}
In what follows, we assume $|t^{-1}|<|q|$ so that
the integration contour is well-defined.
\end{Remark}

Consider the $N+1$ fold Fock tensor spaces
\begin{gather*}
\cF^{(N+1,0)}_{s_0, \ldots, s_N}=\cF^{(1,0)}_{s_0}\otimes \cF^{(N,0)}_{s_1,\ldots,s_N}
\end{gather*}
and naturally extend the screened vertex operators $\phicont_{k,N}$'s
to this space.
Then
we can construct the Macdonald functions from $\cF^{(N,0)}_{s_1,\ldots.s_N}$
and reproduce the integral operator from the additional Fock space $\cF^{(1,0)}_{s_0}$.
That is to say,
the matrix element in the following proposition
can be viewed as the action of $I$ on Macdonald functions.
See also Fig.~\ref{fig: integral_op}.

\begin{figure}[h]
\centering
 \includegraphics[scale=1.0]{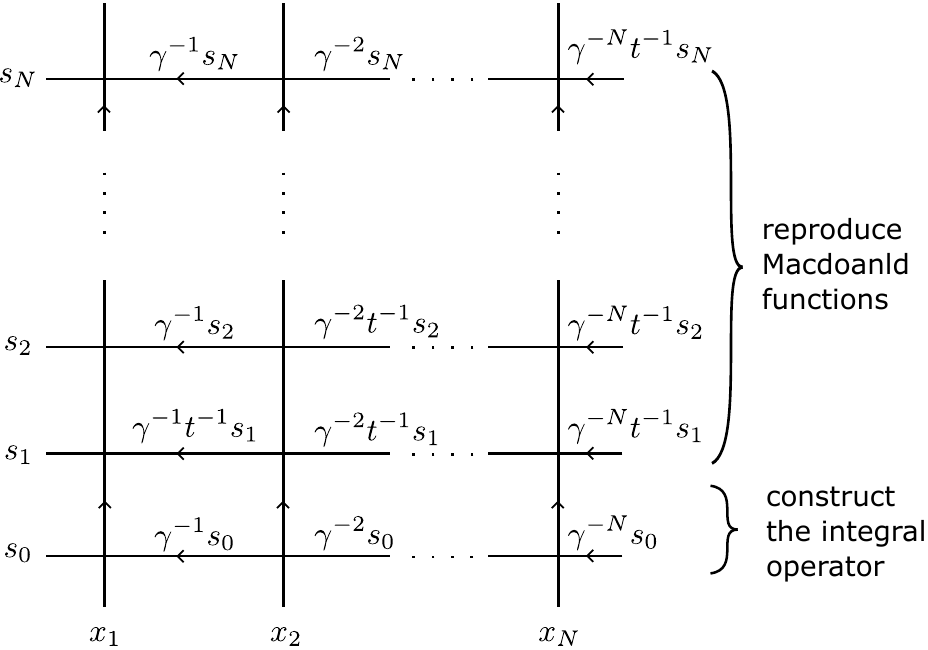}
	\caption{Operators in Proposition~\ref{prop: reproduce I}.}
	\label{fig: integral_op}
\end{figure}

\begin{Proposition}\label{prop: reproduce I}
Let $\vs^+=(s_0,s_1,\ldots, s_N)$.
Then we have
\begin{gather*}
\brazero \phicont_{1,N+1}(\vs^+;x_1) \ldots \phicont_{N,N+1}(x_N) \ketzero
 \\ \qquad
{}=I(s_1/s_0,\ldots,s_N/s_0)\big(f^{\mathfrak{gl}_N}((x_1,\ldots,x_N);(s_1,\ldots,s_N)|q,q/t )\big).
\end{gather*}
\end{Proposition}

\begin{proof}
In this proof,
we put $s_{i,j}:=\gamma^{-i+1}t^{-\delta_{i>j}} s_j$.
Here $\delta_{a>b}$ is $1$ if $a>b$ or $0$ if $a\leq b$.
By~tak\-ing the normal ordering, we have
\begin{gather*}
\brazero \phicont_{1,N+1}(\vs^+;x_1) \cdots \phicont_{N,N+1}(x_N) \ketzero
\\ \qquad
{}=\prod_{0\leq i < j\leq N}^{N}\frac{(q;q)_{\infty} (q/t;q)_{\infty}}
{(\frac{q s_{j,i}}{s_{j,j}};q)_{\infty}(\frac{q s_{j,j}}{ts_{j,i}};q)_{\infty}}
\prod_{1\leq i <j\leq N}
\frac{(qx_{j}/t x_{i} ;q)_{\infty}}{(tx_{j}/x_{i} ;q)_{\infty}}
 \oint \prod_{i=1}^N \frac{{\rm d}y_{i,1}}{2\pi \sqrt{-1}y_{i,1}}
\\ \qquad\phantom{=}
{}\times
\prod_{1\leq i <j\leq N}
\frac{(qy_{j,1}/x_{i} ;q)_{\infty}}{(qy_{j,1}/tx_{i} ;q)_{\infty}}
\frac{(tx_{j}/y_{i,1} ;q)_{\infty}}{(x_{j}/y_{i,1} ;q)_{\infty}}\prod_{1\leq i <j\leq N}
\bigg(1-\frac{y_{j,1}}{y_{i,1}}\bigg)
\frac{(qy_{j,1}/ty_{i,1} ;q)_{\infty}}{(ty_{j,1}/y_{i,1} ;q)_{\infty}}
 \\ \qquad\phantom{=}
{}\times \oint \prod_{2\leq i<j\leq N}
\frac{{\rm d}y_{j,i}}{2\pi \sqrt{-1}y_{j,i}}\prod_{1\leq i \leq j\leq N}
\frac{(qy_{j,2}/y_{i,1} ;q)_{\infty}}{(qy_{j,2}/ty_{i,1} ;q)_{\infty}}
 \prod_{2\leq i < j\leq N}
\frac{(ty_{j,1}/y_{i,2} ;q)_{\infty}}{(y_{j,1}/y_{i,2} ;q)_{\infty}}
 \\ \qquad\phantom{=}
{}\times \brazero S_2(y_{2,2})S_3(y_{3,2})S_3(y_{3,3}) \cdots S_N(y_{N,2})\cdots
S_N(y_{N,N})\ketzero
 \\ \qquad\phantom{=}
{}\times \prod_{1\leq i < j\leq N}
\frac{\theta_q\big(q\tfrac{s_{j,i}}{s_{j,j}} \tfrac{y_{j,i+1}}{y_{j,i}}\big)}
{\theta_q\big(q \tfrac{y_{j,i+1}}{y_{j,i}}\big)} \prod_{i=1}^N
\frac{\theta_q\big(q\tfrac{s_{i,0}}{s_{i,i}} \tfrac{y_{i,1}}{x_{i}}\big)}
{\theta_q\big(q \tfrac{y_{i,1}}{x_{i}}\big)}.
\end{gather*}
As in Fact~\ref{fact: mac from phicont},
we can rewrite the inner integrals by Macdonald functions,
and we can show that
\begin{gather*}
\prod_{1\leq i < j\leq N}^{N}
\frac{(q;q)_{\infty} (q/t;q)_{\infty}}
{\big(\frac{q s_{j,i}}{s_{j,j}};q\big)_{\infty}\big(\frac{q s_{j,j}}{ts_{j,i}};q\big)_{\infty}}
\prod_{1\leq i <j\leq N}
\frac{(qy_{j,1}/ty_{i,k} ;q)_{\infty}}{(qy_{j,1}/y_{i,k} ;q)_{\infty}}\oint \prod_{2\leq i<j\leq N}
\frac{{\rm d}y_{j,i}}{2\pi \sqrt{-1}y_{j,i}}
 \\ \qquad\phantom{=}
{}\times
\prod_{1\leq i \leq j\leq N}
\frac{(qy_{j,2}/y_{i,1} ;q)_{\infty}}{(qy_{j,2}/ty_{i,1} ;q)_{\infty}}
 \prod_{2\leq i < j\leq N}
\frac{(ty_{j,1}/y_{i,2} ;q)_{\infty}}{(y_{j,1}/y_{i,2} ;q)_{\infty}}
 \\ \qquad\phantom{=}
{}\times \brazero S_2(y_{2,2})S_3(y_{3,2})S_3(y_{3,3}) \cdots S_N(y_{N,2})\cdots
S_N(y_{N,N})\ketzero
 \prod_{1\leq i < j\leq N}\frac{\theta_q\big(q\tfrac{s_{j,i}}{s_{j,j}} \tfrac{y_{j,i+1}}{y_{j,i}}\big)}
{\theta_q\big(q \tfrac{y_{j,i+1}}{y_{j,i}} \big)}
 \\ \qquad
{}=\ordmac((y_{1,1},y_{2,1},\ldots , y_{N,1});(s_1,\ldots, s_N)|q,q/t).
\end{gather*}
Therefore, Proposition~\ref{prop: reproduce I} follows.
\end{proof}

We prove the commutativity
between the integral operator $I$ and Macdonald's difference operator.
For the proof,
we need to take care of the analyticity of the domain.
Hence,
let us define the following region.

\begin{notation}\label{not: region}
Define the projection $\pi\colon (\mathbb{C}^*)^N \rightarrow \mathbb{C}^{N-1}$ by
\begin{gather*}
\pi (a_1,\ldots, a_N)=(a_2/a_1,\ldots, a_N/a_{N-1}).
\end{gather*}
Set
\begin{gather*}
U^N_r:=\big\{(z_1,\ldots, z_{N-1}) \in (\mathbb{C})^{N-1}\,|\, z_i<r,\ i=1,\ldots, N-1 \big\},
\\
B^N_r:=\big\{(z_1,\ldots, z_{N-1}) \in (\mathbb{C})^{N-1}\,|\, z_i\leq r,\ i=1,\ldots, N-1\big \}.
\end{gather*}
so that
\begin{gather*}
\pi^{-1} \big(U^N_r\big)
=\big\{(x_1,\ldots, x_N) \in (\mathbb{C}^{*})^N;\ x_{j}/x_i<r^{j-i},\ 1\leq i <j\leq N\big \}.
\end{gather*}
Define the subset $\mathcal{O}(U^N_r) \subset
\mathbb{C}[[x_2/x_1,\ldots,x_N/x_{N-1}]]$
to be the set of power series
which absolutely convergent on the $B^N_{r'}$ for any $r'<r$, i.e., which can be regarded as a holomorphic function
on $U^N_r$.
\end{notation}

\begin{Theorem}\label{thm: [D,I]=0}
On $\mathcal{O}\big(U^N_{|t^{-1}|}\big)$,
we obtain
\begin{gather*}
[ D_x(s_1,\ldots,s_N;q,q/t),\ I(s_1/s_0,\ldots,s_N/s_0)]=0.
\end{gather*}
Here, $D_x(\vs;q,t)$ is the Macdonald $q$-difference operator:
\begin{gather*}
D_x(\vs;q,t):= \sum_{k=1}^N s_k
\prod_{1\leq \ell < k} \frac{1-tx_k/x_{\ell}}{1-x_k/x_{\ell}}
\prod_{k < \ell \leq n} \frac{1-x_{\ell}/tx_k}{1-x_{\ell}/x_{k}}T_{q,x_k}.
\end{gather*}
\end{Theorem}

As for the relation between $D_x(\vs;q,t)$ and ordinary Macdonald's difference operator,
see Remark~\ref{rem: ord Mac and assym Mac}.
In the proof,
we use the following fact.

\begin{fact}[\cite{Macdonald2015Symmetric}]
\label{fact: kernel fn rel (Mac)}
It follows that
\begin{gather*}
D_x\big(1,t^{-1},\ldots, t^{-N+1};q,t\big) \Pi^{(q,t)}(x|y)=
D_{y^{-1}}\big(1,t^{-1},\ldots, t^{-N+1};q,t\big) \Pi^{(q,t)}(x|y).
\end{gather*}
Here, we put
\begin{gather*}
D_{y^{-1}}(\boldsymbol{s};q,t):= \sum_{k=1}^n
s_k
\prod_{1\leq \ell < k} \frac{1-ty_{\ell}/y_{k}}{1-y_{\ell}/y_{k}}
\prod_{k < \ell \leq n} \frac{1-y_{k}/ty_{\ell}}{1-y_{k}/y_{\ell}}
T_{q^{-1},y_k}.
\end{gather*}
\end{fact}

\begin{proof}[Proof of Proposition~\ref{thm: [D,I]=0}]
A direct calculation gives
\begin{gather*}
D_x(\vs;q,q/t)
\prod_{1\leq i < j \leq N} \frac{(qx_j/t x_i ;q)_{\infty} }{(tx_j/x_i ;q)_{\infty}}=
\prod_{1\leq i < j \leq N} \frac{(qx_j/t x_i ;q)_{\infty} }{(tx_j/x_i ;q)_{\infty}}
D_x(\vs;q,t),
\end{gather*}
and
\begin{gather*}
D_x(\vs;q,t)\prod_{i=1}^{N} \frac{\theta_q(qs_0 y_i/s_{i}x_i)}{\theta_q(q y_i/x_i)}
\prod_{1\leq i < j \leq N} \frac{\theta_q(tx_j/y_i)}{\theta_q(x_j/y_i)}
\\ \phantom{D_x(\vs;)}
{}=\prod_{i=1}^{N} \frac{\theta_q(qs_0 y_i/s_{i}x_i)}{\theta_q(q y_i/x_i)}
\prod_{1\leq i < j \leq N} \frac{\theta_q(tx_j/y_i)}{\theta_q(x_j/y_i)}
s_0 D_x\big(\big(1, t^{-1},\ldots, t^{-N+1}\big);q,t\big).
\end{gather*}
By these equations and
Fact~\ref{fact: kernel fn rel (Mac)},
we can show that for a function $f(x_1,\ldots ,x_N) \in \mathcal{O}\big(U^N_{|t^{-1}|}\big)$,
\begin{gather}
D_x(\vs;q,q/t) I(s_1/s_0,\ldots,s_N/s_0)f(x_1,\ldots ,x_N)\nonumber
\\ \qquad
{}= \mathcal{K}(s)
\prod_{1\leq i < j \leq N} \frac{(qx_j/t x_i ;q)_{\infty} }{(tx_j/x_i ;q)_{\infty}}\oint \prod_{i=1}^N \frac{{\rm d}y_{i}}{2 \pi \sqrt{-1}y_{i}}
\prod_{i=1}^{N} \frac{\theta_q(qs_0 y_i/s_{i}x_i)}{\theta_q(q y_i/x_i)}\nonumber
\\ \qquad \qquad\
{}\times\prod_{1\leq i < j \leq N} \frac{\theta_q(tx_j/y_i)}{\theta_q(x_j/y_i)}
\prod_{1\leq i < j \leq N} (1-y_j/y_i)\, f(y_1,\ldots, y_N)s_0 \nonumber
\\ \qquad \qquad\
{}\times D_{y^{-1}}(1,\ldots, t^{-N+1};q,t)\Pi^{(q,t)}(x|y) \nonumber
\\ \qquad
{}= \mathcal{K}(s)
\prod_{1\leq i < j \leq N} \frac{(qx_j/t x_i ;q)_{\infty} }{(tx_j/x_i ;q)_{\infty}}
\oint \prod_{i=1}^N \frac{{\rm d}y_{i}}{2 \pi \sqrt{-1}y_{i}}
\prod_{i=1}^{N} \frac{\theta_q(qs_0 y_i/s_{i}x_i)}{\theta_q(q y_i/x_i)} \nonumber
\\ \qquad \qquad\
{}\times\prod_{1\leq i < j \leq N} \frac{\theta_q(tx_j/y_i)}{\theta_q(x_j/y_i)}
\prod_{1\leq i < j \leq N} (1-y_j/y_i)\, f(y_1,\ldots, y_N)s_0 \nonumber
\\ \qquad \qquad\
{}\times\sum_{k=1}^N\prod_{ \ell \neq k} \frac{1-y_{k}/ty_{\ell}}{1-y_{k}/y_{\ell}}\,
T_{q^{-1},y_k}\Pi^{(q,t)}(x|y).
\label{eq: cal D I}
\end{gather}
We have poles in $y_k$
of the each term containing the difference operator $T_{q^{-1},y_k}$
at $y_k=q^{a}x_k$, $y_k=q^{-a}tx_k$, $y_k=q^{a+1}x_{j}$,
$y_k=q^{-a}tx_i$ ($a=0,1,2,\ldots$, $i<k<j$).
Therefore, by the change of variable $y_k \rightarrow q y_k$
in the each term
(note that there is no pole between $y_k$ and $q y_k$),
we can show that (\ref{eq: cal D I}) is equal to
\begin{gather*}
\mathcal{K}(s)\prod_{1\leq i < j \leq N}
\frac{(qx_j/t x_i ;q)_{\infty} }{(tx_j/x_i ;q)_{\infty}}
\oint \prod_{i=1}^N\frac{{\rm d}y_{i}}{2 \pi \sqrt{-1}y_{i}}\,
\Pi^{(q,t)}(x|y)
\\ \qquad
{}\times \sum_{k=1}^N t^{-k+1}T_{q,y_k}
\prod_{1\leq \ell < k} \frac{1-ty_{\ell}/y_{k}}{1-y_{\ell}/y_{k}}
\prod_{k < \ell \leq n} \frac{1-y_{k}/ty_{\ell}}{1-y_{k}/y_{\ell}}
\prod_{i=1}^{N} \frac{\theta_q(qs_0 y_i/s_{i}x_i)}{\theta_q(q y_i/x_i)}
\\ \qquad
{}\times \prod_{1\leq i < j \leq N} \frac{\theta_q(tx_j/y_i)}{\theta_q(x_j/y_i)}
 \prod_{1\leq i < j \leq N} (1-y_j/y_i) f(y_1,\ldots, y_N)s_0
\\ \qquad
{}=\mathcal{K}(s)
\prod_{1\leq i < j \leq N}
\frac{(qx_j/t x_i ;q)_{\infty} }{(tx_j/x_i ;q)_{\infty}}
\oint \prod_{i=1}^N\frac{{\rm d}y_{i}}{2 \pi \sqrt{-1}y_{i}}\,
\Pi^{(q,t)}(x|y)
\prod_{i=1}^{N} \frac{\theta_q(qs_0 y_i/s_{i}x_i)}{\theta_q(q y_i/x_i)}
 \\ \qquad \qquad\
 {}\times
\prod_{1\leq i < j \leq N} \frac{\theta_q(tx_j/y_i)}{\theta_q(x_j/y_i)}
 \!\!\prod_{1\leq i < j \leq N}\!\! (1-y_j/y_i)
 D_y(s_1,\ldots, s_N|q,q/t) f(y_1,\ldots, y_N).
\end{gather*}
This completes the proof.
\end{proof}

\begin{Corollary}
We have
\begin{gather*}
I(s_1/s_0,\ldots,s_N/s_0) \big( f^{\mathfrak{gl}_N}(x_1,\ldots,x_N;s_1,\ldots,s_N|q,q/t ) \big)
\\ \qquad
{}=f^{\mathfrak{gl}_N}(x_1,\ldots,x_N;s_1,\ldots,s_N|q,q/t ).
\end{gather*}

\begin{proof}
{\sloppy
Since by fixing $\vs$,
the Macdonald function $f^{\mathfrak{gl}_N}(\vx;\vs|q,q/t)$
is in $\mathcal{O}\big(U^N_{|t^{-1}|}\big)$
(Fact~\ref{fact: analyticity}),
we~can use Theorem~\ref{thm: [D,I]=0}.
Moreover, by the uniqueness of Fact~\ref{fact: eigen fn of D},
we can show that
$f^{\mathfrak{gl}_N}(\vx;\vs|q,q/t )$
is an eigenfunction of $I(s_1/s_0,\ldots,s_N/s_0)$:
\begin{gather*}
I(s_1/s_0,\ldots,s_N/s_0) \big( f^{\mathfrak{gl}_N}(\vx;\vs|q,q/t ) \big) \propto f^{\mathfrak{gl}_N}(\vx;\vs|q,q/t ).
\end{gather*}

}
The expansion (Proposition~\ref{prop: Raman expansion})
makes it clear that
the constant term of the LHS is $1$.
Hence, the eigenvalue is $1$.
\end{proof}
\end{Corollary}

\subsection{Integral operator in elliptic case}

In this subsection,
we give brief discussion on the elliptic lift of the integral operator.
Namely, consider the trace at the horizontal representation
of the operator in Proposition~\ref{prop: reproduce I}.
Then we can derive the following integral operator.

\begin{Definition}
Define the operator $\Iellip$
on $\mathbb{C}[[x_2/x_1,\ldots, x_N/x_{N-1}]]$ by
\begin{gather*}
\Iellip(s_1/s_0,\ldots,s_N/s_0)(f(x_1,\ldots , x_N))=
\left(\frac{(pq/t;q,p)_\infty}{(pt;q,p)_\infty} (p;p)\right)^{N}
\mathcal{K}(\vs)
\\ \qquad
{}\times\prod_{1\leq i < j \leq N}
\frac{\Gamma(tx_j/x_i ;q,p) }{\Gamma(qx_j/tx_i ;q,p)}
 \oint \prod_{i=1}^N \frac{{\rm d}y_{i}}{2 \pi \sqrt{-1}y_{i}}
\Pi^{(q,t,p)}(x|y)
\prod_{i=1}^{N} \frac{\theta_q(qs_0 y_i/s_{i}x_i)}{\theta_q(q y_i/x_i)}
\\ \qquad
{}\times\prod_{1\leq i < j \leq N} \frac{\theta_q(tx_j/y_i)}{\theta_q(x_j/y_i)}
\prod_{1\leq i < j \leq N} \theta_{p}(y_j/y_i)\, f(y_1,\ldots , y_N).
\end{gather*}
\end{Definition}
Here, $\mathcal{K}(\vs)$ is defined in Definition~\ref{def: operator I}.
Further we set
\begin{gather*}
\Pi^{(q,t,p)}(x|y)=
\prod_{i,j=1}^N \frac{\Gamma(qy_j/tx_i ;q,p) }{\Gamma(qy_j/x_i ;q,p)}.
\end{gather*}

The following trace can be viewed as
the action of $\Iellip$ on the non-stationary Ruijsenaars functions.

\begin{Proposition}
Let $\vs^+=(s_0,s_1,\ldots, s_N)$.
Then we have
\begin{gather*}
\tr \big(p^d \phicont_{1,N+1}(\vs^+;x_1) \cdots \phicont_{N,N+1}(x_N) \big)
\\ \qquad
{}=\Iellip(s_1/s_0,\ldots,s_N/s_0)
\big(\nonstrui(\vx',p^{1/N}|\vs', t^{-1/N}|q,t)\big),
\end{gather*}
where we used the same notation in Theorem~{\rm \ref{thm: Tr TH TH...}}.
As explained in Remark~{\rm \ref{rem. phicont argument}},
we omitted the spectral parameters $\vs^+$ in the argument of $\phicont$.

\begin{proof}
This can be proved similarly to the one of Proposition~\ref{prop: reproduce I}.
By Lemma~\ref{lem: tr} and Proposition~\ref{prop: trace wcT = fellip},
we can show that
\begin{gather*}
\tr \big(p^d \phicont_{1,N+1}(\vs^+;x_1) \cdots \phicont_{N,N+1}(x_N) \big)
\\ \qquad
{}=\left(\frac{(pq/t;q,p)_{\infty}}{(p;p)_{\infty}(pt;q,p)_{\infty}} \right)^{N}
\Iellip(s_1/s_0,\ldots,s_N/s_0)
\bigg(\prod_{1\leq i<j\leq N} \frac{\Gamma(tx_j/x_i;q,p)}{\Gamma(qx_j/x_i;q,p)}
\\ \qquad \qquad
{}\times\prod_{1\leq i<j\leq N} \frac{(ts_j/s_i;q)_{\infty}}{(qs_j/s_i;q)_{\infty}}\,
\fellip_N(\boldsymbol{s};\vx|q,t,p)\bigg).
\end{gather*}
Applying Theorem~\ref{thm: fgln=fEG}, we obtain the claim.
\end{proof}
\end{Proposition}

We can obtain the commutativity between
the integral operator $\Iellip$ and the Ruijsenaars operator.

{\samepage\begin{Proposition}\label{thm: [ruijD,Iellip]=0}
We obtain
\begin{gather*}
\big[ D_x(s_1,\ldots,s_N;q,q/t,p), \Iellip(s_1/s_0,\ldots,s_N/s_0)\big]=0.
\end{gather*}
Here,
\begin{gather*}
D_x(\vs;q,t,p):= \sum_{k=1}^n
s_k
\prod_{1\leq \ell < k} \frac{\theta_p(tx_k/x_{\ell})}{\theta_p(x_k/x_{\ell})}
\prod_{k < \ell \leq n} \frac{\theta_p(x_{\ell}/tx_k)}{\theta_p(x_{\ell}/x_{k})}
T_{q,x_k}.
\end{gather*}
\end{Proposition}

In the proof, we use the following fact.

}


\begin{fact}[\cite{Ruijsenaars2006zero,Ruijsenaars2009hilbert}]\label{fact: ellip kernel rel}
It follows that
\begin{gather*}
D_x(1,t^{-1},\ldots, t^{-N+1};q,t,p) \Pi^{(q,t,p)}(x|y)=
D_{y^{-1}}\big(1,t^{-1},\ldots, t^{-N+1};q,t,p\big) \Pi^{(q,t,p)}(x|y).
\end{gather*}
Here, we put
\begin{gather*}
D_{y^{-1}}(\boldsymbol{s};q,t,p):= \sum_{k=1}^n
s_k
\prod_{1\leq \ell < k} \frac{\theta_p(ty_{\ell}/y_{k})}{\theta_p(y_{\ell}/y_{k})}
\prod_{k < \ell \leq n} \frac{\theta_p(y_{k}/ty_{\ell})}{\theta_p(y_{k}/y_{\ell})}
T_{q^{-1},y_k}.
\end{gather*}
\end{fact}

\begin{proof}[Proof of Proposition~\ref{thm: [ruijD,Iellip]=0}]
A direct calculation gives
\begin{gather*}
D_x(\vs;q,q/t,p)
\prod_{1\leq i < j \leq N} \frac{\Gamma(tx_j/x_i ;q,p)}{\Gamma(qx_j/t x_i ;q,p)}=
\prod_{1\leq i < j \leq N} \frac{\Gamma(tx_j/x_i ;q,p)}{\Gamma(qx_j/t x_i ;q,p) }
D_x(\vs;q,t,p),
\end{gather*}
and
\begin{gather*}
D_x(\vs;q,t, p)\prod_{i=1}^{N} \frac{\theta_q(qs_0 y_i/s_{i}x_i)}{\theta_q(q y_i/x_i)}
\prod_{1\leq i < j \leq N} \frac{\theta_q(tx_j/y_i)}{\theta_q(x_j/y_i)}
\\ \qquad\qquad\,
{}=\prod_{i=1}^{N} \frac{\theta_q(qs_0 y_i/s_{i}x_i)}{\theta_q(q y_i/x_i)}
\prod_{1\leq i < j \leq N} \frac{\theta_q(tx_j/y_i)}{\theta_q(x_j/y_i)}
s_0 D_x\big(\big(1, t^{-1},\ldots, t^{-N+1}\big);q,t,p\big).
\end{gather*}
By these equations and
Fact~\ref{fact: ellip kernel rel},
we can show that for a function $f(x_1,\ldots ,x_N) \in \mathcal{O}\big(U^N_{|t^{-1}|}\big)$,
\begin{gather}
D_x(\vs;q,q/t,p) \Iellip(s_1/s_0,\ldots,s_N/s_0)f(x_1,\ldots ,x_N)\nonumber
\\ \qquad
{}= \left(\frac{(pq/t;q,p)_\infty}{(pt;q,p)_\infty} (p;p)\right)^{N}\mathcal{K}(s)
\prod_{1\leq i < j \leq N} \frac{\Gamma(tx_j/x_i ;q,p)}{\Gamma(qx_j/t x_i ;q,p)}\nonumber
\\ \qquad \qquad
{}\times \oint \prod_{i=1}^N \frac{{\rm d}y_{i}}{2 \pi \sqrt{-1}y_{i}}
\prod_{i=1}^{N} \frac{\theta_q(qs_0 y_i/s_{i}x_i)}{\theta_q(q y_i/x_i)}
\prod_{1\leq i < j \leq N} \frac{\theta_q(tx_j/y_i)}{\theta_q(x_j/y_i)}\nonumber
\\ \qquad \qquad
{}\times\prod_{1\leq i < j \leq N} \theta_{p}(y_j/y_i) f(y_1,\ldots, y_N)
s_0 D_{y^{-1}}\big(1,\ldots, t^{-N+1};q,t\big)\Pi^{(q,t)}(x|y) \nonumber
\\ \qquad
{}=\left(\frac{(pq/t;q,p)_\infty}{(pt;q,p)_\infty} (p;p)\right)^{N}\mathcal{K}(s)
\prod_{1\leq i < j \leq N} \frac{\Gamma(tx_j/x_i ;q,p)}{\Gamma(qx_j/t x_i ;q,p)}\nonumber
\\ \qquad\qquad
{}\times \oint \prod_{i=1}^N \frac{{\rm d}y_{i}}{2 \pi \sqrt{-1}y_{i}}
\prod_{i=1}^{N} \frac{\theta_q(qs_0 y_i/s_{i}x_i)}{\theta_q(q y_i/x_i)}
\prod_{1\leq i < j \leq N} \frac{\theta_q(tx_j/y_i)}{\theta_q(x_j/y_i)}\nonumber
\\ \qquad\qquad
{}\times\prod_{1\leq i < j \leq N} \theta_{p}(y_j/y_i)
 f(y_1,\ldots, y_N)s_0 \sum_{k=1}^N
\prod_{ \ell \neq k} \frac{\theta_p(y_{k}/ty_{\ell}) }{\theta_p(y_{k}/y_{\ell})}\,
T_{q^{-1},y_k}\Pi^{(q,t)}(x|y).
\label{eq: cal ellip D I}
\end{gather}
Similarly to the proof of Theorem~\ref{thm: [D,I]=0},
by the change of variable $y_k \rightarrow q y_k$
in the each term,
we can show that (\ref{eq: cal ellip D I}) is equal to
\begin{gather*}
\left(\frac{(pq/t;q,p)_\infty}{(pt;q,p)_\infty} (p;p)\right)^{N}\mathcal{K}(s)
\prod_{1\leq i < j \leq N} \frac{\Gamma(tx_j/x_i ;q,p)}{\Gamma(qx_j/t x_i ;q,p)}
\oint \prod_{i=1}^N\frac{{\rm d}y_{i}}{2 \pi \sqrt{-1}y_{i}}\Pi^{(q,t,p)}(x|y)
\\ \qquad
{}\times\prod_{i=1}^{N} \frac{\theta_q(qs_0 y_i/s_{i}x_i)}{\theta_q(q y_i/x_i)}
\prod_{1\leq i < j \leq N} \frac{\theta_q(tx_j/y_i)}{\theta_q(x_j/y_i)}
\prod_{1\leq i < j \leq N} \theta_{p}(y_j/y_i)
\\ \qquad
{}\times D_y(s_1,\ldots, s_N|q,q/t,p) f(y_1,\ldots, y_N).
\end{gather*}
This completes the proof.
\end{proof}

In this subsection,
we have derived the integral operator $\Iellip$
and
given commutativity with the Ruijsenaars operator.
Unfortunately,
the non-stationary Ruijsenaars functions $\nonstrui$
are not the eigenfunctions of the Ruijsenaars operator.
So, neither for $\Iellip$.
It is left to a future study to
find an operator whose eigenfunctions are $\nonstrui$.

\section[Conformal limit q to 1]{Conformal limit $\boldsymbol{q \rightarrow 1}$}
\label{sec: q->1}

\subsection{Preparation}

In this section,
we will derive the
relation of the Virasoro algebra and the primary field
from the relation of the $q$-Virasoro algebra and the Mukad\'e operator $\cTV$.
Firstly, we define the following algebra.

\begin{Definition}
For $k = 2,\dots,N$,
define
\begin{gather*}
X^{(k)}(z) =
\sum_{n\in\mathbb{Z}} X^{(k)}_n z^{-n} = \Xo\big(\gamma^{2(1-k)} z\big)\Xo\big(\gamma^{2(2-k)} z\big)\cdots\Xo(z)
\in \mathrm{End}(\mathcal{F}_{\vu})\big[\big[z^{\pm 1}\big]\big].
\end{gather*}
\end{Definition}

\begin{fact}[\cite{AFHKSY2011notes}]
The operator $X^{(k)}(z)$ is of the form
\begin{gather*}
X^{(k)}(z)
= \sum_{1\leq j_1 <\cdots <j_k \leq N} {:} \Lambda^{(j_1)}(z) \cdots \Lambda^{(j_k)}\big((q/t)^{k-1}z\big){:} u_{j_1} \cdots u_{j_k}.
\end{gather*}
Here, $\Lambda^{(j)}(z)$ is defined in Definition~\ref{Def_X^k}.
\end{fact}

The algebra generated by $X^{(k)}(z)$ can be regarded as
the tensor product of the $q$-deformed $W_N$ algebra
and some Heisenberg algebra \cite{FHSSY2010Kernel}
(See Proposition~\ref{prop: decomp X} in the next subsection).
Their PBW(Poincar\'{e}--Birkhoff--Witt)-type basis is well-understood.

\begin{Definition}
For an $N$-tuple of partitions
$\vl=\big(\lo, \lt, \ldots, \lambda^{(N)}\big)$,
we define 
the vectors $\Ket{X_{\vl}}=\Ket{X_{\vl}(\vu)} \in \mathcal{F}_{\vu}$ and
$\Bra{X_{\vl}}=\Bra{X_{\vl}(\vu)}\in \mathcal{F}_{\vu}^*$ by
\begin{gather*}
\Ket{X_{\vl}(\vu)} :=
 X^{(1)}_{-\lambda^{(1)}_1} X^{(1)}_{-\lambda^{(1)}_2} \cdots X^{(2)}_{-\lambda^{(2)}_1} X^{(2)}_{-\lambda^{(2)}_2}\cdots X^{(N)}_{-\lambda^{(N)}_1} X^{(N)}_{-\lambda^{(N)}_2} \cdots \ketzero, \\
 \Bra{X_{\vl}(\vu)} :=
\brazero \cdots X^{(N)}_{\lambda^{(N)}_2} X^{(N)}_{\lambda^{(N)}_1} \cdots X^{(2)}_{\lambda^{(2)}_2} X^{(2)}_{\lambda^{(2)}_1} \cdots
X^{(1)}_{\lambda^{(1)}_2} X^{(1)}_{\lambda^{(1)}_1} .
\end{gather*}
\end{Definition}

\begin{fact}[\cite{FOS2019Generalized,Ohkubo2017Kac}]
The set $(\Ket{X_{\vl}})$ (resp.~$(\Bra{X_{\vl}})$) form a PBW-type basis
of $\mathcal{F}_{\vu}$ (resp.~$\mathcal{F}^*_{\vu}$), if
$u_i\neq q^st^{-r}u_j$ and $u_i\neq 0$
for all $i, j$ and $r, s \in \mathbb{Z}$.
\end{fact}


\begin{Definition}
Define the linear operator $\cV(x)=\cVweight{\vv \\ \vu}{x}\colon \cF_{\boldsymbol{u}} \to \cF_{\boldsymbol{v}}$
by
\begin{gather*}
\left(1-\frac{x}{z}\right)X^{(i)}(z) \cV(x)
=\left(1- (t/q)^i\frac{x}{z}\right)\cV(x) X^{(i)}(z)
\qquad (i \in \{1,2,\dots,N\})
\end{gather*}
and the normalization condition $\bra{\boldsymbol{0}}\cV(x) \ket{\boldsymbol{0}}=1$.
\end{Definition}

It is known that the operator $\cV(x)$ exists uniquely \cite{FOS2019Generalized}.
Moreover, their matrix element formula is proved.

\begin{fact}[\cite{FOS2019Generalized}]
\label{fact: mat. el. formula}
It follows that
\begin{gather*}
\bra{K_{\boldsymbol{\lambda}}(\vv)}\cV (x)\ket{K_{\boldsymbol{\mu}}(\vu)}
=\frac{\big((-\gamma^2)^N e_{N}(\vu)x\big)^{|\vl|}}{( \gamma^2 x )^{|\vm|}}
\prod_{i=1}^N\frac{u_i^{|\mu^{(i)}|}g_{\mu^{(i)}}}
{\big(v_i^{|\lambda^{(i)}|}g_{\lambda^{(i)}}\big)^{N-1}}
\prod_{i,j=1}^N N_{\lambda^{(i)},\mu^{(j)}}(qv_i/tu_j).
\end{gather*}
\end{fact}

\begin{Remark}
The operator $ \cTV(\boldsymbol{u})$ satisfies \cite{FOS2019Generalized}
\begin{gather*}
 \left(1- \frac{w}{z}\right)X^{(i)}(z) \cTV(\boldsymbol{u}, \boldsymbol{v};w) = \gamma^{-i}\left(1- \gamma^{2 i} \frac{w}{z}\right)\cTV(\boldsymbol{u}, \boldsymbol{v};w) X^{(i)}(z) .
\end{gather*}
Thus, the $\cV (x)$ can be realized by $\cTV$
with some modifications to spectral parameters.
\end{Remark}

\subsection[Conformal limit q to 1]{Conformal limit $\boldsymbol{q \rightarrow 1}$}

Consider the limit $q \rightarrow 1$ of $\cV(x)=\cVweight{\vv \\ \vu}{x}$
in the case $N=2$. Put
\begin{gather}\label{eq: q t parametrization}
\begin{array}{l}
q={\rm e}^{b^{-1} \hbar}, \qquad t={\rm e}^{-b \hbar},
\\
u_i={\rm e}^{u'_i \hbar},\qquad v_i={\rm e}^{v'_i \hbar}\qquad (i=1,2),
\end{array}
\end{gather}
and we parametrize $v'_i$ and $u'_i$ by the parameters $\parap'$, $\parap$ and $\alpha$
such that
\begin{gather}\label{eq: parametrization}
v'_i-u'_j=\parap'_i-\parap_j-\alpha, \qquad i,j=1,2.
\end{gather}
Here, we set $\parap_1=-\parap_2=\parap$, $\parap'_1=-\parap'_2=\parap'$.
Write
\begin{gather*}
Q:=b+b^{-1}.
\end{gather*}

The parameters $b$, $\parap$, $\parap'$, and $\alpha$
directly correspond to $b$, $P$, $P'$, and $\alpha$
in \cite{AFLT2011combinatorial},
which come from the central charges of the Virasoro algebra,
highest weights of its representation, and so on.

\begin{Remark}
There is an ambiguity of the choice of parametrization satisfying
(\ref{eq: parametrization}).
No matter how we choose it,
final results (Theorem~\ref{thm: the Vir rel}) are not affected if
(\ref{eq: parametrization}) is satisfied.
For example, the simplest parametrization is
\begin{gather}
u'_1=\parap+\alpha/2, \qquad u'_2=-\parap+\alpha/2, \label{eq: simple para 1}
\\
v'_1=\parap'-\alpha/2, \qquad v'_2=-\parap'-\alpha/2. \label{eq: simple para 2}
\end{gather}
Even if we add an arbitrary parameter to
(\ref{eq: simple para 1}) and (\ref{eq: simple para 2})
to keep the degree of freedom of $u'_i$ and $v'_i$,
the results stay the same.
\end{Remark}

It is easy to show the following lemma.

\begin{Lemma}
If \eqref{eq: parametrization} is satisfied,
it follows that
\begin{gather*}
u'_1-u'_2=2\parap, \qquad v'_1-v'_2=2\parap',
\\
u'_1+u'_2-v'_1-v'_2=2\alpha.
\end{gather*}
\end{Lemma}

The parametrization above 
is designed
so that the following factor appears in the limit of
the Nekrasov factor.

\begin{Definition}
Set
\begin{gather*}
Z_{bif}(\alpha|\parap',\vl|\parap,\vm)
=\prod_{i,j=1}^2 \prod_{(k,l) \in \lambda^{(i)}}
\big( \parap'_i-\parap_j-\alpha +b^{-1} (a_{\lambda^{(i)}}(k,l)+1) -b \ell_{\mu^{(j)}}(k,l) \big)
\\ \hphantom{Z_{bif}(\alpha|\parap',\vl|\parap,\vm)=}
{}\times \prod_{(k,l) \in \mu^{(j)}}
\big( \parap'_i-\parap_j-\alpha -b^{-1} a_{\mu^{(j)}}(k,l)
+b(\ell_{\lambda^{(i)}}(k,l)+1) \big).
\end{gather*}
\end{Definition}

{\samepage\begin{Proposition}\label{prop: limit of Nek}
We have
\begin{gather*}
\prod_{i,j=1}^2 N_{\lambda^{(i)}, \mu^{(i)}}(qv_i/tu_j)=
Z_{bif}(\alpha|\parap',\vl|\parap,\vm) \cdot \hbar^{2(|\vl|+|\vm|)} + \mathcal{O}\big(\hbar^{2(|\vl|+|\vm|)+1} \big).
\end{gather*}

\begin{proof}
Follows by direct calculation.
\end{proof}
\end{Proposition}

}

Let us define bosons independent of $\hbar$
which is naturally obtained from the parameterization~(\ref{eq: q t parametrization})
and the limit $\hbar \rightarrow 0$.

\begin{Definition}
Define the Heisenberg algebra $\mathsf{a}_{n}$ ($n \in \mathbb{Z}_{\neq 0}$)
by the commutation relation
\begin{gather*}
[\mathsf{a}_{n},\mathsf{a}_{m}]=-b^{-2}n\delta_{n+m,0}.
\end{gather*}
Set
\begin{gather*}
\mathsf{a}^{(1)}_n=\mathsf{a}_n \otimes 1,\qquad
\mathsf{a}^{(2)}_n= 1\otimes \mathsf{a}_n,
\end{gather*}
and we assume that
\begin{gather*}
a^{(i)}_{-n}=\mathsf{a}^{(i)}_{-n},\qquad
a^{(i)}_{n}=\big({-}b^{2}\big)\frac{1-q^n}{1-t^n}\mathsf{a}^{(i)}_{n} \qquad (n>0),
\\
\lim_{\hbar \rightarrow 0} a^{(i)}_{-n}=\mathsf{a}^{(i)}_{-n}, \qquad
\lim_{\hbar \rightarrow 0} a^{(i)}_{n}=\mathsf{a}^{(i)}_{n}.
\end{gather*}
\end{Definition}

It is known that
the generalized Macdonald functions are reduced to
the generalized Jack functions \cite{MS2014Towards}
in the limit $\hbar \rightarrow 0$.
They are defined as eigenfunctions of the following Hamiltonian.
The existence theorem also follows similarly to the generalized Macdonald functions.

\begin{Definition}
Define the operator $H_b$ by
\begin{gather*}
H_b=H_b^{(1)}+H_b^{(2)}
-\big(1+b^{-2}\big)\sum_{n=1}^{\infty} n \mathsf{a}^{(1)}_{-n}
\mathsf{a}^{(2)}_{n},
\end{gather*}
where $H_b^{(i)}$ ($i=1,2$) is a modified Hamiltonian of the Calogero--Sutherland model:
\begin{gather*}
H_b^{(i)}
\mathbin{:=} \frac{1}{2}\sum_{n,m}
\left( \mathsf{a}^{(i)}_{-n} \mathsf{a}^{(i)}_{-m}
\mathsf{a}^{(i)}_{n+m}
+ \mathsf{a}^{(i)}_{-n-m}
 \mathsf{a}^{(i)}_n \mathsf{a}^{(i)}_m \right)
-\sum_{n=1}^{\infty} \left( b^{-1} u_i' + \frac{1+b^{-2}}{2} n \right) \mathsf{a}^{(i)}_{-n} \mathsf{a}^{(i)}_n.
\end{gather*}
\end{Definition}

\begin{fact}[\cite{Ohkubo2014Generalized}]
There exists a unique function $\ket{J_{\vl}}=\ket{J_{\vl}(u'_1,u'_2)}$
satisfying the following two conditions:
\begin{gather*}
 \ket{J_{\vl}}=\prod_{i=1}^2 m_{\lambda^{(i)}}\big(\mathsf{a}^{(i)}_{-n}\big) \ketzero
+\sum_{\vm <^{\mathrm{L}} \vl} d_{\vl \vm}
\prod_{i=1}^2 m_{\mu^{(i)}}\big(\mathsf{a}^{(i)}_{-n}\big) \ketzero , \qquad
d_{\vl \vm} \in \mathbb{C};
\\
 H_{b} \ket{J_{\vl}}= e'_{\vl} \ket{J_{\vl}} ,\qquad e'_{ \vl} \in \mathbb{C}.
\end{gather*}
Similarly,
there exists a unique function $\bra{J_{\vl}}=\bra{J_{\vl}(u'_1,u'_2)}$
such that
\begin{gather*}
\bra{J_{\vl}}=\brazero \prod_{i=1}^2 m_{\lambda^{(i)}}\big(\mathsf{a}^{(i)}_{n}\big)
+\sum_{\vm <^{\mathrm{R}} \vl} d'_{\vl \vm}
\brazero \prod_{i=1}^2 m_{\mu^{(i)}}\big(\mathsf{a}^{(i)}_{n}\big) , \qquad
d'_{\vl \vm} \in \mathbb{C};
\\
\bra{J_{\vl}} H_{b}= e'_{\vl} \bra{J_{\vl}} ,\qquad e'_{ \vl} \in \mathbb{C}.
\end{gather*}
\end{fact}

We call the eigenfunctions $\ket{J_{\vl}}$ and
$\bra{J_{\vl}}$ the generalized Jack functions.

\begin{fact}[\cite{Ohkubo2014Generalized}]
\label{fact: limit of Gen. Mac}
We have
\begin{gather*}
\ket{P_{\vl}(\vu)} \underset{\substack{\hbar \rightarrow 0,
\\
u_i={\rm e}^{u_i'\hbar },\ q={\rm e}^{b^{-1} \hbar},\ t={\rm e}^{-b \hbar} }}{\longrightarrow} \ket{J_{\vl}(u'_1,u'_2)},
\\
\bra{P_{\vl}(\vu)} \underset{\substack{\hbar \rightarrow 0,
\\
_i={\rm e}^{u_i'\hbar },\ q={\rm e}^{b^{-1} \hbar},\ t={\rm e}^{-b \hbar} }}{\longrightarrow} \bra{J_{\vl}(u'_1,u'_2)}.
\end{gather*}
\end{fact}

In fact, the generalized Jack functions are defined
for general $N$.
Also for general $N$,
they correspond to the limit of the generalized Macdonald function.

Next, we take the limit of the generator $X^{(i)}(z)$.
In advance, we decompose the generator $X^{(i)}(z)$ into
the $q$-Virasoro algebra and some Heisenberg algebra
in order to obtain the relation of the Virasoro Primary fields.
This decomposition can be obtained by the following linear transformation
of the bosons.

\begin{Definition}
For $n>0$, define
\begin{gather*}
b'_{-n}:=\frac{(1-t^{-n})}{(1+(q/t)^n)}\big(a^{(1)}_{-n}-\gamma^{-n}a^{(2)}_{-n}\big), \qquad
b'_{n}:=-\frac{(1-t^{n})}{(1+(q/t)^n)}\big(a^{(1)}_{n}-\gamma^{-n}a^{(2)}_{n}\big),
\\
b''_{-n}:=(1-t^{-n})\big(\gamma^{-2n} a^{(1)}_{-n}+\gamma^{-n}a^{(2)}_{-n}\big), \qquad
b''_{n}:=-(1-t^{n})\big(a^{(1)}_{n}+\gamma^{n}a^{(2)}_{n}\big).
\end{gather*}
Furthermore,
for $n \in \mathbb{Z}_{\neq 0}$,
define
\begin{gather*}
\beta^{(1)}_{n}:=-\frac{\sqrt{-1}}{2}b\big(\mathsf{a}^{(1)}_{n}-\mathsf{a}^{(2)}_{n}\big), \qquad
\beta^{(2)}_{n}:=-\frac{\sqrt{-1}}{2}b\big(\mathsf{a}^{(1)}_{n}+\mathsf{a}^{(2)}_{n}\big).
\end{gather*}
\end{Definition}

\begin{Definition}
Define the zero mode by
\begin{gather*}
\beta^{(1)}_0 :=\sqrt{-1}\cdot \frac{u'_1-u'_2}{2}=\sqrt{-1}\parap, \qquad
\beta^{(2)}_0 :=\sqrt{-1}\cdot \frac{u'_1+u'_2}{2} \qquad
\mbox{on $\cF_{\vu}$},
\\
\beta^{(1)}_0 :=\sqrt{-1}\cdot \frac{v'_1-v'_2}{2}=\sqrt{-1}\parap', \qquad
\beta^{(2)}_0 :=\sqrt{-1}\cdot \frac{v'_1+v'_2}{2} \qquad
\mbox{on $\cF_{\vv}$}.
\end{gather*}
\end{Definition}

\begin{Proposition}
For $n,m\in \mathbb{Z}$, it follows that
\begin{gather*}
[b'_n,b'_{m}]=-n\frac{(1-q^n)(1-t^{-n})}{(1+(q/t)^n)}\delta_{n+m,0},
\\
[b'_n,b''_m]=\big[\beta^{(1)}_n, \beta^{(2)}_n\big]=0,
\\
[\beta^{(2)}_{n},\beta^{(2)}_{m}]=\frac{n}{2}\delta_{n+m,0}.
\end{gather*}

\begin{proof}
It follows from direct computation.
\end{proof}
\end{Proposition}

\begin{Definition}
Set
\begin{gather*}
\Lambda'_1(z)={:}\exp\bigg(\sum_{n\neq 0} \frac{b'_n}{|n|}z^{-n} \bigg){:}, \qquad
\Lambda'_2(z)={:}\exp\bigg({-}\sum_{n\neq 0} \frac{b'_n}{|n|}(t/q)^nz^{-n} \bigg){:},
\\
\Lambda''(z)=
\exp\bigg(\sum_{n>0} \frac{b''_{-n}}{n(1+(q/t)^{n})}z^{n} \bigg)
\exp\bigg(\sum_{n>0} \frac{(q/t)^{n}b''_{n}}{n(1+(q/t)^{n})}z^{-n} \bigg).
\end{gather*}
Moreover, define
\begin{gather*}
T(z)=\Lambda'_1(z){\rm e}^{-\sqrt{-1}\hbar \beta^{(1)}_0 }+\Lambda'_2(z){\rm e}^{\sqrt{-1}\hbar \beta^{(1)}_0 },
\qquad
Y(z)=\Lambda''(z){\rm e}^{-\sqrt{-1}\hbar \beta^{(2)}_0}.
\end{gather*}
\end{Definition}

\begin{Proposition}\label{prop: decomp X}
$\Xo(z)$ can be decomposed as
\begin{gather*}
\Xo(z)=T(z)Y(z).
\end{gather*}
Moreover we have
\begin{gather*}
X^{(2)}(z)={:}\exp\bigg(\sum_{n\neq 0} \frac{b''_n}{|n|}z^{-n}\bigg){:}{\rm e}^{-2 \sqrt{-1}
\hbar \beta^{(2)}_0}.
\end{gather*}
\end{Proposition}

\begin{Proposition}
The operator $T(z)$ satisfies the defining relation of the $q$-Virasoro algebra:
\begin{gather*}
f(w/z)T(z)T(w)-f(z/w)T(w)T(z)=-\frac{(1-q)\big(1-t^{-1}\big)}{1-(q/t)}
(\delta(qw/tz)-\delta(tw/qz)),
\end{gather*}
where
\begin{gather*}
f(z)= \exp \bigg(\sum_{n>0}\frac{1}{n} \frac{(1-q^n)(1-t^{-n})}{1+(q/t)^n} z^n\bigg).
\end{gather*}
\end{Proposition}

These proposition can be shown by direct calculation.
(See also \cite{FHSSY2010Kernel}.)
For representation theory of the $q$-Virasoro algebra,
we refer the reader to \cite{SKAO1995quantum}.
Now, we consider the limit of these generators.

\begin{fact}[\cite{SKAO1995quantum}]\label{fact: qVir expansion}
The $\hbar$-expansion of $T(z)$
is of the form
\begin{gather*}
T(z)=2-\left(L(z)-\frac{Q^2}{4}\right)\hbar^2+\mathcal{O}\big(\hbar^4\big).
\end{gather*}
Here $L(z)=\sum_{n\in \mathbb{Z}} L_nz^{-n}$ is
some operator written by the boson $\beta^{(1)}_n$.
Moreover, $L_n$'s generate the Virasoro algebra
with the central element $c=1+6Q^2$:
\begin{gather*}
[L_n,L_m]=(n-m)L_{n+m}+
\frac{n(n-1)(n+1)c}{12}\delta_{n+m,0}.
\end{gather*}
\end{fact}

\begin{Remark}\label{rem: regard F as vir mod}
On $\cF_{\vu}^{(2,0)}$, the Virasoro algebra acts as
\begin{gather*}
L_0\ketzero =\left( \frac{Q^2}{4}-\parap^2\right) \ketzero, \qquad
L_n\ketzero=0 \qquad (n>0).
\end{gather*}
Similarly,
on $\cF_{\vv}^{(2,0)*}$
\begin{gather*}
\brazero L_0 =\left( \frac{Q^2}{4}-\parap'^2\right) \brazero, \qquad
\brazero L_{-n}=0 \qquad (n>0).
\end{gather*}
If $\parap$, $\parap'$ and $b$ are generic,
the following vectors form a basis on $\cF^{(2,0)}_{\vu}$ and $\cF^{(2,0)*}_{\vv}$, respectively:
\begin{gather*}
\ket{L_{\lambda, \mu}}:=L_{-\lambda_1}L_{-\lambda_2}\cdots
\beta^{(2)}_{-\mu_1}\beta^{(2)}_{-\mu_2}\cdots\ketzero,
\qquad \lambda, \mu \in \parset,
\\
\bra{L_{\lambda, \mu}}:=\brazero \cdots\beta^{(2)}_{\mu_2}\beta^{(2)}_{\mu_1}
\cdots L_{\lambda_2}L_{\lambda_1},\qquad \lambda, \mu \in \parset.
\end{gather*}
Hence, we can identify $\cF_{\vu}^{(2,0)}$ with
the tensor product of the Verma module of
the Virasoro algebra $\langle L_n \rangle $
and the Fock space of the Heisenberg algebra $ \langle \beta^{(2)}_n \rangle$.
Further,
$\ketzero$ can be regarded as
the tensor product of
the highest wight vector of the highest weight $\frac{Q^2}{4}-\parap^2$
and the vacuum state of the Fock space.
For simplicity,
we hereafter assume that $\parap$, $\parap'$ and $b$ are generic
so that the modules are irreducible.
\end{Remark}

By Fact~\ref{fact: mat. el. formula}, Proposition~\ref{prop: limit of Nek}
and Fact~\ref{fact: limit of Gen. Mac},
we can assume the following expansion.
Namely, there is no pole at $\hbar=0$.

\begin{Definition}
Define $\cV_i(z)$ ($i=0, 1, \ldots$) to be each coefficient in the $\hbar$-expansion of
$\cV(z)=\cVweight{\vv \\ \vu}{z}$, i.e.,
\begin{gather*}
\cV(z)=\cV_0(z)+\cV_1(z)\hbar + \cV_2(z)\hbar^2 +\cdots.
\end{gather*}
Furthermore, we introduce
\begin{gather*}
\Phi^{\mathrm{H.V}}(z)=z^{-\parap'{}^2+\parap^2-\alpha(Q-\alpha)}\cV_0(z).
\end{gather*}
\end{Definition}

We obtain the result that the operator $\Phi^{\mathrm{H.V}}(z)$ corresponds to the Virasoro primary fields.

\begin{Theorem}\label{thm: the Vir rel}
$\Phi^{\mathrm{H.V}}(z)$ satisfies
the relation
\begin{gather*}
[L_n,\Phi^{\mathrm{H.V}}(z)]=z^n\left(z \frac{\partial}{\partial z}
+\alpha(Q-\alpha)(n+1)\right)\Phi^{\mathrm{H.V}}(z).
\end{gather*}
Moreover, we obtain
\begin{gather*}
\big[\beta^{(2)}_n, \Phi^{\mathrm{H.V}}(w)\big]=\sqrt{-1} w^n (Q-\alpha) \Phi^{\mathrm{H.V}}(w), \qquad n>0,
\\
\big[\beta^{(2)}_{-n}, \Phi^{\mathrm{H.V}}(w)\big]=-\sqrt{-1} w^{-n} \alpha \Phi^{\mathrm{H.V}}(w), \qquad n\geq 0.
\end{gather*}
\end{Theorem}

The proof is given in Section~\ref{sec: proof of Vir rel}.
Let us also state the following proposition.

\begin{Proposition}
We have
\begin{gather*}
\frac{\braket{\widetilde{J}_{\vl}(v'_1,v'_2)|\Phi^{\mathrm{H.V}}(z)|\widetilde{J}_{\vm}(u'_1,u'_2)}}
{\brazero \Phi^{\mathrm{H.V}}(z)\ketzero}
=Z_{bif}(\alpha|p',\vl|p,\vm),
\end{gather*}
where we put
\begin{gather*}
\ket{\widetilde{J}_{\vl}(u'_1,u'_2)}= \ket{J_{\vl}(u'_1,u'_2)} \lim_{\hbar \rightarrow 0} \hbar^{-N|\vl|} \mathcal{C}_{\vl}^{+}(\vu),
\\
\bra{\widetilde{J}_{\vl}(v'_1,v'_2)}= \bra{J_{\vl}(v'_1,v'_2)} \lim_{\hbar \rightarrow 0} \hbar^{-N|\vl|} \mathcal{C}_{\vl}^{-}(\vv).
\end{gather*}

\begin{proof}
This is clear by Fact~\ref{fact: mat. el. formula},
Fact~\ref{fact: limit of Gen. Mac} and Proposition~\ref{prop: limit of Nek}.
\end{proof}
\end{Proposition}

The relations in Theorem~\ref{thm: the Vir rel}
are the exactly same as the ones in \cite{AFLT2011combinatorial} up to notation.
We~can also check that $\ket{\widetilde{J}_{\vl}}$ corresponds to the Alba, Fateev, Litvinov and Tarnopolski's (AFLT) basis~\cite{AFLT2011combinatorial}
under the identification between the Fock space
$\cF_{\vu}$ and
the Verma module of the algebra $(\mathrm{Virasoro}) \otimes (\mathrm{Heisemberg})$.
However,
$\ket{\widetilde{J}_{\vl}}$ and the AFLT basis are defined
in the different ways.
While $\ket{J_{\vl}}$ is defined as eigenfunctions of $H_b$,
the AFLT basis is defined by the condition that
the matrix elements of the primary fields reproduce the $Z_{bif}$.
Our matrix elements formula (Fact~\ref{fact: mat. el. formula})
and Theorem~\ref{thm: the Vir rel}
prove that $\ket{\widetilde{J}_{\vec{\lambda}}}$ and
the AFLT basis actually coincide.
Furthermore,
these results also prove
the 4D AGT correspondence \cite{AGT2010liouville}
which states the duality between the (non-deformed) Virasoro algebra
and the 4D $\mathcal{N}=2$ gauge theory.
We can also expect the similar results for general $N$.

\subsection{Proof of Theorem~\ref{thm: the Vir rel}}
\label{sec: proof of Vir rel}

\begin{Definition}
Define the currents
\begin{gather*}
J^{(1)}(z)=-2 \sqrt{-1}\sum_{n \in \mathbb{Z}} \beta^{(2)}_n z^{-n},
\\
J^{(2)}(z)=\frac{1}{8}{ :}J^{(1)}(z)^2{:}-\frac{1}{4}
\sum_{n>0} n\big(\big(1+2b^2\big)\mathsf{a}^{(1)}_{-n} + b^{2}\mathsf{a}^{(2)}_{-n}\big)z^n
\\ \hphantom{J^{(2)}(z)=}
{}-\frac{1}{4} \sum_{n>0}n \big(\big(2+b^{2}\big) \mathsf{a}^{(1)}_n+ \mathsf{a}^{(2)}_n \big) z^{-n},
\\
\widetilde{J}^{(2)}(z)= \frac{1}{2} {:}J^{(1)}(z)^2{:}
-\frac{1}{2}\sum_{n>0} n\big(\big(2+3b^2\big)\mathsf{a}^{(1)}_{-n} + \big(1+2b^{2}\big)\mathsf{a}^{(2)}_{-n}\big)z^n
\\ \hphantom{\widetilde{J}^{(2)}(z)= }
{} +\frac{1}{2}\sum_{n>0} n\big({-}\,\mathsf{a}^{(1)}_n +b^2 \mathsf{a}^{(2)}_n\big) z^{-n}.
\end{gather*}
\end{Definition}

These currents
appear in the following expansion.

\begin{Lemma}\label{lem: Y and Xt expansion}
We have
\begin{gather*}
Y(z)=1+ \frac{1}{2} J^{(1)}(z)\hbar + J^{(2)}(z) \hbar^2+ \mathcal{O}\big(\hbar^3\big),
\\
(t/q)Y(z)=1+ \left({-}Q+ \frac{1}{2}J^{(1)}(z) \right) \hbar
+ \left( J^{(2)}(z)-\frac{1}{2}Q J^{(1)}(z) +\frac{Q^2}{2} \right)\hbar^2+ \mathcal{O}\big(\hbar^3\big),
\\
\Xt(z)=1+ J^{(1)}(z) \hbar + \widetilde{J}^{(2)}(z) \hbar^2+ \mathcal{O}\big(\hbar^3\big),
\\
(t/q)^2 \Xt(z)=1+ \left({-}2Q+J^{(1)}(z)\right) \hbar
+ \left( \widetilde{J}^{(2)}(z)-2Q J^{(1)}(z) +2Q^2 \right) \hbar^2+ \mathcal{O}\big(\hbar^3\big).
\end{gather*}

\begin{proof}
By direct calculation.
\end{proof}

\end{Lemma}

At first,
we prove the relation with the Heisenberg algebra $\beta^{(2)}_n$.

\begin{Proposition}\label{prop: beta2 cV_0 rel.s}
We have
\begin{gather}
\big[\beta^{(2)}_n, \cV_0(w)\big]=\sqrt{-1} w^n (Q-\alpha) \cV_0(w), \qquad n>0,
\label{eq: beta V_0 rel +}
\\
\big[\beta^{(2)}_{-n}, \cV_0(w)\big]=-\sqrt{-1} w^{-n} \alpha \cV_0(w), \qquad n\geq 0.
\label{eq: beta V_0 rel -}
\end{gather}
\end{Proposition}

\begin{proof}
We calculate the $\hbar$-expansion of
the defining relation
\begin{gather}
\left( \Xt(z)-\frac{w}{z} \Xt(z)\right) \cV(w)
=\cV(w)\left( \Xt(z)-(t/q)^2 \frac{w}{z} \Xt(z) \right).
\label{eq: Xt cV rel}
\end{gather}
Lemma~\ref{lem: Y and Xt expansion}
shows that
the coefficient in front of $\hbar^1$
in the expansion of (\ref{eq: Xt cV rel}) is
\begin{gather}
\left( J^{(1)}(z)-\frac{w}{z}J^{(1)}(z) \right) \cV_0(w)
= \cV_0(w) \left( J^{(1)}(z)-\frac{w}{z} \left({-}2Q+ J^{(1)}(z)\right) \right).
\label{eq: J1 and cV}
\end{gather}
The coefficients of (\ref{eq: J1 and cV}) in front of $z^n$
gives the following relations.
\begin{gather*}
\shortintertext{$\bullet$ Coefficient of $z^n$ ($n>0$): }
\big[\beta^{(2)}_{-n},\cV_0(w)\big]=w \big[\beta^{(2)}_{-n-1},\cV_0(w)\big],
\\
\shortintertext{$\bullet$ Coefficient of $z^0$: }
\big[\beta^{(2)}_{-1},\cV_0(w)\big]=-\sqrt{-1} \frac{\alpha}{w} \cV_0(w),
\shortintertext{$\bullet$ Coefficient of $z^{-1}$:}
\big[\beta^{(2)}_1,\cV_0(w)\big]=\sqrt{-1}w(Q-\alpha)\cV_0(w),
\shortintertext{$\bullet$ Coefficient of $z^{-n}$ ($n>1$):}
\big[\beta^{(2)}_n,\cV_0(w)\big]=w \big[\beta^{(2)}_{n-1},\cV_0(w)\big].
\end{gather*}
By solving inductively these relations,
we can get Proposition~\ref{prop: beta2 cV_0 rel.s}.
\end{proof}

In the proof of Proposition~\ref{prop: beta2 cV_0 rel.s},
we computed the coefficient in front of $\hbar^1$
with respect to the relation between $\cV(w)$
and $\Xt(z)$.
Actually,
the same relation can be obtained from the relation between
$\cV(w)$ and $\Xo(z)$.
(The coefficients of $\hbar^0$ give trivial relations.)
Next, we calculate the coefficient of $\hbar^2$,
from which
the relation with the Virasoro algebra arises.
Let us introduce the following operator.

\begin{Definition}
Define $\mathcal{L}(z)=\sum_{n\in \mathbb{Z}} \mathcal{L}_n z^{-n}$ by
\begin{gather*}
\mathcal{L}(z):=2J^{(2)}-\widetilde{J}^{(2)}(z)
= \sum_{n,m \in \mathbb{Z}}{:}\beta^{(2)}_n\beta^{(2)}_m{:}z^{-n-m}
-\sqrt{-1} Q \sum_{n\in \mathbb{Z}} n \beta^{(2)}_{n}z^{-n}.
\end{gather*}
\end{Definition}

\begin{Lemma}\label{lem: L cL and cV rel.}
We have
\begin{gather}
\left({-}L(z) + \mathcal{L}(z) -\frac{w}{z}\big({-}L(z) + \mathcal{L}(z) \big) \right)\cV_0(w)\nonumber
\\ \qquad
{}=\cV_0(w)\left({-}L(z) + \mathcal{L}(z) -\frac{w}{z}
\big({-}L(z) + \mathcal{L}(z) +Q J^{(1)}(z)-Q^2\big) \right).
\label{eq: L cL and cV rel.}
\end{gather}

\begin{proof}
We calculate the $\hbar$-expansion of the defining relation
\begin{gather}\label{eq: def rel Xo cV}
\left(1-\frac{w}{z}\right)\Xo(z)\cV(w)
=\left(1-(t/q)\frac{w}{z}\right)\cV(w)\Xo(z).
\end{gather}
By Fact~\ref{fact: qVir expansion} and Lemma~\ref{lem: Y and Xt expansion},
the coefficient of $\hbar^2$ in the expansion of (\ref{eq: def rel Xo cV})
gives the relation
\begin{gather}
\left(1-\frac{w}{z} \right)
\left({-}L(z)+\frac{Q^2}{4}+2 J^{(2)}(z) \right)\cV_0(w)
+\left(1-\frac{w}{z} \right)J^{(1)}(z) \cV_1(w) \nonumber
\\ \qquad
{}=\cV_0(w) \left({-}L(z)+\frac{Q^2}{4}+2 J^{(2)}(z)
-\frac{w}{z}\left({-}L(z)+\frac{Q^2}{4}+2 J^{(2)}(z)
-Q J^{(1)}(z) +Q^2 \right) \right)\nonumber
\\ \qquad\phantom{=}
{}+\cV_1(w)\left(J^{(1)}(z) -\frac{w}{z}\left(J^{(1)}(z)-2Q\right) \right).\label{eq: h2 coeff of X1}
\end{gather}
The coefficient of $\hbar^2$ in the $\hbar$-expansion of
the defining relation (\ref{eq: Xt cV rel}) is
\begin{gather}
\left(1-\frac{w}{z} \right)\widetilde{J}^{(2)}(z) \cV_0(z)
+\left(1-\frac{w}{z} \right)J^{(1)}(z) \cV_1(w)\nonumber
\\ \qquad
{}=\cV_0(w)
\left( \widetilde{J}^{(2)}(z)-\frac{w}{z}\big(\widetilde{J}^{(2)}(z)-2QJ^{(1)}(z)+2Q^2 \big) \right)\nonumber
\\ \qquad \phantom{=}
{}+\cV_1(w) \left( J^{(1)}(z) -\frac{w}{z}\big( J^{(1)}(z)-2Q \big) \right).\label{eq: h2 coeff of X2}
\end{gather}
By subtracting (\ref{eq: h2 coeff of X2}) from (\ref{eq: h2 coeff of X1}),
we can cancel the terms of $\cV_1(w)$ and obtain (\ref{eq: L cL and cV rel.}).
\end{proof}
\end{Lemma}

To obtain the relation of the Virasoro primary,
we have to remove the contribution of $\mathcal{L}(z)$
from (\ref{eq: L cL and cV rel.}).
For this purpose,
we give the following lemma.

\begin{Lemma}\label{lem: cL cV rel}
It follows that
\begin{gather}
\left(\mathcal{L}(z)-\frac{w}{z} \mathcal{L}(z)\right)\cV_0(w)
- \cV_0(w)\left(\mathcal{L}(z)-\frac{w}{z}
\big(\mathcal{L}(z) +Q J^{(1)}(z)-Q^2 \big)\right)\nonumber
\\ \qquad
{}=\alpha(Q-\alpha) \cV_0(w) \sum_{n \in \mathbb{Z}} (w/z)^n.
\label{eq: cL cV rel}
\end{gather}

\begin{proof}
By Proposition~\ref{prop: beta2 cV_0 rel.s},
we can show that for $k>0$,
\begin{gather*}
\bigg[\sum_{m\in \mathbb{Z}} {:}\beta^{(2)}_{k-m}\beta^{(2)}_{m}{:},\cV_0(w)\bigg]
\\ \qquad
{}=\cV_0(w)\bigg({-}2\sqrt{-1} \alpha \sum_{m\geq k} \beta^{(2)}_{m} w^{k-m}
+2\sqrt{-1}(Q-\alpha)\sum_{0<m<k} \beta^{(2)}_mw^{k-m} \bigg)
\\ \qquad \phantom{=}
{}-(Q-\alpha)^2 (k-1)w^k \cV_0(w)
+2\sqrt{-1}(Q-\alpha)\sum_{m\leq 0}\beta^{(2)}_mw^{k-m} \cV_0(w).
\end{gather*}
For $k=0$,
\begin{gather*}
\bigg[\sum_{m\in \mathbb{Z}} {:}\beta^{(2)}_{-m}\beta^{(2)}_{m}{:},
\cV_0(w)\bigg]= -2\sqrt{-1} \alpha \cV_0(w)\sum_{m> 0} \beta^{(2)}_{m} w^{-m}
+\alpha \left(\frac{u'_1+u'_2+v'_1+v'_2}{2} \right)\cV_0(w)
 \\ \qquad
{}+2\sqrt{-1} (Q-\alpha) \sum_{m> 0} \beta^{(2)}_{-m} w^{m}\cV_0(w).
\end{gather*}
For $k<0$,
\begin{gather*}
\bigg[\sum_{m\in \mathbb{Z}} {:}\beta^{(2)}_{k-m}\beta^{(2)}_{m}{:},\cV_0(w)\bigg]
=-2\sqrt{-1} \alpha \cV_0(w) \sum_{0\leq m} w^{k-m} \beta^{(2)}_m-(k+1)\alpha^2 w^k \cV_0(w)
\\ \qquad
{}+\bigg({-}2\sqrt{-1} \alpha \sum_{k\leq m<0} w^{k-m}\beta^{(2)}_m
+2\sqrt{-1} (Q-\alpha)\sum_{m<k} w^{k-m}\beta^{(2)}_m \bigg)\cV_0(w).
\end{gather*}
By using these commutation relations,
we can prove (\ref{eq: cL cV rel})
at each coefficient of $z^{-k}$.
\end{proof}
\end{Lemma}

By the above lemmas,
we can obtain the following relation between
the Virasoro algebra $L_n$ and $\cV_0(w)$.

\begin{Proposition}
For $n \in \mathbb{Z}$, we have
\begin{gather}\label{eq: Ln cV rel}
[L_n,\cV_0(w)]=w[L_{n-1},\cV_0(w)]+\alpha(Q-\alpha)w^n\cV_0(w).
\end{gather}
\begin{proof}
By Lemmas~\ref{lem: L cL and cV rel.} and~\ref{lem: cL cV rel},
we have
\begin{gather*}
\left(1-\frac{w}{z} \right)L(z)\cV_0(w)
- \left(1-\frac{w}{z} \right)\cV_0(w)L(z)
=\alpha(Q-\alpha) \cV_0(w) \sum_{n \in \mathbb{Z}} (w/z)^n.
\end{gather*}
By taking the coefficient of $z^{-n}$,
we get (\ref{eq: Ln cV rel}).
\end{proof}
\end{Proposition}

We have proved the relations among $\cV_0(w)$, the Heisenberg algebra
$\beta^{(2)}_{n}$ and the Virasoro algebra $L_n$.
Conversely, we can show that an operator satisfying
these relations is unique up to the vacuum expectation value.

\begin{Lemma}\label{lem: cV uniqueness}
Let $\widetilde{\cV}_{0}(w)\colon\cF^{(2,0)}_{\vu} \rightarrow \cF^{(2,0)}_{\vv}$
be an operator satisfying the relations obtained by replacing
$\cV_{0}(w)$ with $\widetilde{\cV}_{0}(w)$
in \eqref{eq: beta V_0 rel +}, \eqref{eq: beta V_0 rel -} and \eqref{eq: Ln cV rel}.
Then, we have
\begin{gather*}
\widetilde{\cV}_{0}(w)=\cV_{0}(w) \times \brazero \widetilde{\cV}_{0}(w) \ketzero.
\end{gather*}
\end{Lemma}

\begin{proof}
As we explained in Remark~\ref{rem: regard F as vir mod},
we can regard $\cF^{(2,0)}_{\vu}$ and $\cF^{(2,0)*}_{\vv}$
as the modules of
the Virasoro algebra $\langle L_n \rangle$ and the Heisenberg algebra $\langle \beta^{(2)}_n \rangle$.
Since $\ket{L_{\lambda, \mu}}$ and $\bra{L_{\lambda, \mu}}$ form
bases on~$\cF^{(2,0)}_{\vu}$ and $\cF^{(2,0)*}_{\vv}$,
respectively,
the linear operators from $\cF^{(2,0)}_{\vu}$
to $\cF^{(2,0)}_{\vv}$ can be characterized by the matrix elements with respect to them.
By (\ref{eq: beta V_0 rel +}), (\ref{eq: beta V_0 rel -}) and (\ref{eq: Ln cV rel}),
the matrix elements $\braket{L_{\vl}| \cV_0(w) |L_{\vm}}$
can be attributed to the matrix elements
with respect to partitions of smaller size.
Indeed,
if $\lambda^{(1)} \neq 0$, then (\ref{eq: Ln cV rel}) gives
\begin{gather*}
\braket{L_{\lo,\lt }| \cV_0(w) |L_{\mo, \mt}}=
\sum_{\nu} (\mbox{const})
\braket{L_{(\lo-1, \lo_2, \lo_3,\ldots),\lt }| \cV_0(w) | L_{\nu, \mt}}.
\end{gather*}
Here, $\nu$ runs partitions of size $|\nu|=\big|\mt\big|-\big|\lo\big|$ or $\big|\mt\big|-\big|\lo\big|+1$.
Eventually,
there remains only the vacuum expectation value $\brazero \cV_0(w) \ketzero$ in the calculation of the
matrix elements.
Therefore,
the operator is unique up to the vacuum expectation value.
\end{proof}

By the uniqueness of Lemma~\ref{lem: cV uniqueness},
we can prove Theorem~\ref{thm: the Vir rel}.

\begin{proof}[Proof of Theorem~\ref{thm: the Vir rel}]
Firstly, we prepare notations of the Verma modules and the Fock space
explained in Remark~\ref{rem: regard F as vir mod}.
Let $M_{h_1}$ (resp.~$M_{h_2}$) be the Verma module of the Virasoro algebra
with the highest weight vector $\ket{h_1}$ (resp.~$\ket{h_2}$) of highest weight $h_1=\frac{Q^2}{4}-\parap^2$ $\big($resp.~$h_2=\frac{Q^2}{4}-\parap'^2\big)$.
Let $F_u$ be the Fock space of the Heisenberg algebra $\langle \beta^{(2)}_n \rangle$
in which the zero mode acts as $\beta^{(2)}_0=u$.
Then we can show that
\begin{gather*}
\cF_{\vu}\cong M_{h_1} \otimes F_{u'}, \qquad u'=\sqrt{-1} (u'_1+u'_2)/2,
\\
\cF_{\vv}\cong M_{h_2} \otimes F_{v'}, \qquad v'=\sqrt{-1} (v'_1+v'_2)/2
\end{gather*}
as representation spaces of the algebra $\langle L_n \rangle \otimes \langle\beta^{(2)}_n\rangle$.
Further,
let $\bra{h_2}$ be the dual vector
such that $\bra{h_2}L_0=h_2 \bra{h_2}$, $\bra{h_2} L_{-n}=0$ ($n>0$),
and $\braket{h_2|h_2}=1$.

Define $V^{H}(w)\colon F_{u'}\rightarrow F_{v'}$ by
\begin{gather*}
V^{H}(w)=
\exp \bigg( 2\sqrt{-1}(\alpha-Q) \sum_{n>0}\frac{\beta^{(2)}_n}{n}w^{-n}\bigg)
\exp \bigg({-}2\sqrt{-1}\alpha \sum_{n>0}\frac{\beta^{(2)}_{-n} }{n}w^{n}\bigg).
\end{gather*}
Then we have
\begin{gather*}
\big[\beta^{(2)}_n, V^{H}(w)\big]=\sqrt{-1} w^n (Q-\alpha) V^{H}(w), \qquad n>0,
\\
\big[\beta^{(2)}_{-n}, V^{H}(w)\big]=-\sqrt{-1} w^{-n} \alpha V^{H}(w), \qquad n\geq 0.
\end{gather*}
Define $V^{\rm Vir}(w)\colon M_{h_1} \rightarrow M_{h_2}$ to be the Virasoro primary field
of conformal dimension $\alpha(Q-\alpha)$,
i.e.,
it satisfies
\begin{gather*}
\big[L_n,V^{\rm Vir}(w)\big]
=w^n\left(w \frac{\partial}{\partial w}+\alpha(Q-\alpha)(n+1)\right)V^{\rm Vir}(w)
\end{gather*}
and
\begin{gather*}
\bra{h_2} V^{\rm Vir}(w) \ket{h_1}= w^{h_2-h_1-\alpha(Q-\alpha)} \times (\mbox{scalar}).
\end{gather*}
Then it follows that
\begin{gather*}
\big[L_n,V^{\rm Vir}(w)\big]=w\big[L_{n-1},V^{\rm Vir}(w)\big]+\alpha(Q-\alpha)w^n\cV_0(w).
\end{gather*}
Therefore,
Lemma~\ref{lem: cV uniqueness} shows that
\begin{gather*}
\Phi^{\mathrm{H.V}}(z)=V^{\rm Vir}(w) \otimes V^{H}(w) \times (\mbox{scalar})^{-1}.
\end{gather*}
This completes the proof.
\end{proof}

\appendix

\section[Asymptotically free eigenfunction of Macdonald's difference operator]{Asymptotically free eigenfunction\\ of Macdonald's difference operator}\label{sec: asymp mac}

In this appendix, we briefly review basic facts of
the asymptotically free eigenfunction of Macdonald's difference operator.
Consider the following modification of Macdonald's difference operator.

\begin{Definition}
Define the operator $D_N(\boldsymbol{s};q,t)$
on $\mathbb{Q}(q,t,\vs)[[x_2/x_1,\ldots, x_N/x_{N-1}]]$ by
\begin{gather*}
D_N(\boldsymbol{s};q,t):=\sum_{k=1}^N s_k
\prod_{1\leq \ell < k} \frac{1-tx_k/x_{\ell}}{1-x_k/x_{\ell}}
\prod_{k < \ell \leq N} \frac{1-x_{\ell}/tx_k}{1-x_{\ell}/x_{k}}T_{q,x_k},
\end{gather*}
where $T_{q,x_k}$ is the difference operator defined by
\begin{gather*}
T_{q,x_k}F(x_1,\ldots,x_N)=F(x_1,\ldots,qx_k,\ldots,x_N).
\end{gather*}
\end{Definition}

An combinatorial formula for the eigenfunction of
$D_N(\boldsymbol{s};q, t)$ was given in \cite{BFS2014Macdonald,NS2012direct,Shiraishi2005conjecture}.
That is
the function $\ordmac$ given in Section~\ref{sec: nonst.R and intertwiner}.

\begin{fact}[\cite{BFS2014Macdonald,NS2012direct,Shiraishi2005conjecture}]
\label{fact: eigen fn of D}
The function $\ordmac(\boldsymbol{x};\boldsymbol{s}|q,t)$ is an unique formal solution to the eigenfunction equation
\begin{gather*}
D_N(\vs;q,t)\ordmac(\vx;\vs|q,t) =(s_1+\cdots +s_N)\ordmac(\vx;\vs|q,t)
\end{gather*}
up to scalar multiples.
\end{fact}

\begin{Remark}\label{rem: ord Mac and assym Mac}
The ordinary Macdonald polynomials \cite{Macdonald2015Symmetric} are defined as the eigenfunctions of the operator
\begin{gather*}
\mathcal{D}_x=\sum_{i=1}^{N}\prod_{\substack{1\leq j \leq N \\ j\neq i}} \frac{x_j-tx_i}{x_j-x_i}\,T_{q,x_i},
\end{gather*}
which acts on the ring of symmetric polynomials
$\mathbb{Q}(q,t)[x_1,\ldots, x_N]^{\mathfrak{S}_N}$.
The eigenfunctions of this operator, that are Macdonald polynomials,
are parametrized by the partitions $\lambda=(\lambda_1,\lambda_2,\ldots )$.
By specializing $s_i$'s as $s_i=q^{\lambda_i}t^{N-i}$,
the asymptotically free eigenfunctions $\ordmac
$
give Macdonald polynomials with partitions $\lambda$.
That is to say, if $\ell(\lambda)\leq N$ and $s_i=q^{\lambda_i}t^{N-i}$,
the infinite series $x^{\lambda}\ordmac(\vx;\vs|q,t)$ becomes
a polynomial $\big(x^{\lambda}:=\prod_{i\geq 1}x_i ^{\lambda_i}\big)$, and we have
\begin{gather*}
\mathcal{D}_x x^{\lambda} \ordmac(\vx;\vs|q,t)
=\sum_{i=1}^N q^{\lambda_i}t^{N-i} x^{\lambda} \ordmac(\vx;\vs|q,t).
\end{gather*}
\end{Remark}

The parameters $\vs=(s_1, \ldots, s_n)$ and $\vx=(x_1,\ldots, x_N)$
are symmetric to each other in the meaning of the following
bispectral duality.

\begin{fact}[\cite{NS2012direct}]\label{fact: bispec dual}
It follows that
\begin{gather*}
\prod_{1\leq i< j\leq N} \frac{(qs_j/s_i;q)_{\infty}}{(qs_j/ts_i;q)_{\infty}}
\,\ordmac(\vx;\vs|q,t)=
\prod_{1\leq i< j\leq N} \frac{(qx_j/x_i;q)_{\infty}}{(qx_j/tx_i;q)_{\infty}}
\, \ordmac(\vs;\vx|q,t).
\end{gather*}
\end{fact}

The following Poincar\'{e} duality is also important.

\begin{fact}[\cite{NS2012direct}]\label{fact: Poincare dual}
It follows that
\begin{gather*}
\ordmac(\vx;\vs|q,t) = \prod_{1\leq i < j \leq N} \frac{(t x_j/x_i ; q)_\infty}{(qx_j/tx_i;q)_\infty} \,\ordmac(\vx;\vs|q,q/t) .
\end{gather*}
\end{fact}

We defined $\ordmac(\boldsymbol{x};\boldsymbol{s}|q,t)$
as a formal power series.
However, its analyticity is well-understood,
and we can treat $\ordmac(\boldsymbol{x};\boldsymbol{s}|q,t)$
as a function of several complex variables.

\begin{notation}
Define the subset $D^N \subset \mathbb{C}^{N-1}$ to be
\begin{gather*}
D^N=\big\{ (w_1,\ldots , w_{N-1}) \in \mathbb{C}^{N-1}\,|\,
w_i\cdots w_{j-1} \notin q^{-\mathbb{Z}}\cup \{ 0\}\ (1\leq i<j\leq N) \big\}
\end{gather*}
so that
\begin{gather*}
\pi^{-1}\big(D^N\big)=\big\{ (s_1,\ldots , s_{N}) \in (\mathbb{C}^{*})^{N}\,|\,
s_j/s_i \notin q^{-\mathbb{Z}} \ (1\leq i<j\leq N) \big\}.
\end{gather*}
\end{notation}

\begin{fact}[\cite{NS2012direct}]
\label{fact: analyticity}
Let $\tau$ be a generic complex parameter.
We regard $\ordmac(\boldsymbol{x};\boldsymbol{s}|q,\tau)$
as formal power series in
\begin{gather*}
(z_1,\ldots, z_N)=(x_2/x_1,\ldots, x_N/x_{N-1}).
\end{gather*}
Set $r_0 = |q/\tau|^{\frac{n-2}{n-1}}$
if $|q/\tau|\leq 1 $, and $r_0 = |\tau/q|$ if $|q/\tau|\geq 1$.
Then for any $r< r_0$ and any compact subset $K \subset D^N$,
the series $\ordmac(\boldsymbol{x};\boldsymbol{s}|q,\tau)$
is absolutely convergent,
uniformly on $B^N_r \times K$.
Hence $\ordmac(\vx;\vs|q,\tau)$ defines a holomorphic function
on $U^N_{r_0}\times D^N$.
(For the notation $B^N_r$ and $U^N_r$,
see Notation~\ref{not: region}.)
\end{fact}

We have the correspondence between the non-stationary Ruijsenaars functions
and the ordinary Macdonald functions.

\begin{fact}[\cite{Shiraishi2019affine}]
\label{fact: p->0 lim}
Let $p^{\delta}\vx=(p^{N/N}x_1, \ldots ,p^{1/N}x_N)$
and $\kappa^{\delta}\vs=(\kappa^{N/N}x_1, \ldots ,\kappa^{1/N}x_N)$.
Then, it follows that
\begin{gather*}
\lim_{p\rightarrow 0}
\nonstrui(p^{\delta}\vx,p|\kappa^{\delta}\vs,\kappa|q,t)
=\ordmac(\vx;\vs|q,q/t).
\end{gather*}
\end{fact}

By this limit, we can check that Theorem~\ref{thm: Tr TH TH...} is certainly
consistent with Fact~\ref{fact: macdonald from mukade}.

\section[Non-stationary Ruijsenaars functions and affine screening current]{Non-stationary Ruijsenaars functions\\ and affine screening current} \label{sec: nonst. Ruijsenaars and affine sc.}

We recall the screening currents
and vertex operators depending on the parameter $\kappa$.
The~ope\-rators in the main text correspond to
the specialized case $\kappa \rightarrow t^{-1}$.

\begin{Definition}
Let $S_i(z), \phi_i(z) \in \widehat{\mathcal{H}}^N[z]$ be
operators satisfying
\begin{gather*}
S_i(z)S_i(w)=
{( w/z;q)_\infty \over (t w/z;q)_\infty }
{( qw/tz;q)_\infty \over (q w/z;q)_\infty } {:}S_i(z)S_i(w){:},
\\
S_i(z)S_{i+1}(w)={\big(\kappa^{1/N} q w/z;q\big)_\infty \over \big(\kappa^{1/N} qw/tz;q\big)_\infty }
 {:}S_i(z)S_{i+1}(w){:},
 \\
S_{i+1}(z)S_i(w)={\big(\kappa^{-1/N} t w/z;q\big)_\infty \over \big(\kappa^{-1/N} w/z;q\big)_\infty }
 {:}S_{i+1}(z)S_i(w){:},
\\
\phi_i(z)\phi_i(w)={(w/z;q,\kappa)_\infty \over (tw/z;q,\kappa)_\infty }
{(\kappa q w/z;q,\kappa)_\infty \over (\kappa qw/tz;q,\kappa)_\infty }{:}\phi_i(z)\phi_i(w){:}\qquad
(0\leq i\leq N-1),
\\
\phi_i(z)\phi_j(w)=
{\big(\kappa^{(-i+j)/N} w/z;q,\kappa\big)_\infty \over \big(\kappa^{(-i+j)/N} t w/z;q,\kappa\big)_\infty }
{\big(\kappa^{(-i+j)/N} q w/z;q,\kappa\big)_\infty \over \big(\kappa^{(-i+j)/N} qw/tz;q,\kappa\big)_\infty }{:}\phi_i(z)\phi_j(w){:}
\\ \hphantom{\phi_i(z)\phi_j(w)=}
(0\leq i<j\leq N-1),
\\
\phi_i(z)\phi_j(w)=
{\big(\kappa^{(-i+j+N)/N} w/z;q,\kappa\big)_\infty \over \big(\kappa^{(-i+j+N)/N} t w/z;q,\kappa\big)_\infty }
{\big(\kappa^{(-i+j+N)/N} q w/z;q,\kappa\big)_\infty \over \big(\kappa^{(-i+j+N)/N} qw/tz;q,\kappa\big)_\infty }{:}\phi_i(z)\phi_j(w){:}
\\ \hphantom{\phi_i(z)\phi_j(w)=}
(0\leq j<i\leq N-1),
\end{gather*}
and
\begin{gather*}
\phi_i(z) S_{i+1}(w)={\big(\kappa^{1/N}qw/z;q\big)_\infty \over \big(\kappa^{1/N}qw/tz;q\big)_\infty} {:} \phi_i(z) S_{i+1}(w){:},
\\
S_{i+1}(w)\phi_i(z)={\big(\kappa^{-1/N}tz/w;q\big)_\infty \over \big(\kappa^{-1/N}z/w;q\big)_\infty} {:} \phi_i(z) S_{i+1}(w){:},
\\
\phi_i(z) S_i(w)={(w/z;q)_\infty \over ( tw/z;q)_\infty} {:} \phi_i(z) S_i(w){:},\qquad
 S_i(w)\phi_i(z)={(qz/tw;q)_\infty \over (qz/w;q)_\infty} {:} \phi_i(z) S_i(w){:},
 \\
\phi_i(z) S_j(w)= {:} \phi_i(z) S_j(w){:}, \qquad S_j(w)\phi_i(z)= {:} \phi_i(z) S_j(w){:}\qquad (j\neq i,i+1).
 \end{gather*}
\end{Definition}

The screened vertex operator
having the parameter $\kappa$
is defined as follows.

\begin{Definition}
For $1\leq i\leq N-1$ and $\lambda\in \parset$, define
\begin{gather*}
\phi^{i}_{\lambda}(z)=\phi_{i-\ell}(\kappa^{(\ell+1)/N} z)
\prod^{\curvearrowleft}_{1\leq j\leq \ell}S_{i-j+1}\big(\kappa^{j/ N} q^{\lambda_{j}}z\big),
\\
\Phi^i_{\lambda}(z)=
\left( {(q/t;q)_\infty \over (q;q)_\infty}\right)^{\ell(\lambda)}
\phi^i_{\lambda}(z),
\\
\Phi^i(z|x,p)=\sum_{\lambda\in \parset}\Phi^i_{\lambda}(z)
\prod_{k\geq 1} \big(p^{1/N} x_{N-i+k}/x_{N-i+k-1}\big)^{\lambda_k}.
\end{gather*}
\end{Definition}

\begin{notation}
For $1\leq i\leq N$, set
\begin{gather*}
\omega^i=(\omega^i_1,\ldots,\omega^i_N)=(\overbrace{0,\ldots,0}^{i \ {\rm times}},1,\ldots,1)\in \mathbb{Z}^N.
\end{gather*}
We write
\begin{gather*}
t^{\omega^i}  x=\big(t^{\omega^i_1} x_1,\ldots,t^{\omega^i_N} x_N \big).
\end{gather*}
\end{notation}

The non-stationary Ruijsenaars functions
can be constructed by the screened vertex operators.

\begin{fact}[\cite{Shiraishi2019affine}]\label{fact: <Phi...>=f}
Let $N\in \mathbb{Z}_{\geq 2}$.
Then, it follows that
\begin{gather*}
\langle 0| \Phi^0\big(1/s_N|t^{\omega^N} x,p\big) \Phi^1\big(1/s_{N-1}|t^{\omega^{N-1}} x,p\big) \cdots \Phi^{N-1}\big(1/s_1|t^{\omega^1} x,p\big) |0\rangle
 \\ \qquad
{}=\sum_{\lambda^{(1)},\ldots,\lambda^{(N)}\in \parset}
\langle 0| \Phi^0_{\lambda^{(N)}}(1/s_N) \Phi^1_{\lambda^{(N-1)}}(1/s_{N-1}) \cdots
 \Phi^{N-1}_{\lambda^{(1)}}(1/s_1) |0\rangle
 \\ \qquad \qquad
 {}\times \mathfrak{t}(\vl)
\prod_{j=1}^N \prod_{i\geq 1} (p^{1/N} x_{j+i}/x_{j+i-1})^{\lambda^{(j)}_i}
\\ \qquad
{} =\prod_{1\leq i< j\leq N}
 {(\kappa^{(j-i)/N} s_j/s_i;q,\kappa)_\infty \over (\kappa^{(j-i)/N} t s_j/s_i;q,\kappa)_\infty }
{(\kappa^{(j-i)/N} q s_j/s_i;q,\kappa)_\infty \over (\kappa^{(j-i)/N} qs_j/ts_i;q,\kappa)_\infty }
f^{\widehat{\mathfrak gl}_N} \big(x,p^{1/N}|s,\kappa^{1/N}|q,t\big).
\end{gather*}
Here, we put
\begin{gather*}
\mathfrak{t}(\vl):=
\prod_{i=1}^N t^{-|\lambda^{(i)}|+|\lambda^{(i)}|^{(0)}}
 \prod_{0\leq i< j \leq N-1} t^{|\lambda^{(N-i)}|^{(N+i-j)}+|\lambda^{(N-j)}|^{(-i+j-1)}}
 \\ \hphantom{\mathfrak{t}(\vl):}
=\prod_{i=1}^N t^{-|\lambda^{(i)}|^{(i-1)}+|\lambda^{(i)}|^{(0)}}.
\end{gather*}
Note that there is a typo in \cite{Shiraishi2019affine}, and the corresponding formula Theorem 2.13 in \cite{Shiraishi2019affine} should be read as in Fact~\ref{fact: <Phi...>=f}.
\end{fact}

\section{Some combinatorial formulas}\label{sec: proofs}

\subsection{Proof of Lemma~\ref{lem: rational num eq}}\label{sec: pf of Nek fac lemm}

By using factorials, we can rewrite the Nekrasov factors as follows. If $k\neq i$,
\begin{gather*}
N_{\mu^{(k-1)},\mu^{(k)}}(Q)
\\ \qquad
{}=\prod_{\substack{h\equiv i-k\\ j-h\equiv 0 \\ j\geq h\geq1}}\!\!\big(Q q^{-\lambda_{h} +\lambda_{j+N+1}}t^{\frac{h-j}{N}};q\big)_{\lambda_{j+1}-\lambda_{j+N+1}}
\prod_{\substack{h\equiv i-k+1\\ j-h\equiv 0 \\ j\geq h\geq1}}\!\!
\big(Q q^{\lambda_{h} -\lambda_{j-1}}t^{\frac{j-h}{N}+1} ;q\big)_{\lambda_{j-1}-\lambda_{j+N-1}}.
\end{gather*}
Decomposing factors which will be canceled afterward,
we have
\begin{gather*}
N_{\mu^{(k-1)},\mu^{(k)}}(Q)
=
A^{(k-1,k)}(Q) \cdot \widetilde{A}^{(k-1,k)}(Q)
\qquad (k\neq i),
\end{gather*}
where we put
\begin{gather*}
A^{(k-1,k)}(Q):=
\prod_{\substack{h\equiv i-k\\ j-h\equiv 0 \\ j\geq h\geq1}}
\big(Q q^{-\lambda_{h} +\lambda_{j+N}}t^{\frac{h-j}{N}} ;q\big)_{\lambda_{j+1}-\lambda_{j+N}} \prod_{\substack{h\equiv i-k+1\\ j-h\equiv 0 \\ j\geq h\geq1}}
\big(Q q^{\lambda_{h} -\lambda_{j}}t^{\frac{j-h}{N}+1} ;q\big)_{\lambda_{j}-\lambda_{j+N-1}},
\\
\widetilde{A}^{(k-1,k)}(Q):=\!
\prod_{\substack{h\equiv i-k\\ j-h\equiv 0 \\ j\geq h\geq1}}\!\!
\big(Q q^{-\lambda_{h} +\lambda_{j+N+1}}t^{\frac{h-j}{N}} ;q\big)_{\lambda_{j+N}-\lambda_{j+N+1}} \!\!\! \prod_{\substack{h\equiv i-k+1\\ j-h\equiv 0 \\ j\geq h\geq1}}\!\!\!\!
\big(Q q^{\lambda_{h} -\lambda_{j-1}}t^{\frac{j-h}{N}+1} ;q\big)_{\lambda_{j-1}-\lambda_{j}}.
\end{gather*}
Similarly, we have
\begin{gather*}
N_{\mu^{(i-1)},\mu^{(i)}}(Q)
=\prod_{\substack{h\equiv N\\ j-h\equiv 0 \\ j\geq h\geq1}}
\big(Q q^{-\lambda_{h} +\lambda_{j+1}}t^{\frac{h-j}{N}} ;q\big)_{\lambda_{j-N+1}-\lambda_{j+1}}
\!\!\prod_{\substack{\alpha \equiv 1\\ \beta-\alpha \equiv -1 \\ \beta \geq \alpha \geq 1}}\!\!
\big(Q q^{\lambda_{\alpha} -\lambda_{\beta}}t^{\frac{\beta-\alpha+1}{N}} ;q\big)_{\lambda_{\beta}-\lambda_{\beta+N}}\
\\ \hphantom{N_{\mu^{(i-1)},\mu^{(i)}}(Q)}
{}=A^{(i-1,i)}(Q) \cdot \widetilde{A}^{(i-1,i)}(Q),
\end{gather*}
where we put
\begin{gather*}
A^{(i-1,i)}(Q):=
\prod_{\substack{h\equiv N\\ j-h\equiv 0 \\ j\geq h\geq1}}
\big(Q q^{-\lambda_{h} +\lambda_{j}}t^{\frac{h-j}{N}} ;q\big)_{\lambda_{j-N+1}-\lambda_{j}}
\prod_{\substack{\alpha \equiv 1\\ \beta-\alpha \equiv -1 \\ \beta \geq \alpha \geq 1}}
\big(Q q^{\lambda_{\alpha} -\lambda_{\beta+1}}t^{\frac{\beta-\alpha+1}{N}} ;q\big)_{\lambda_{\beta+1}-\lambda_{\beta+N}}, \\
\widetilde{A}^{(i-1,i)}(Q):=
\prod_{\substack{h\equiv N\\ j-h\equiv 0 \\ j\geq h\geq1}}
\big(Q q^{-\lambda_{h} +\lambda_{j+1}}t^{\frac{h-j}{N}} ;q\big)_{\lambda_{j}-\lambda_{j+1}}
\prod_{\substack{\alpha \equiv 1\\ \beta-\alpha \equiv -1 \\ \beta \geq \alpha \geq 1}}
\big(Q q^{\lambda_{\alpha} -\lambda_{\beta}}t^{\frac{\beta-\alpha+1}{N}} ;q\big)_{\lambda_{\beta}-\lambda_{\beta+1}}.
\end{gather*}
For $k=1,\ldots, N$,
\begin{gather*}
N_{\mu^{(k)},\mu^{(k)}}(Q)
=\prod_{\substack{h\equiv i-k\\ j-h\equiv 0 \\ j\geq h\geq1}}
\big(Q q^{-\lambda_{h} +\lambda_{j+N}}t^{\frac{h-j}{N}} ;q\big)_{\lambda_{j}-\lambda_{j+N}}
\prod_{\substack{h\equiv i-k\\ j-h\equiv 0 \\ j\geq h\geq1}}
\big(Q q^{\lambda_{h} -\lambda_{j}}t^{\frac{j-h}{N}+1} ;q\big)_{\lambda_{j}-\lambda_{j+N}}
\\ \hphantom{N_{\mu^{(k)},\mu^{(k)}}(Q)}
{}=B^{(k)}(Q)\cdot \widetilde{B}^{(k)}(Q),
\end{gather*}
where we put
\begin{gather*}
B^{(k)}(Q)
:=\prod_{\substack{h\equiv i-k\\ j-h\equiv 0 \\ j\geq h\geq1}}
\big(Q q^{-\lambda_{h} +\lambda_{j+N}}t^{\frac{h-j}{N}} ;q\big)_{\lambda_{j+1}-\lambda_{j+N}}
\prod_{\substack{h\equiv i-k\\ j-h\equiv 0 \\ j\geq h\geq1}}
\big(Q q^{\lambda_{h} -\lambda_{j}}t^{\frac{j-h}{N}+1} ;q\big)_{\lambda_{j}-\lambda_{j+N-1}}, \\
\widetilde{B}^{(k)}(Q)
:=
\prod_{\substack{h\equiv i-k\\ j-h\equiv 0 \\ j\geq h\geq1}}
\big(Q q^{-\lambda_{h} +\lambda_{j+1}}t^{\frac{h-j}{N}} ;q\big)_{\lambda_{j}-\lambda_{j+1}}
\prod_{\substack{h\equiv i-k\\ j-h\equiv 0 \\ j\geq h\geq1}}
\big(Q q^{\lambda_{h} -\lambda_{j+N-1}}t^{\frac{j-h}{N}+1} ;q\big)_{\lambda_{j+N-1}-\lambda_{j+N}}.
\end{gather*}
Then, it is clear that
\begin{gather*}
\prod_{1\leq k \leq N}\widetilde{B}^{(k)}(1)=\sfN^{(0)}_{\lambda \lambda}\big(1|q,t^{-1/N}\big).
\end{gather*}
Since it follows that
\begin{gather*}
A^{(i-1,i)}(Q)=
\prod_{\substack{h\equiv N\\ j-h\equiv 0 \\ j\geq h\geq1}}
\big(Q q^{-\lambda_{h} +\lambda_{j+N}}t^{\frac{h-j}{N}-1} ;q\big)_{\lambda_{j+1}-\lambda_{j+N}}
\prod_{\substack{h\equiv N\\ h\geq 1}}
(Q;q)_{\lambda_{h-N+1}-\lambda_{h}}
\\ \hphantom{A^{(i-1,i)}(Q)=}
{}\times
\prod_{\substack{\alpha \equiv 1\\ \beta-\alpha \equiv -1 \\ \beta \geq \alpha \geq 1}}
\big(Q q^{\lambda_{\alpha} -\lambda_{\beta+1}}t^{\frac{\beta-\alpha+1}{N}} ;q\big)_{\lambda_{\beta+1}-\lambda_{\beta+N}}
\\ \hphantom{A^{(i-1,i)}(Q)}
{}=\prod_{\substack{h\equiv N\\ j-h\equiv 0 \\ j\geq h\geq1}}
\big(Q q^{-\lambda_{h} +\lambda_{j+N}}t^{\frac{h-j}{N}-1} ;q\big)_{\lambda_{j+1}-\lambda_{j+N}}
\prod_{\substack{\alpha \equiv 1\\ \beta-\alpha \equiv 0 \\ \beta \geq \alpha \geq 1}}
\big(Q q^{\lambda_{\alpha} -\lambda_{\beta}}t^{\frac{\beta-\alpha}{N}} ;q\big)_{\lambda_{\beta}-\lambda_{\beta+N-1}},
\end{gather*}
we have
\begin{gather*}
\prod_{1\leq k \leq N}
\frac{A^{(k-1,k)}(t^{\delta_{k,i}})}{B^{(k)}(1)}
=1.
\end{gather*}
By the above computation,
it follows that
\begin{gather*}
\prod_{1\leq k \leq N}
\frac{ N_{\mu^{(k-1)}, \mu^{(k)}}(t^{\delta_{k,i}}) }
{N_{\mu^{(k)}, \mu^{(k)}}(1)}=
\frac{1}{\sfN^{(0)}_{\lambda \lambda}(1|q,t^{-1/N}) }
\, \widetilde{A}^{(i-1,i)}(t)\prod_{k\neq i} \widetilde{A}^{(k-1,k)}(1).
\end{gather*}
Furthermore, it can be shown that for $k \neq i$,
\begin{gather*}
\widetilde{A}^{(k-1,k)}(1)=
\prod_{\substack{h\equiv i-k\\ j-h\equiv 0 \\ j\geq h\geq1}}
\big(q^{-\lambda_{h} +\lambda_{j+N+1}}t^{\frac{h-j}{N}} ;q\big)_{\lambda_{j+N}-\lambda_{j+N+1}}
\\ \phantom{\widetilde{A}^{(k-1,k)}(1)=}
{}\times
\prod_{\substack{h\equiv i-k+1\\ j-h\equiv 0 \\ j\geq h\geq1}}
\big(q^{\lambda_{h} -\lambda_{j+N-1}}t^{\frac{j-h}{N}+2} ;q\big)_{\lambda_{j+N-1}-\lambda_{j+N}}
\prod_{\substack{h\equiv i-k+1\\ h\geq1}}
\big(q^{\lambda_{h} -\lambda_{h-1}}t ;q\big)_{\lambda_{h-1}-\lambda_{h}}
\\ \phantom{\widetilde{A}^{(k-1,k)}(1)}
{}=\prod_{\substack{h\equiv i-k\\ j-h\equiv 0 \\ j\geq h\geq1}}
\big(q^{-\lambda_{h} +\lambda_{j+1}}t^{\frac{h-j}{N}+1} ;q\big)_{\lambda_{j}-\lambda_{j+1}}
\prod_{\substack{\alpha \equiv i-k+1\\ \beta-\alpha\equiv -1 \\ \beta \geq \alpha \geq1}}
\big(q^{\lambda_{\alpha} -\lambda_{\beta}}t^{\frac{\beta-\alpha+1}{N}+1} ;q\big)_{\lambda_{\beta}-\lambda_{\beta+1}}.
\end{gather*}
Therefore,
we obtain
\begin{gather*}
\prod_{1\leq k \leq N}
\frac{ N_{\mu^{(k-1)}, \mu^{(k)}}(t^{\delta_{k,i}}) }
{N_{\mu^{(k)}, \mu^{(k)}}(1)}
=
\frac{\sfN^{(0)}_{\lambda \lambda}(t|q,t^{-1/N})}
{\sfN^{(0)}_{\lambda \lambda}(1|q,t^{-1/N}) }.
\end{gather*}

\subsection{Proof of Theorem~\ref{fact: chang pref. direc.}}
\label{sec: pf of chang pref. direc.}

By Fact~\ref{fact: mat el. of TV},
we have
\begin{gather*}
\bra{P_{\vl}(\vv)} \mathcal{T}^V(\vu, \vv, w) \ket{Q_{\vm}(\vu)}
\\ \qquad
{}= \zeta^{\sharp} \times
\frac{\prod_{ i, j=1 }^N N_{\lambda^{(i)}, \mu^{(j)}}(v_i/\gamma u_j)}
{\prod_{k=1}^N c_{\lambda^{(k)}}c'_{\mu^{(k)}}
\prod_{1\leq i<j \leq N} N_{\mu^{(i)}, \mu^{(j)}}(qu_i/tu_j)
\prod_{1\leq i<j \leq N} N_{\lambda^{(j)}, \lambda^{(i)}}(qv_j/tv_i)},
\end{gather*}
where
\begin{gather*}
\zeta^{\sharp}:=\mathcal{M}(\vu,\vv; \vl, \vm;w)
\cdot
\xi^{(+)}_{\vm}(\vu)^{-1}\cdot
\xi^{(-)}_{\vl}(\vv)^{-1}.
\end{gather*}
On the other hand,
by using the relation
\begin{gather*}
N_{\lambda,\mu}(\gamma^{-1}x)=N_{\mu, \lambda}(\gamma^{-1}x^{-1})x^{|\lambda|+|\mu|}\frac{f_{\lambda}}{f_{\mu}},
\end{gather*}
we have
\begin{gather*}
\bra{\vm}\cTH(\vu, \vv;w)\ket{\vl}
\\ \qquad
{}=\zeta^{\flat} \times
\frac{\prod_{ i, j=1 }^N N_{\lambda^{(i)}, \mu^{(j)}}(v_i/\gamma u_j)}
{\prod_{k=1}^N c_{\lambda^{(k)}}c'_{\mu^{(k)}}
\prod_{1\leq i<j \leq N} N_{\mu^{(i)}, \mu^{(j)}}(qu_i/tu_j)
\prod_{1\leq i<j \leq N} N_{\lambda^{(j)}, \lambda^{(i)}}(qv_j/tv_i)},
\end{gather*}
where
\begin{gather*}
\zeta^{\flat}:=\prod_{i=1}^N
\hat{t}\big(\lambda^{(i)},\tfrac{u_1\cdots u_i}{v_1\cdots v_i}w ,v_i,0\big)c_{\lambda^{(i)}}
\hat{t}^*\big(\mu^{(i)},u_i,\tfrac{u_1\cdots u_{i-1}}{v_1\cdots v_{i-1}}w ,0\big)c'_{\mu^{(i)}}
\\ \hphantom{\zeta^{\flat}:=}
{}\times \prod_{1\leq i <j\leq N}
(\gamma u_j/u_i)^{-|\mu^{(i)}|-|\mu^{(j)}|}\frac{f_{\mu^{(i)}}}{f_{\mu^{(j)}}}
\prod_{1\leq i <j\leq N}
(u_j/u_i)^{|\lambda^{(i)}|+|\mu^{(j)}|}\frac{f_{\mu^{(j)}}}{f_{\lambda^{(i)}}}.
\end{gather*}
A direct calculation gives
\begin{gather*}
\zeta^{\sharp}=w^{|\vl|-|\vm|} (-1)^{N|\vl|+(N+1)|\vm|}
\gamma^{(N+1)|\vl|+(-2N+1)|\vm|}e_N(\vu)^{|\vl|}
\\ \phantom{\zeta^{\sharp}=}
{}\times \prod_{i=1}^N
q^{n(\mu^{(i)'})+n(\lambda^{(i)'})+|\mu^{(i)}|}
g_{\lambda^{(i)}}^{-N+i-1} g_{\mu^{(i)}}^{N-i}
(-1)^{i|\lambda^{(i)}|+i|\mu^{(i)}|}
\gamma^{-i|\lambda^{(i)}|+i|\mu^{(i)}|}
\\ \phantom{\zeta^{\sharp}=}
{}\times \prod_{i=1}^Nu_i^{(N-i+1)|\mu^{(i)}|-\sum_{k=1}^i |\mu^{(k)}|}
\times \prod_{i=1}^N
v_i^{(i-N)|\lambda^{(i)}|-\sum_{k=i}^N|\lambda^{(k)}|}
=(-1)^{|\vl|+|\vm|} \zeta^{\flat}.
\end{gather*}
This completes the proof of Theorem~\ref{fact: chang pref. direc.}.

\subsection*{Acknowledgments}
The authors would like to thank H.~Awata, B.~Feigin, A.~Hoshino,
H.~Kanno, Y.~Matsuo, M.~Noumi and S.~Yanagida for valuable comments.
The authors are also grateful to the referees for helpful feedback.
The research of J.S.~is supported by JSPS KAKENHI (Grant Numbers 19K03512).
Y.O.\ and M.F.\ are partially supported by Grant-in-Aid for JSPS Research Fellow
(Y.O.:~18J00754, M.F.:~17J02745).

\pdfbookmark[1]{References}{ref}
\LastPageEnding

\end{document}